\documentclass[12pt,oneside,reqno]{amsart}
\usepackage{mathrsfs}
\usepackage{amssymb,amsmath,amsthm,color}
\usepackage{graphicx, cite}
\usepackage{hyperref}
\usepackage{url}
\usepackage{setspace}
\usepackage{enumerate}
\usepackage[margin=1in]{geometry}
\usepackage{mathtools}
\usepackage{verbatim}

\usepackage{footnote}
\usepackage{cite}
\usepackage{amscd}
\usepackage{color}
\usepackage{euscript}
\usepackage{calc}                   

\usepackage{tikz}
\usepackage{graphicx} 

\newtheorem{theorem}{Theorem}[section]
\newtheorem{lemma}[theorem]{Lemma}
\newtheorem{proof of lemma}[theorem]{Proof of Lemma}
\newtheorem{proposition}[theorem]{Proposition}

\newtheorem{corollary}[theorem]{Corollary}

\theoremstyle{definition}
\newtheorem{definition}[theorem]{Definition}

\newtheorem{remark}[theorem]{Remark}

\numberwithin{equation}{section}

\newcommand{\abs}[1]{\lvert#1\rvert}

\newcommand{\T}{\mathbb{T}}
\newcommand{\R}{\mathbb{R}}
\newcommand{\Z}{\mathbb{Z}}
\newcommand{\cH}{\mathcal{H}}

\newcommand{\Sp}{\mathfrak{S}}
\newcounter{cases}
\newcounter{subcases}

\begin{document}
	\allowdisplaybreaks
	\title[Orthonormal Strichartz estimates in partial regularity frame]{Refined Strichartz estimates and their orthornomal counterparts for Schr\"odinger equations on torus}
	
	\author{Divyang G. Bhimani, Subhash.  R.  Choudhary and S. S. Mondal}
	
	\address{Department of Mathematics, Indian Institute of Science Education and Research-Pune, Homi Bhabha Road, Pune 411008, India}
	
	\email{divyang.bhimani@iiserpune.ac.in}
	\email{subhashranjan.choudhary@students.iiserpune.ac.in}
	\address{\endgraf  Stat-Math Unit, Indian Statistical Institute  Kolkata, BT Road,  Baranagar, Kolkata  700108, India }
	\email{mondalshyam055@gmail.com}
	
	
	\date{}
	
	\begin{abstract} The aim of the paper is twofold. We establish refined Strichartz estimates for the Schr\"odinger equation on tori within the framework of partial regularity. As a result, we reveal that the solution of the free Schr\"odinger equation has better regularity in mixed Lebesgue spaces. This complements the well-established theory over the past few decades, where initial data comes from the Sobolev space with respect to all spatial variables. As an application, we obtain local well-posedness for non-gauge-invariant nonlinearities with partially regular initial data. On the other hand, we extend refined Strichartz estimates for infinite systems of orthonormal functions, which generalizes the classical orthonormal Strichartz estimates on the torus by Nakamura \cite{nakamura2020orthonormal}. As an application, we establish well-posedness for the Hartree equation for infinitely many fermions in some Schatten spaces. In the process, we develop several harmonic analysis tools for mixed Lebesgue spaces, e.g. Fourier multiplier transference principle, vector-valued Bernstein inequality, and vector-valued Littlewood--Paley theory for densities of operators, which may be of independent interest and complement the results of \cite{ward2010mixedLebesgue,sabin2016littlewood}.
	\end{abstract}
	\maketitle
	\tableofcontents
	\section{Introduction}
	In this paper, we develop  Strichartz estimates, along with their orthonormal counterparts, within the framework of partial regularity for the nonlinear   Schr\"odinger equation (NLS) on the torus $\mathbb T^d$. This reveals that the initial data do not need to have complete regularity like $H^s(\mathbb T^d),$ but only a partial regularity space $X_k^s$ and $H^s \subsetneq X_k^s$ (to be described below, see \eqref{evolution space embedding}). To explain this and to motivate our main results, we begin with the following nonlinear Schr\"odinger equation (NLS) on the torus:
	\begin{equation}\label{NLS}
		\begin{cases}
			i\partial_tu -\Delta u = w\ast F_{p}(u)\\
			u(0, x)=f(x),
		\end{cases} (t,x) \in \mathbb R \times \mathbb T^d.
	\end{equation}
	Here, the term  $F_p$  satisfies the growth condition with exponent $p > 1$:
	\begin{equation}\label{growth condition}
		\left|F_p(u)\right| \lesssim|u|^p \quad \text { and } \quad|u|\left|F_p^{\prime}(u)\right| \sim\left|F_p(u)\right| ,
	\end{equation}
	and the function $w$ is taken from the mixed Sobolev-type spaces $ W^{s,r}(\T^d).$   
	Consider NLS on the Euclidean space in the following form:
	\begin{equation}\label{Eucleidean}
		\begin{cases}
			i\partial_tu -\Delta u = F_{p}(u)\\
			u(0, x)=f(x),
		\end{cases} (t,x) \in \mathbb R \times \mathbb R^d.
	\end{equation}
	The typical examples of  \eqref{growth condition} are $F_p(u)= \pm |u|^{p-1}u$ and $F_p(u)=\pm |u|^p;$ with this nonlinearity, \eqref{Eucleidean} is invariant under scaling: if $u(t,x)$ is a solution of \eqref{Eucleidean}, then so is	
	\begin{equation*}	
		u_{\delta}(t, x)= \delta^{\frac{2}{p-1}} u(\delta^{2} t, \delta x), \quad \delta>0,
	\end{equation*}	 
	with  rescaled initial data $f_{\delta}(x)=u_{\delta}(0, x)$. The following scaling relation
	\begin{equation}\label{scaling}
		\| f_{\delta}\|_{\dot{H}^s(\mathbb{R}^{d})} =\delta^{\frac{2}{p-1}+s-\frac{d}{2}} \|f\|_{ \dot{H}^s(\mathbb{R}^{d})}
	\end{equation}
	determines  the critical exponent  $p$ (or scale-invariant Sobolev space $\dot{H}^{s_c}$), which can be  expressed as
	$$
	p=p_c(d, s)=1+\frac{4}{d-2s}, \quad \text{alternatively}\,\, s=s_c= \frac{d}{2}-\frac{2}{p-1}.
	$$
	The critical exponent serves to separate various regimes of solution behavior: for $p > p_c(d, s)$, the problem is said to be supercritical; and for $p < p_c(d, s)$, it is subcritical. In fact,  it is expected that the problem is well-posed in the subcritical regime and ill-posed in the supercritical regime. Indeed, enormous progress has been made over the last three decades to study the Cauchy problem \eqref{Eucleidean} with initial data in the full regularity space $H^s(\mathbb R^d)$. See for e.g.  \cite{birnir1996ill, cazenave2003, dai2013continuous, kato1995nonlinear, cazenave1990cauchy, killip2013nonlinear} and references therein.

	On the other hand, if we decompose the spatial variable $z \in \mathbb{R}^d$ as $(x, y) \in \mathbb{R}^{d-k} \times \mathbb{R}^k$, and  suppose that the initial data are from the mixed Sobolev space $L^2(\mathbb{R}^{d-k}; {\dot{H}^s(\mathbb{R}^k))}$ instead of  the full regularity  space {$\dot{H}^s(\mathbb{R}^d)$}, then  for a rescaled initial datum $f_\delta,$ we have
	$$\| f_{\delta}\|_{L_x^2 \dot{H}_y^s} =  \delta^{\frac{2}{p-1} + s - \frac{d}{2}} \|f\|_{L_x^2 \dot{H}_y^s}.$$
	This shows that the scaling exponent $ \frac{2}{p-1} + s - \frac{d}{2}$ is identical to that in the case of the full regularity setting \eqref{scaling}, and also independent of the choice of $k$. Therefore, the same admissible range of $p$ holds even when regularity is imposed only partially i.e., the initial data need not possess full regularity; having partial regularity alone is sufficient. In view of this perspective,   Koh, Lee, and Seo in \cite{koh2023local} developed a new well-posedness theory under weaker regularity assumptions i.e.,   when the initial data do not possess full regularity in the classical Sobolev space $H^s(\mathbb{R}^d)$, but instead lies in a partially regular space of the form $L^2(\mathbb{R}^{d-k}; H^{s} (\mathbb{R}^k))$ for some $1 \leq k \leq d$. To this end, the authors \cite{koh2023local} established a refined Strichartz estimate with a different regularity for each spatial variable.

	Taking this  consideration into account, we focus on studying Strichartz estimates and their orthonormal counterparts within the framework of partial regularity on the torus $\T^d$. To the best of our knowledge, the partial regularity framework on torus $\T^d$ has not yet been explored in the existing literature. 

	\subsection{Strichartz estimates}
	For the linear Schr\"odinger equation of \eqref{Eucleidean} on $\mathbb{R}^d$,  the following   mixed space-time estimates,
	well known as classical Strichartz estimates  (see \cite{KeelTao, ORS})
	\begin{eqnarray}\label{ost}
		\|e^{-it\Delta} f\|_{L^q_t (\mathbb R, L^r_x (\mathbb R^d)))} \lesssim \|f\|_{L_{x}^2(\mathbb R^d)}
	\end{eqnarray}
	holds if and only if $(q,r)$ satisfies \textbf{Schr\"odinger admissible} condition $$\frac{2}{q}+ \frac{d}{r}= \frac{d}{2},  \quad 2 \leq r, q \leq \infty, \quad (q, r, d) \neq (2, \infty,2).$$
	The above remarkable estimate was first obtained by Strichartz \cite{ORS} when $r=q$ in connection with Fourier restriction theory.  However, in the non-diagonal setting, the proof of the estimate  \eqref{ost} is a combination of an abstract functional analysis argument known as the $TT^*$ duality argument and the following dispersive estimate
	\begin{equation}\label{cec}
		\|e^{-it\Delta} f\|_{L^{\infty}_x(\mathbb R^d)} \lesssim |t|^{-\frac{d}{2}} \|f\|_{L^1_x (\mathbb R^d)},
	\end{equation}
	except for the endpoint case $(q,r)=(1,\frac{d}{d-2})$ for $d\geq3$, which was proved by Keel-Tao \cite{KeelTao}. The situation changes significantly on the torus $\T^d$, where the dispersive estimate \eqref{cec} no longer holds, which makes it challenging to establish Strichartz estimates of the form \eqref{ost} for the linear Schrödinger equation \eqref{NLS}.   One crucial feature
	of the equation on $\T^d$
	is that the dispersion of the solution is weaker than the
	solution of the equation on $\R^d$
	since $\T^d$
	is compact and hence, new difficulty occurs
	to established the well-posedness theory.

	To establish the well-posedness theory for the NLS on torus, Bourgain in his seminal paper \cite{Bourgainrestriction} initiated to study the estimates of the form 
	\begin{eqnarray}\label{ft}
		\|e^{-it\Delta} f\|_{L_{t,x}^{q}(I \times \mathbb T^d)} \lesssim_{|I|}  N^{ d/2- (d+2)/q}\|f\|_{L_{x}^2(\mathbb T^d)}, \quad f\in L_{x}^2(\mathbb T^d),
	\end{eqnarray}
	for $q=\frac{d+2}{2}$ with  $\text{supp} \hat{f} \subset [-N, N]^d,$ where $I$ is a compact interval, and proved it for $d=1$ and $2$, using number theoretical argument.  Subsequently, Bourgain and Demeter \cite{bourgain2015} extend these results for dimensions $d\geq 3$ by utilizing the decoupling theory.
	We also note that  Burq, G\'erard and Tzvetkov  \cite{burq2004strichartzcompact}  established Strichartz estimate with loss of $1/q$ derivatives for the Sch\"odinger equation on compact manifold
	\begin{eqnarray}\label{BurqEst}
		\|e^{-it\Delta} f\|_{L^q(I, L^r(\mathbb T^d))} \leq C \|f\|_{H^{1/q}(\mathbb T^d)},
	\end{eqnarray}
	where $(q,r)$ is a Schr\"odinger admissible condition.
	While for the fractional Laplacian on torus, Demirbas, Erdo\u{g}an and  Tzirakis \cite{Demirbas2016ALM} established 
	\[\|e^{it (-\Delta)^{\theta}} f\|_{L^4 (\mathbb T \times \mathbb T)} \leq C \|f\|_{H^{\rho}}, \quad \rho > \frac{2-\theta}{8},  \ \theta \in (1,2).\]
	
	Now introduce the  frequency  cut-off  operator $P_{\leq N},$    as  $P_{\leq N} \phi=(\textbf{1}_{[-N,N]^d} \widehat{\phi})^{\lor},$ where $\phi$ is Schwartz function on $\mathbb{T}^d$.  
	Recently,  estimate \eqref{ft} has been extended in the  mixed norm setting by Dinh  \cite{Dinh} for the fractional  Laplacian on torus $\mathbb{T}^d$  and for all $N>1$, they proved the following Strichartz estimates:   
	\begin{equation}\label{SE Torus}
		\|e^{it(-\Delta)^{\frac{\theta}{2}}}P_{\leq N}f\|_{L^{q}_{t}(I, L_{x}^{r} (\mathbb{T}^d) )} \lesssim_{|I|} N^{\sigma}\|f\|_{L_{x}^{2}(\mathbb{T}^d )},
	\end{equation}
	equivalently, \begin{equation*}
		\|e^{it(-\Delta)^{\frac{\theta}{2}}}f\|_{L^{q}_{t}(I, L_{x}^{r} (\mathbb{T}^d) )} \lesssim_{|I|} \|f\|_{H^{\sigma}(\mathbb{T}^d )},
	\end{equation*} 
	for  $2 \leq q \leq \infty, 2 \leq r< \infty$ with $(q,r,d) \neq (2,\infty,2)$ satisfies  
	\begin{align}\label{admssible less equal}
		\frac{2}{q}+\frac{d}{r} \leq \frac{d}{2} \quad{and}\quad 
		\sigma = \begin{cases}
			\frac{1}{r} &  \text{if} ~~ \theta \in (1,\infty), \\
			\frac{2-\theta}{r} & \text{if} ~~  \theta \in (0,1).
		\end{cases}
	\end{align}
	We refer the readers to  \cite{nandakumaran2005, koh2015strichartz, bahouri2021strichartz, schippa2025refinements, tataru2000strichartz, blair2012strichartz} and the references therein for Strichartz estimates in different frameworks.

	\subsection{Orthonormal Strichartz estimates}
	In recent years, many researchers have devoted considerable attention to studying Strichartz estimates within the context of many-body quantum systems. In this context, Frank-Lewin-Lieb-Seiringer \cite{FrankJEMS2014,frank2017restriction} and Frank-Sabin \cite{frank2017restriction}, first establish Strichartz estimates for orthonormal families of initial data for the free Schr\"odinger equation, extending the classical Strichartz estimate  \eqref{ost} for a single function,  of the form
	\begin{equation}\label{OST R^n}
		\left\|\sum_{j\in J} \lambda_j|e^{-it\Delta}f_j|^2\right\|_{L_t^q(\mathbb{R},L_x^r(\mathbb{R}^d))} \leq C \|\lambda \|_{\ell^\frac{2r}{r+1}},
	\end{equation}
	where $\{f_j\}$ is an  orthonormal  system (ONS for short) in $L^2(\mathbb{R}^d)$, $\lambda=\{\lambda_j\}\in \ell^\frac{2r}{r+1}(\mathbb{C})$ and $r, q\in [1, \infty]$ satisfies $$ \frac{2}{q}+ \frac{d}{r}= d, \quad 1\leq r<\frac{d+1}{d-1}.$$
	Subsequently, the Strichartz estimate for single function on torus \eqref{SE Torus} (with $\theta=2$)   has been generalized to  orthonormal framework  by  Nakamura \cite{nakamura2020orthonormal}, yielding the following form
	\begin{eqnarray}\label{onstr}
		\left\| \sum_{j} \lambda_j  |e^{-it\Delta} P_{\leq N} f_j|^2\right\|_{L^q_t(I, L^r_x(\mathbb{T}^d)} \lesssim_{|I|}  N^{\frac{1}{q}} \|\lambda\|_{\ell^{ \frac{2r}{r+1}}}, 
	\end{eqnarray}
	for any $N\geq 1$, $\lambda=\{\lambda_j\}\in \ell^\frac{2r}{r+1}(\mathbb{C})$,   $\{f_j\}$ is an  orthonormal  system in $L^2(\mathbb{T}^d)$,   and $r, q\in [1, \infty]$ satisfies 
	\begin{equation}\label{SAP}
		\frac{2}{q}+ \frac{d}{r}= d, \quad 1\leq r<\frac{d+1}{d-1}.
	\end{equation}
	For general $\theta\in (0,\infty)$, the orthonormal analogue of the estimate \eqref{SE Torus}   has been investigated  recently by the first two  authors in \cite{bhimani2025orthonormal}. Furthermore, in a recent work, Wang–Zhang–Zhang \cite{wang2025strichartz} established analogous estimates on compact Riemannian manifolds without boundary.
	
	The motivation to extend fundamental inequalities from single to  system of orthonormal functions came from the theory of many body quantum mechanics. The first initiative work of such generalization goes back to the famous work established by Lieb and Thirring \cite{lieb1975bound}, which extended the Gagliardo–Nirenberg–Sobolev inequality, one of the key tools used to understand the stability of matter. There are several reasons to analyze estimates of the form \eqref{OST R^n} and \eqref{onstr}, in particular, if we take $f_j = 0$ for all $j$ except a single one in these orthonormal estimates, we obtain the usual one, and hence  these are the refinement of the usual classical Strichartz estimates \eqref{ost} on  $\mathbb{R}^d$  and \eqref{BurqEst} on $\mathbb{T}^d$, respectively. Secondly, these estimates can be apply to understand the dynamics of infinitely many fermions in a quantum system, in particular, in the
	analysis of the Hartree equations describing infinitely many particle system of fermions (see \eqref{Hartree} below), see \cite{chen2017global,chen2018scattering, sabin2014hartree,FrankJEMS2014, FrankJEMS2014}.  Recently, the orthonormal estimate of the form \eqref{OST R^n} has been further extended to various contexts in different frameworks.   We refer to 
	\cite{nealbez2021, neal2019, Hoshiya2024, HoshiyaJMP2025, feng2024orthonormal,mondal2025, wang2025strichartz} and references therein for recent works in the direction of orthonormal inequality and its applications. We  extend our analysis of the Schr\"odinger propagator $e^{it\Delta}$ on the torus from the single function setting to the orthonormal one, under the framework of partial regularity.
	
	\section{Statement of the main results}
	Our first  main   theorem  is a refined version of the Strichartz estimates \eqref{BurqEst} and  \eqref{SE Torus}.   
	\begin{theorem}[Strichartz estimates] \label{SE Single function}
		Let $d \geq 1,  1 \leq k \leq d, 2 \leq \widetilde{r} \leq r<\infty, $ and  $2<q \leq \infty$. If $ q, r, \tilde{r}$ satisfy 
		\begin{align}\label{eq0}
			\frac{2}{q} \geq (d-k)\left(\frac{1}{2}-\frac{1}{r}\right)+k\left(\frac{1}{2}-\frac{1}{\widetilde{r}}\right), 
		\end{align}
		then for $f\in {H}^{\frac{1}{q}}(\mathbb{T}^d)$, we have
		\begin{align}\label{For single}
			\left\|e^{-i t\Delta} f\right\|_{L_t^q\left(\mathbb{T} ; L_x^r\left(\mathbb{T}^{d-k} ; L_y^{\tilde{r}}\left(\mathbb{T}^k\right)\right)\right)} \lesssim\|f\|_{{H}^{\frac{1}{q}}(\mathbb{T}^d)}.
		\end{align}
	\end{theorem}  
	It should be noted that in general we have a sharp embedding $L^r_x(\mathbb T^{d-k}, L^{\tilde{r}}(\mathbb T^k)) \subsetneq  L^{\tilde{r}}(\mathbb T^d).$ Thus Theorem \ref{SE Single function} reveals that solution to the free Schr\"odinger equation  enjoy the better regularity compare to the classical Strichartz estimates \eqref{SE Torus}.

	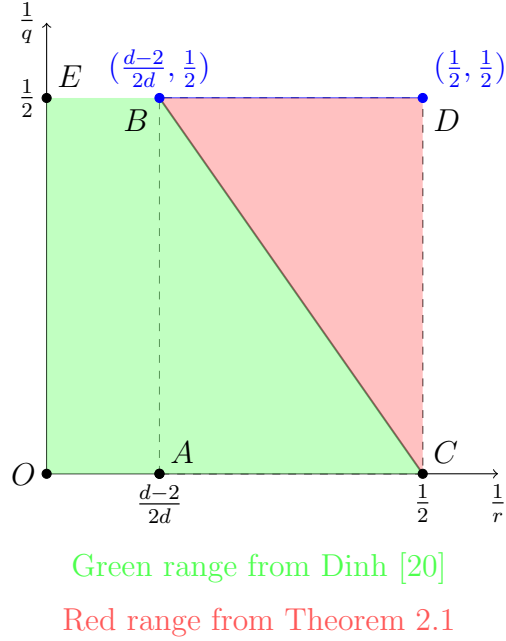
\begin{figure}[h]
		\centering
		\begin{tikzpicture}[scale=1]
			
			\draw[->] (0,0) -- (0,6) node[left] {$\tfrac{1}{q}$};
			\draw[->] (0,0) -- (6,0) node[below] {$\tfrac{1}{r}$};
			
			\draw[dashed] (1.5,0) rectangle (5,5);
			\draw[thick] (1.5,5)--(5,0);
			\draw[blue] (1.5,5)-- (5,5);

			
			\fill[green!40, opacity=0.6]
			(0,0) -- (0,5) --(1.5,5) -- (5,0) -- cycle;
			
			\fill[red!40, opacity=0.6]
			(5,5) -- (1.5,5) -- (5,0) -- cycle;
			
			\filldraw[blue] (5,5) circle (0.06) node[above right] {$(\tfrac{1}{2},\tfrac{1}{2})$};
			\filldraw[blue] (1.5,5) circle (0.06) node[above ] {$(\frac{d-2}{2d},\tfrac{1}{2})$};
			\filldraw[black] (0,0) circle (0.06) node[left] {$O$};
			\filldraw[black] (0,5) circle (0.06) node[left] {$\tfrac{1}{2}$};
			\filldraw[black] (1.5,0) circle (0.06) node[below] {$\tfrac{d-2}{2d}$};
			\filldraw[black] (5,0) circle (0.06) node[below] {$\tfrac{1}{2}$};
			\filldraw[] (5,5)  node[below right] {$D$};
			\filldraw[] (1.5,5)  node[below left] {$B$};
			\filldraw[] (1.5,0) circle (0.06) node[above right] {$A$};
			\filldraw[] (5,0)  node[above right] {$C$};
			\filldraw[] (0,5)  node[above right] {$E$};
			
		\end{tikzpicture}

		\begin{center}

			\begin{tikzpicture}[scale=1]	\node[left] at (4.81,4.1) { \color{green!70} Green range from      Dinh \cite{Dinh}}; \end{tikzpicture}\\
			
			\begin{tikzpicture}[scale=1]	\node[left] at (4.8,4.1) {{\color{red!65}  Red range from  Theorem \ref{SE Single function}}};  \end{tikzpicture}
		\end{center}
		\caption{Complete region for Strichartz estimate on $\mathbb{T}^d, d \geq 3.$}
		\label{fig:strichartz-region r= tilde r}
	\end{figure}
	For any two points $X,Y\in [0,\frac{1}{2}]^2,$ we use notation 
	$$[X,Y]=\{(1-\tau)X+\tau Y: \tau \in [0,1]\}$$
	to represent the line segment connecting $X$ and $Y.$ We denote $XYZ$ for the convex hull of points $X,Y,Z \in [0,\frac{1}{2}]^2;$ and $XYZ \setminus X$ for the convex hull of points $X,Y,Z \in [0,\frac{1}{2}]^2$ excluding the point $X.$ Similarly $XYZ \setminus [X,Y]$ denote for the convex hull of points $X,Y,Z \in [0,\frac{1}{2}]^2$ excluding the line $[X,Y].$  
	\begin{corollary}
		Let $d \geq 1, (\frac{1}{r},\frac{1}{q}) \in OEDC \setminus [B,D]$ in Figure \ref{fig:strichartz-region r= tilde r}.
		Then for $f\in {H}^{\frac{1}{q}}(\mathbb{T}^d)$, we have
		\begin{align}\label{For single r= tilda r}
			\left\|e^{-i t\Delta} f\right\|_{L_t^q\left(\mathbb{T} ; L_x^r\left(\mathbb{T}^{d} \right)\right)} \lesssim\|f\|_{{H}^{\frac{1}{q}}}.
		\end{align}
	\end{corollary}
	Note that,  for $2 \leq r = \tilde r < \infty,$ and $2 <q \leq \infty,$   by Theorem \ref{SE Single function}, estimate \eqref{For single r= tilda r} holds whenever
	\[
	\frac{2}{q} + \frac{d}{r} \ge \frac{d}{2}, \quad \text{i.e., }  ~\left(\frac{1}{r},\frac{1}{q}\right) \in BDC \setminus [B,D].
	\]
	On the other hand, Dinh \cite{Dinh} has already shown that the estimate that the estimate \eqref{For single r= tilda r} (see \eqref{SE Torus} for $\theta = 2$ and \eqref{BurqEst}) holds whenever
	\[
	\frac{2}{q} + \frac{d}{r} \le \frac{d}{2}, 
	\qquad (q,r,d) \neq (2,\infty,2).
	\]
	Combining these two facts, we conclude that the estimate \eqref{For single r= tilda r} holds for
	$(\frac{1}{r},\frac{1}{q}) \in OEDC \setminus [B,D]$
	in Figure \ref{fig:strichartz-region r= tilde r}.  Thus  Theorem \ref{SE Single function} extend the range of $(q,r)$ of  Burq, G\'erard and Tzvetkov   \eqref{BurqEst} and   Dinh \eqref{SE Torus}. We refer to Subsection \ref{cop} for further comments.


	

	\begin{remark}
		The restriction  on $2 \leq \widetilde{r} \leq r<\infty$ in Theorem \ref{SE Single function} is due to the restrictions of Littlewood-Paley theorem on mixed Lebesgue spaces $L_x^r L_y^{\widetilde{r}}\left(\mathbb{T}^{d-k} \times \mathbb{T}^k\right)$, see Proposition \ref{mixed LP}.
	\end{remark}
	
	Our next main result  is the following orthonormal extension of the refined Strichartz estimates \eqref{For single} in the case of \textbf{sharp admissible condition}, i.e., for $1 \leq q,r, \tilde{r} \leq \infty,$ when  $(q,r,\Tilde{r})$ satisfies  the condition 
	\begin{eqnarray}\label{adc}
		\frac{1}{q}+\frac{1}{2}\left(\frac{d-k}{r}+\frac{k}{\Tilde{r}} \right)=\frac{d}{2}.
	\end{eqnarray}
	Note that, in the case $r=\tilde{r},$ the condition \eqref{adc} reduces to \eqref{SAP}. 
	We define $\gamma$ in terms of $r$ and $\tilde{r}$ such that  \begin{equation}\label{gamma def}
		\frac{1}{\gamma}=\frac{1}{r}\left(1-\frac{k}{d}\right)+\frac{1}{\Tilde{r}}\frac{k}{d}.
	\end{equation}
	Then by the sharp admissibility condition \eqref{adc}, one can  obtain 
	\begin{equation}\label{adc gamma}
		\frac{2}{q}+\frac{d}{\gamma}=d.
	\end{equation}

	\begin{theorem}[Orthonormal Strichartz estimates]\label{ose}
		Let $d \geq 1,1 \leq \Tilde{r} \leq r \leq \infty$ and $1 \leq \gamma < \frac{d+1}{d-1},$ where $\gamma$ defined in \eqref{gamma def}, and  $$1 \leq \alpha' \leq \frac{2 \gamma}{\gamma+1}.$$ Then, for any triple $(q,r,\Tilde{r})$ satisfying \eqref{adc},  $N > 1,$  $\lambda \in \ell^{\alpha'}, $ and  $(f_j)$ are a family of orthonormal functions in $L^{2}(\mathbb{T}^d)$,   the following orthonormal estimate 
		\begin{equation}\label{IN12}
			\left\| \sum_{j} \lambda_j  |e^{it \Delta} P_{\leq N} f_j|^2\right\|_{L^q_tL^r_xL^{\Tilde{r}}_{y} (I\times \mathbb{T}^{d-k} \times \mathbb{T}^{k})} \lesssim_{|I|}N^{\frac{1}{q}} \| \lambda \|_{\ell^{\alpha'}}
		\end{equation}
		holds true.
	\end{theorem}
	
Note that  that Theorem \ref{ose} gives orthonormal Strichartz estimates  in mixed  Lebesgue spaces of the form $ 
L_x^r L_y^{\tilde r}(\mathbb T^{d-k} \times \mathbb T^k)$,
rather than in a single Lebesgue space over the full torus. In contrast, Nakamura’s result \cite{nakamura2020orthonormal} is formulated in \(L_{x,y}^{\tilde r}(\mathbb T^d)\), where the same integrability exponent is imposed uniformly on all spatial variables. When \(r\neq \tilde r\), the mixed space \(L_x^rL_y^{\tilde r}\) is strictly contained in \(L_{x,y}^{\tilde r}\), and therefore provides more detailed information on the spatial integrability of the solution. In particular, this formulation allows the exponent \(r\) to take values larger than the classical threshold \(\frac{d+1}{d-1}\), while remaining consistent with the dispersive scaling, a feature that does not arise in estimates based on a single spatial exponent.

\begin{remark}
     Estimate \eqref{IN12} recover known  Nakamura's estimate \eqref{onstr}. In fact,  for $r=\tilde{r}$, by \eqref{gamma def},  we have  $\gamma=r$; and in this case the   admissible condition  \eqref{adc gamma} reduces to  $    \frac{2}{q}+\frac{d}{r}=d.$  
\end{remark}
	\begin{remark}
		Using the triangle inequality together with the refined Strichartz estimate for a single function \eqref{For single} yields \eqref{IN12} in the case  $\alpha'=1$ 
		, and at this point no orthogonality of the initial data is needed. It is therefore natural to ask how much additional improvement, if any can be obtained from exploiting orthogonality by allowing 
		$\alpha'\geq  1$ to be as large as possible. In the case $r=\tilde r$, the value $\alpha'=\frac{2r}{r+1}$ is sharp in the sense that the estimate \eqref{IN12} fails for $\alpha'>\frac{2r}{r+1}$; see \cite{nakamura2020orthonormal}. In contrast, within the partial regularity framework considered here, we do not know whether the corresponding upper bound $\alpha'=\frac{2\gamma}{\gamma+1}$ is optimal or not.
	\end{remark}
	    To prove Theorem \ref{ose}, we will not prove directly the inequality \eqref{IN12}, rather  we prove an inequality which is dual to \eqref{IN12}, see Lemma \ref{PL1}.  First, we reformulate \eqref{IN12} in terms of the Fourier extension operator restricted to the $d-$dimensional cube $
	S_{d,N}  = \Z^{d} \cap [-N, N]^d$; (see Remark \ref{usrev}) and prove the dual of the reformulate  equation. For the dual inequality, a key step is the kernel estimate, which relies on the observation that the higher-dimensional kernel factors into a product of one-dimensional kernels; see Proposition \ref{l8}. Then we follow the spirit of the Hardy-Littlewood-Sobolev inequality  in one dimensional via Frank-Sabin’s $TT^*$  argument in Schatten space and Stein's analytic interpolation method in Schatten spaces. 
	\subsection{Applications to well-posedness}\label{symbols}
	To state our well-posedness results, we briefly set notations. We decompose the spatial variable \( z \in \mathbb{T}^d \) as
	\( z=(x,y)\in \mathbb{T}^{d-k}\times \mathbb{T}^k \), where \( 1\le k<d \).
	Correspondingly, for \( \xi\in\mathbb{Z}^d \), we write
	\( \xi=(\xi_1,\xi_2)\in\mathbb{Z}^{d-k}\times\mathbb{Z}^k \), and
	\( |\xi|^2=|(\xi_1,\xi_2)|^2=|\xi_1|^2+|\xi_2|^2 .\)
	For $s \in \mathbb{R},$ we define the Fourier multiplier operators on $\mathbb{T}^{d-k}\times \mathbb{T}^k$ as
	\[
	\langle \nabla \rangle^{s} f= (1 + |\nabla|^{2})^{\frac{s}{2}}f=\left((1 + |(\xi_{1},\xi_{2})|^{2})^{\frac{s}{2}} \widehat{f}(\xi_{1},\xi_{2}) \right)^{\vee}
	\]
	and 
	$$\langle \nabla_{y} \rangle^{s} f= (1 + |\nabla_{y}|^{2})^{\frac{s}{2}}f= \mathcal{F}_{y}^{-1}\left((1 + |\xi_{2}|^{2})^{\frac{s}{2}} \mathcal{F}_{y}(f)(\cdot,\xi_{2}) \right),$$
	where $\mathcal{F}_{y}^{-1}$ is partial Fourier transform (see Section \ref{sec2}).
	The  $L^r$\textbf{-Sobolev spaces} on $\mathbb T^d$ is given by the following norm
	$$\|f\|_{ W^{s,r}}= \|\langle \nabla \rangle^{s} f\|_{L^r(\mathbb T^d)}. $$
	For $r=2,$ we simply  write $W^{s,2}=H^s.$ For $2\leq q \leq \infty$ and $\frac1q \leq s \leq 1+\frac{1}{q},$ we define \textbf{partial regularity Sobolev space} as follows
	\begin{equation}\label{evolution space embedding}
		X_k^s = \big\{ f : \mathbb{T}^d \to \mathbb{C} : \langle \nabla_y \rangle^{s-\frac1q} f \in H^{1/q} \big\},
	\end{equation}
	via the following norm 
	$$\|f\|_{X^s} = \|\langle \nabla_y\rangle^{s-\frac1q} f\|_{H^{1/q}(\mathbb T^{d-k} \times \mathbb T^k)}.$$
	Note that\footnote{For instance,  consider the function $h(x,y) = f(x) g(y)$ with the property that $f \in H^{1/q}_x(\mathbb{T}^{d-k}) \setminus H^s_x(\mathbb{T}^{d-k})$ and $g \in H^s_y(\mathbb{T}^{k}).$ Then one can see that $h\in X^s$ but $h\notin {H}^s(\mathbb{T}^{d}).$} 
	\[H^{1/q} = X^{1/q} \quad \text{and} \quad {H}^s \subsetneq	X_k^s \  \  \text{for} \  \ s>1/q, 1\leq k <d. \  \]
	In view of Theorem \ref{SE Single function} (in the case of sharp admissible condition) and Lemma \ref{nonlinearity estimate x y}, for  $1 \le k < d$, the  set of all \textbf{refined Schr\"odinger triplets} is defined by
	\begin{equation}\label{Adk}
		\mathbb{A}(d, k)
		= \left\{
		(q, r, \tilde{r}) :
		2 < q \le \infty,\;
		2 \le \tilde{r} \le r < \infty,\;
		\frac{2}{q} + \frac{d - k}{r} + \frac{k}{\tilde{r}} = \frac{d}{2}
		\right\}.
	\end{equation}
	
	Many authors have studied the  NLS  with the nonlinearity $F_p(u)=|u|^p;$ which is non-gauge invariance type nonlinearity; see \cite{MINonGauge, HiroNGI, AdvMathNoGI}. The nonlinear effect of $|u|^p$ is quite different from that of $|u|^{p-1}u.$  Notice that in   \eqref{NLS}, the nonlinearity could of  non-gauge invariance type, e.g. $w\ast |u|^2$; and its nonlinear effect could be different from that of Hartree type $(w\ast |u|^2)u$ nonlinearity, which is gauge invariance (see \eqref{HN} and \eqref{Hartree} below). This  convolution type nonlinearity $w\ast F_p(u)$ also has been extensively studied  in the  context of heat equation to model nonlocal interactions and derive sharp functional inequalities, see \cite{cazenave2008, mitidieri2005, loayza2012global}. In this work, we assume that the potential $w $ belongs to the Sobolev space, $ W^{\frac{2}{q},1}(\T^d)$, and the result below addresses the well-posedness of the nonlinear problem \eqref{NLS}.
	
	\begin{theorem}[local well-posedness]\label{wellposedness}
		Let $d \ge 3,$ $0 \leq s < 1$ and $w \in W^{\frac{2}{q},1}.$ Suppose that 
		\[
		1 < p < 1 + \frac{4}{d-2s}
		\quad \text{and} \quad 
		\langle \nabla_{y} \rangle^{s} f \in H^{\frac{1}{q}}(\mathbb{T}^{\,d-2} \times \mathbb{T}^{2}).
		\]
		Then there exists a time
		\[
		T = T\!\left( \| w \|_{W^{\frac{2}{q},1}},\, \|\langle \nabla_{y} \rangle^{s} f\|_{H^{\frac{1}{q}}} \right) > 0
		\]
		for which \eqref{NLS} admits a unique solution $u$ satisfying
		\[
		\langle \nabla_{y} \rangle^{s} u 
		\in C([0,T]; H^{\tfrac{1}{q}})
		\cap L_{t}^{q}([0,T]; L_{x}^{r} L_{y}^{\tilde{r}})
		\]
		for every $(q,r,\tilde{r}) \in \mathbb{A}(d,2)$.
		
	\end{theorem}
	 To prove Theorem \ref{wellposedness}, we establish crucial nonlinear estimate (Lemma \ref{nonlinearity estimate x y}) in  the space  $ {L_t^{q'} L_x^{r'} W_y^{s, \tilde{r}'}}$; which enables us to apply fixed-point argument.    This estimate relies fundamentally on the fractional chain rule and the Sobolev embedding on the 2D torus.
	\begin{corollary}\label{in term of Xsk}
		Let $d \ge 3,$ $(q,r,\tilde{r}) \in \mathbb{A}(d,2),$ $\frac{1}{q} \leq s < 1+\frac{1}{q}$ and $w \in W^{\frac{2}{q},1}.$ Suppose that 
		\[
		1 < p < 1 + \frac{4}{d-2s}
		\quad \text{and} \quad 
		f \in X_2^s.
		\]
		Then there exists a time
		\[
		T = T\!\left( \| w \|_{W^{\frac{2}{q},1}},\, \|f\|_{X_2^s} \right) > 0
		\]
		for which \eqref{NLS} admits a unique solution $u$ satisfying
		\[
		u 
		\in C([0,T]; X_2^s)
		\cap L_{t}^{q}([0,T]; L_{x}^{r} W_{y}^{s-\frac{1}{q},\tilde{r}}).
		\]
	\end{corollary}
	Under the same assumption on $d,$ $(q,r,\tilde{r}),$ $s$ and $p$ in above Corollary \ref{in term of Xsk}, let $w \in W ^{s+\frac{1}{q},1} $ and $f \in H^s.$  Then there exists a time
	\[
	T = T\!\left( \| w \|_{W ^{\frac{2}{q},1}},\, \|f\|_{X_2^s} \right) > 0
	\]
	for which \eqref{NLS} admits a unique solution satisfying
	\[
	u 
	\in C([0,T]; H^s)
	\cap L_{t}^{q}([0,T]; L_{x}^{r} W_{y}^{s-\frac{1}{q},\tilde{r}}).
	\]
	Note that this gives local well-posedness in the Sobolev space $H^s.$
	\begin{remark}
		The restriction   $1 < p < 1 + \frac{4}{d-2s}$ and the specific choice $k=2$ in the decomposition $\mathbb{T}^{d-2} \times \mathbb{T}^2$ in the assumptions of Theorem \ref{wellposedness} arise from the Sobolev embeddings used in the argument. In particular, they are required for applying the embeddings between Sobolev spaces and into $L^r$ spaces; see \eqref{Embedding} in the proof of Theorem \ref{wellposedness}.  
	\end{remark}
	
	\subsection{Hartree equations for infinitely many particles}
	Consider $M$ couple of Hartree equations that describe the dynamics of $M$ fermions interacting via a potential $\omega: \mathbb{T}^{d-k} \times \mathbb{T}^k \to \mathbb R$ as follows:
\begin{align}\label{HN}
    \left\{\begin{array}{rlrl}
        i \partial_t u_1 & =\left(-\Delta+\omega * \rho\right) u_1, & & \left.u_1\right|_{t=0}=f_1 \\
        & \vdots \\
        i \partial_t u_M & =\left(-\Delta+\omega * \rho\right) u_M, & & \left.u_M\right|_{t=0}=f_M,
    \end{array}\right.
\end{align}
where $(x, t) \in \mathbb{T}^{d-k} \times \mathbb{T}^k \times \mathbb{R},$ and to account for the Pauli principle, the family $\left\{f_j\right\}_{j=1}^M$ assumed to be an orthonormal system in $L^2\left(\mathbb{T}^{d-k} \times \mathbb{T}^k \right)$. Here, $\rho$ is the total density functions of particles defined by 
$$\rho(x, t)=\sum_{j=1}^M\left|u_j(x, t)\right|^2.$$ 

Let us introduce the density matrix corresponding to \eqref{HN} as
\begin{eqnarray*}
    \gamma_M (t) = \sum_{j=1}^M | u_j(t)\rangle \langle u_j(t) |,
\end{eqnarray*}
where $|u \rangle \langle v|$ denotes the operator $f\mapsto \langle v, f\rangle u$ from $L^2\left(\mathbb{T}^{d-k} \times \mathbb{T}^k \right)$ into itself. Then system \eqref{HN} can be rewritten as a operator-valued equation
\begin{equation}\label{oHN}
    \begin{cases}
        i \partial_t\gamma_{M} = [-\Delta+ \omega \ast \rho_ {\gamma_M}, \gamma_M],\\
        \gamma_M(t=0) = \sum_{j=1}^M | f_{j}\rangle \langle f_{j} |,
    \end{cases}
\end{equation}
where $[X, Y]=XY-YX,$ is the operator commutator and the density function $\rho_ {\gamma_M}$ is given by 
\[\rho_{\gamma_M} (t,z)= \gamma_{M} (t, z,z), \] 
where $\gamma_M (t, \cdot, \cdot)$ is the integral kernel of operator $\gamma_M(t)$.
	Our goal is to investigate the behavior of the solution of system \eqref{HN} as $M \rightarrow \infty$. In this case, we naturally arrive at the following operator-valued equivalent formulation 
\begin{align}\label{Hartree} 
    \left\{\begin{array}{l}i \partial_t \gamma=\left[-\Delta+\omega * \rho_\gamma, \gamma\right], \\ \left.\gamma\right|_{t=0}=\gamma_0.\end{array}\right.
\end{align}
It is known that equation \eqref{Hartree} describes the mean-field dynamics of an interacting gas containing infinitely many fermions in the torus $\mathbb T^d.$ See \cite{chen2017global,chen2018scattering, sabin2014hartree,FrankJEMS2014, FrankJEMS2014}.
Here, the unknown $\gamma=\gamma(t)$ is a bounded self-adjoint operator on $L^2(\mathbb T^d)$ representing the one-particle density matrix and $\rho_\gamma:\mathbb R \times \T^{d-k}\times \T^k \to \mathbb R$ is the scalar \textbf{density function} associated to $\gamma(t),$ formally defined by 
$$\rho_\gamma(t, z)=\gamma (t, z,z),$$
where $\gamma (t, x, x')$ denote (with the abuse of notation) the integral kernel of operator $\gamma(t),$ i.e.,
\[\gamma (t) \phi (x) = \int_{\mathbb T^d} \gamma (t, x, x') \phi (x') dx'.\]  
The potential belongs to mixed-type Besov spaces (see Subsection \ref{Bs})	
$$\omega \in B^{s}_{(r',\tilde{r}'), \infty}(\T^{d-k} \times \T^k) $$ 
In particular, for $s\in\mathbb R$, the localization argument combined with Littlewood--Paley theory yields
\[
|x|^{-a}\in B^{s}_{(r',\tilde r'),\infty}
\bigl(\mathbb T^{d-k}\times\mathbb T^k\bigr)
\quad \text{whenever} \quad
a+s\leq \frac{d-k}{r'}+\frac{k}{\tilde r'}.
\]

For $s\in\mathbb R$ and $\alpha' \geq 1,$ we introduce the \textbf{Sobolev-Schatten space} $\Sp^{\alpha',s}$, equipped with the norm
\begin{equation}\label{sobolev schatten}
    \|\gamma\|_{\Sp^{\alpha',s}}
=
\|\langle\nabla\rangle^{s}\,\gamma\,\langle\nabla\rangle^{s}\|_{\Sp^{\alpha'}}.
\end{equation}
Here, $\langle\nabla\rangle^{s}$ denotes the Fourier multiplier
on $\mathbb T^{d-k}\times\mathbb T^{k}$ (see Subsection~\ref{symbols}),
and $\Sp^{\alpha'}$ is the Schatten class of order $\alpha'$
on $L^{2}(\mathbb T^{d-k}\times\mathbb T^{k})$ (see Subsection~\ref{schp}).
We are now ready state our well-posedness result for \eqref{Hartree} with data in Sobolev-Schatten space.

\begin{theorem}\label{TIN7}
    Let the triple \(\left( q,r,\tilde{r} \right) \) be as in \eqref{adc}, \(s > \frac{2}{q}\), and $\omega \in B^{s}_{(r',\tilde{r}'), \infty}(\T^{d-k} \times \T^k)$. Assume that \(\gamma_{0} \in \Sp^{\frac{2 q}{q+1}, s} (L^2_z (\T^{d-k} \times \T^k))\). Then 
    
    \begin{enumerate}
        \item \label{TIN71} (local well-posedness) there exist 
        \(
        {\small T = T\left(\| \gamma_{0} \|_{\Sp^{\frac{2 q}{q+1}, s}}, \| \omega \|_{B^{s}_{(r',\tilde{r}'), \infty}} \right) > 0}
        \)
        and a unique solution 
        \[
        \gamma \in C([0, T], \Sp^{\frac{2 q}{q+1}, s} (L^2 (\T^{d-k} \times \T^k)))
        \]
        satisfying \eqref{Hartree} on \([0, T] \times \T^{d-k} \times \T^k\) whose density 
        \(
        \rho_\gamma \in L_{t}^{q} L_{x}^{r}L_y^{\tilde{r}} ([0, T] \times \T^{d-k} \times \T^k).\)
        \item \label{TIN72} (small data global well-posedness) For any $T >0,$ there exists a small\\ ${\small R_{T}=R_{T}\left(\| \omega \|_{B^{s}_{(r', \tilde{r}', \infty}},s\right)}$ $>0$ such that if $\| \gamma_{0} \|_{\Sp^{\frac{2 q}{q+1}, s}} \leq R_{T},$ then there exists a unique solution $$\gamma \in C([0, T], \Sp^{\frac{2 q}{q+1}, s} (L^2 (\T^{d-k} \times \T^k))),$$ satisfying \eqref{Hartree} on $[0, T] \times \T^{d-k} \times \T^k$ whose density $\rho_{\gamma} \in L_t^q L_x^r L_y^{\tilde r}\!\left([0,T]\times\mathbb T^{d-k}\times\mathbb T^k\right).$
    \end{enumerate} 
\end{theorem} 
    
It is worth noting that several authors have studied the Hartree equation for infinite number of fermions within various frameworks. In particular, Frank--Sabin \cite{frank2017restriction} and 
Lewin and Sabin \cite{lewin2015hartree, lewin2014hartree} initiated the study of \eqref{Hartree}, establishing local and global existence results in suitable Schatten spaces under strong assumptions on the interaction potential $\omega$. For compact domains, the works \cite{nakamura2020orthonormal} and \cite{wang2025strichartz} focused on interaction potentials of the type $\omega(x) = |x|^{-a}$. In comparison, Bez, Lee, and Nakamura \cite{nealbez2021} considered the case when the potential $\omega \in B^{s}_{q, \infty}(\mathbb R^n)$ for some $s,q$. We also refer to \cite{chen2017global, chen2018scattering, HoshiyaJMP2025, bhimani2025orthonormal} and references therein for related results and recent contributions to this topic.

Note that the density function in Theorem \ref{TIN7} enjoy regularity in mixed Lebesgue spaces, which is in quite contrast with the several aforementioned known results in the literature. We refer to following Subsection \ref{cop} for further comments.

	\subsection{Proof techniques and novelties}\label{cop}
	\begin{itemize}
		\item[--]We establish Theorem \ref{SE Single function} in Section \ref{sec3}. Since on the torus no dispersive estimate holds globally in time or uniformly across all frequencies, a frequency–localized approach is unavoidable. For this reason, our first step is to reduce the expression in \eqref{For single} to the frequency–localized Schrödinger propagator $e^{it\Delta}P_{\le N}$, using the Littlewood–Paley theorem on the mixed Lebesgue space $L_x^r L_y^{\tilde r}(\mathbb{T}^{d-k}\times\mathbb{T}^k)$. This reduction allows us to isolate dyadic frequency contributions and treat each scale separately. After spatial localization, we further localize in time by restricting to intervals $I_N$ of length $1/N$, on which the localized propagator satisfies estimates strong enough for our analysis. Combining these bounds with fixed-time  and frequency localized estimates (Lemmas \ref{eq6} and  \ref{Loc}), we conclude the proof of Theorem \ref{SE Single function}.\\
        
		\item[--] Littlewood–Paley theory on mixed Lebesgue spaces $L_x^r L_y^{\widetilde r}\left(\mathbb{T}^{d-k}\times\mathbb{T}^k\right)$ for $1<r,\widetilde r<\infty$ does not appear to be available in the existing literature. In the present work, we develop this theory using the classical transference method together with the Khintchine inequality for Rademacher signs. See Section \ref{sec22}.\\
		
		\item[--] The proof of Theorem \ref{ose} relies on the standard kernel estimate \eqref{kernel} on the torus. A crucial ingredient in our argument is the embedding  between $L^p$ spaces on the torus; as a result, this method does not readily extend to the setting of partial frames on $\R^d.$\\ 
		
		\item[--] Although the well-posedness of the Hartree equation \eqref{Hartree} for infinitely many fermions stated in Theorem~\ref{TIN7} can be established by a standard argument based on orthonormal Strichartz estimates, our mixed-norm framework requires additional preparatory results. Specifically, it is necessary to first establish a vector-valued version of the Littlewood--Paley theorem for operator densities (Theorem~\ref{vector LP}) and a vector-valued Bernstein inequality (Corollary~\ref{vvr}). These results are of independent interest and provide the foundation for the nonlinear analysis in the mixed-norm setting.\\	
        
		\item[--]  The mixed Lebesgue spaces \( L_x^r L_y^{\tilde r}(\mathbb{T}^{d-k}\times\mathbb{T}^k) \) extend the classical Lebesgue spaces and naturally arise in the study of functions depending on independent variables with distinct structural features. This additional flexibility makes mixed Lebesgue spaces particularly well suited for the analysis of partial differential equations, regularity theory, partial Fourier transforms, and related applications. In the mixed-norm setting \( L_x^r L_y^{\tilde r}(\mathbb{T}^{d-k}\times\mathbb{T}^k) \) with \( 1<r,\tilde r<\infty \), our Littlewood--Paley theory (see Proposition~\ref{mixed LP}) provides the first systematic treatment on the torus and extends the corresponding results of \cite{ward2010mixedLebesgue}.\\

		\item[--] To obtain the vector-valued version of Bernstein’s inequality stated in Corollary~\ref{vvr}, we establish $\ell^2$-valued extension inequalities (Theorem~\ref{mixed vector valued}) in mixed-norm spaces and combine them with the H\"ormander–Mikhlin Fourier multiplier theorem (Theorem~\ref{Hormander-mikhlm}) and the transference principle.\\
				
		\item[--]  We also extended our Littlewood–Paley theory (see Proposition \eqref{mixed LP}) to densities of operators in Theorem \ref{vector LP}, which is completely new in the direction of mixed Lebesgue spaces for densities of operators on torus $\T^d$.  We follow an approach similar to that of Sabin \cite{sabin2016littlewood} to prove vector valued Littlewood–Paley theory. However, at a certain stage (see \eqref{11}), when applying the H\"ormander–Mikhlin multiplier theorem in the mixed-norm setting, the argument used in \cite{sabin2016littlewood} no longer applies directly. In particular, the Khintchine inequality for Rademacher functions, which is employed in Sabin’s work, cannot be used in this context directly. To address this difficulty, we instead apply Jensen’s inequality, followed by Minkowski’s integral inequality.
		However, one can prove  the Littlewood–Paley theory  for  densities
		of operators on $\R^d$ in a similar manner but without using the transference principle. Note that, we can derive our  usual Littlewood-Paley theory for single function (see Proposition \ref{mixed LP})  from our vector-valued Littlewood-Paley theory (see Theorem \ref{vector LP}), for details,  we refer to  Subsection \ref{vector valued LP decomposition}.
		
		The motivation to extend Littlewood-Paley decomposition for mixed spaces to operator densities comes from many-body quantum mechanics. Indeed, a system of $d$ fermions in $\mathbb{R}^d$ can be realized via an orthogonal projection $\gamma$ on $L^2\left(\mathbb{R}^d\right)$ of rank $d$. For detailed study on vector valued Littlewood-Paley and its application on $\R^d$, we refer to \cite{sabin2016littlewood} and references therein.

	\end{itemize}

	Apart from the introduction, the paper is organized as follows: In Section \ref{sec2}, we review harmonic analysis on the torus, the relevant function spaces, and preliminary tools that will be employed throughout the paper. We also introduce the partial-regularity framework on torus.  In Section \ref{sec22}, we define and study some important properties for  Littlewood–Paley operators on mixed Lebesgue spaces and  vector-valued Littlewood-Paley theorem for densities of operators. Section \ref{sec3} is devoted to proving refined Strichartz estimates for single function using the Littlewood–Paley theory torus.  We investigate the local well-posedness for the nonlinear  Schr\"odinger equation (\ref{NLS}) on torus with partially regular initial data using the refined Strichartz estimes in Section \ref{sec4}.  In Section \ref{sec5},  we extend these refined Strichartz estimates from single to a system of orthonormal functions.  Finally, we discuss the well-posedness of the Hartree equation \eqref{Hartree} for infinitely many fermions as an application of our obtained OSE in Section \ref{sec7}.
	
	\section{Preliminaries}\label{sec2}
	In this section  we recall basic facts, function spaces, and preliminary tools that will be used throughout the paper.
	We state the definitions for the full operator and for the operators acting in the
	$y$-variable, noting that the analogous constructions in the $x$-variable follow
	in the same manner.
	We write \( A \lesssim B \) to denote that there exists a constant \( C > 0 \) such that \( A \leq C B \). Similarly, we use \( A \lesssim_u B \) to indicate that there exists a constant \( C(u) > 0 \), depending on \( u \), such that \( A \leq C(u) B \).
	We write \( A \sim B \) if both \( A \lesssim B \) and \( B \lesssim A \) hold.
	The notation \( A \approx B \), though less rigorous, will sometimes be used to indicate that \( A \) and \( B \) are essentially of the same size, up to small and insignificant error terms.
	
	As discussed in the introduction, we write \( z = (x,y) \in \mathbb{T}^{d-k} \times \mathbb{T}^{k} \) for \( 1 \le k \le d \), with the convention that the \( y \)-variable is absent when \( k = d \). Accordingly, \( x \in \mathbb{T}^{d-k} \), \( y \in \mathbb{T}^{k}, \) and the product measure decomposes as \( dz = dx\,dy. \) Similarly, if $\xi \in \mathbb{Z}^{d-k} \times \mathbb{Z}^{k}$, we write 
	$\xi = (\xi_{1},\xi_{2})$, where $\xi_{1} \in \mathbb{Z}^{d-k}$ and 
	$\xi_{2} \in \mathbb{Z}^{k}$, and we denote $z \cdot\xi = x \cdot \xi_{1}+ y \cdot \xi_{2}.$
	
	We make use of mixed Lebesgue spaces of the form \( L^p_x(\mathbb{T}^{d-k}, L^q_y(\mathbb{T}^k)) \), which we also denote by \( L^p_xL^q_y(\mathbb{T}^{d-k} \times \mathbb{T}^k) \) or simply \( L^p_xL^q_y \). These spaces are equipped with the norm:
	$$ \|f\|_{L^p_xL^q_y} = \left( \int_{\mathbb{T}^{d-k}} \left( \int_{\mathbb{T}^k}  |f(x, y)|^q\, dy \right)^{\frac{p}{q}} dx \right)^{\frac{1}{p}}.$$

	We define the \textbf{Fourier transform} of a function $f$ as follows
	$$(\mathcal{F}f)(\xi)= \widehat{f}(\xi)=  \int_{\mathbb{T}^{d-k}}\int_{\mathbb{T}^k}f(z) e^{-2 \pi i z \cdot \xi}\,dz, \quad \xi \in \mathbb{Z}^{d-k}\times \mathbb{Z}^k,$$
	and we denote by $f^{\vee}$ the inverse Fourier transform of function $f: \Z^{d-k} \times \Z^k \to \mathbb{C}.$
	On the other hand, the \textbf{partial Fourier transform} with respect to the $y$-variable is defined as 
	$$(\mathcal{F}_yf)(x, \xi_{2}) = \int_{\mathbb{T}^k}f(x, y) e^{-2 \pi i y \cdot \xi_{2}}\,dy, \quad \xi_{2} \in  \mathbb{Z}^k, $$ 
	and we denote by $\mathcal{F}_{y}^{-1}(f(x, \cdot))$ the inverse partial Fourier transform with respect to the $y$-variable of function $f(x, \cdot):\Z^k \to \mathbb{C}.$

	Let $\mathfrak{M} : \mathbb{Z}^d \to \mathbb{C}$ (on the torus) or $\mathfrak{M} : \mathbb{R}^d \to \mathbb{C}$ (on Euclidean space) be a given function. The associated {\bf Fourier multiplier operator} $T_{\mathfrak{M}}$ acts on a function $f$ by multiplying its Fourier transform by $\mathfrak{M}$.
	
	On $\mathbb{R}^d$, we define
	\[
	\widehat{T_{\mathfrak{M}} f}(\xi_1) := \mathfrak{M}(\xi_1)\,\widehat{f}(\xi_1), 
	\qquad \xi_1 \in \mathbb{R}^d,
	\]
	whenever the right-hand side is well defined. On $\mathbb{T}^d$, we define
	\[
	\widehat{T_{\mathfrak{M}} f}(\xi_2) := \mathfrak{M}(\xi_2)\,\widehat{f}(\xi_2), 
	\qquad \xi_2 \in \mathbb{Z}^d.
	\]
	
	We define two Fourier multiplier operators on $\mathbb{T}^{d-k} \times \mathbb{T}^k$ as follows:
	$$
	|\nabla|^{s} f
	=\left(|(\xi_{1},\xi_{2})|^{s} \, \widehat{f}(\xi_{1},\xi_{2})\right)^{\vee} \qquad \text{and} \qquad |\nabla_{y}|^{s} f
	=\mathcal{F}_{y}^{-1}\left(|\xi_{2}|^{s} \, \mathcal{F}_{y}(f)(\cdot,\xi_{2})\right).
	$$
	
	\begin{lemma}[Sobolev embedding (Corollary 1.2 in \cite{tadahiro2013sobolevtorus})]\label{sobolev embbeding T}
		Let $d \geq 1,$ $s>0$ and $1 < r < \tilde{r} < \infty.$ Then the Sobolev embedding
		\[ 
		W^{s,r}(\T^d) \hookrightarrow L^{\tilde{r}}(\T^d)
		\]
		holds under the condition
		\[
		\frac{s}{d} \geq \frac{1}{r} - \frac{1}{\tilde{r}}.
		\]

	\end{lemma}
	\begin{remark}
		We will frequently rely on the following standard properties of Sobolev norms on $\mathbb{T}^d$.  
		These estimates follow from the transference principle for $L^p$-multipliers together with  
		\cite[Theorem~6.3.2]{bergh1976interpolationspace}.  
		For $s \geq 0$ and $1 \le r \le \infty$, we have
		\begin{equation}\label{inhomogeneous dominate}
			\|\,|\nabla|^{s} f\|_{L^{r}}
			\lesssim
			\|\langle \nabla \rangle^{s} f\|_{L^{r}},
		\end{equation}
		and
		\begin{equation}\label{homogeneous dominate}
			\|\langle \nabla \rangle^{s} f\|_{L^{r}}
			\lesssim
			\|f\|_{L^{r}}
			+
			\|\,|\nabla|^{s} f\|_{L^{r}} .
		\end{equation}
	\end{remark}
	We next define \textbf{mixed Sobolev-type space} $L_{x}^{r}W_{y}^{s,\tilde{r}}:=L_{x}^{r}W_{y}^{s,\tilde{r}}(\T^{d-k} \times \T^k)$ by
	$$\|f\|_{ L_{x}^{r}W_{y}^{s,\tilde{r}}}:= \|\langle \nabla_{y} \rangle^{s} f\|_{L_{x}^rL_{y}^{\tilde{r}}}. $$
	In addition, we also denote $H^{s}:=W^{s,2}$ and $L_{x}^{r}H_{y}^{s}:=L_{x}^{r}W_{y}^{s,2}$ throughout this paper.  
	
	Furthermore, for any interval $I \subset \mathbb{T},$ we will make use of the mixed space-time space $L^q_t(I, L^p_x(\mathbb{T}^{d-k}, H^s_y(\mathbb{T}^{k})))$ which will be briefly denotes by $ L^q_t(I)L^p_xH_y^s$ or simply $L^q_tL^p_xH_y^s.$ These spaces are eqquiped with the norm 
	$$\|F\|_{L^q_tL^p_xH_y^s}:= \left( \int_I \|F(t)\|^q_{L^p_xH^s_y}\, dt \right)^{\frac{1}{q}}.$$
	For suitable functions $f$ and $g$ defined on $\T^{d-k} \times \T^k$ and $x \in \T^{d-k},$ we denote by $f\ast_y g$ the convolution of $f$ and $g$ with respect to $y$-variable defined as 
	$$f\ast_yg (x, y):= \int_{\T^k} f(x, y-y') g(x, y')\, dy'. $$
	In order to restrict  the  frequency variable on $d-$dimensional cube, we introduce
	\begin{eqnarray}\label{rdp}
		S_{d,N}  = \Z^{d} \cap [-N, N]^d.  
	\end{eqnarray}
	Let $\phi: \mathbb{R}^{d-k}\times \R^{k}  \to [0,1]$ be a smooth even function satisfying 
	\begin{equation}\label{phi cutoff}
		\phi(z)= \begin{cases}
			1 &  \text{if} \quad \abs{z} \leq 1, \\
			0  &  \text{if} \quad \abs{z} \geq 2.
		\end{cases}
	\end{equation}
	\begin{remark}
		We shall use the symbol $\eta$ to denote either the indicator function 
		$\mathbf{1}_{S_{d,1}}$ or the smooth cutoff $\phi$; the specific choice 
		will be stated whenever it is relevant, where $S_{d,1}$ is defined in \eqref{rdp}.
	\end{remark}
	
	Let $N \in 2^{\mathbb{N}}, \mathbb{N}=\{0\}\cup \mathbb{Z}_+,$ be a dyadic integer. We define \textbf{Littlewood-Paley projectors} on $\T^d$ by 
	\begin{equation}\label{P less N}
		\widehat{P_{\leq N}f}(\xi)= \eta\left(\frac{\xi}{N}\right) \widehat{f}(\xi), \quad \xi \in \Z^d,
	\end{equation}
	and 
	\begin{align}\label{lpope}
		P_{N}f=P_{\leq N}f-P_{\leq \frac{N}{2}}f,
	\end{align}
	where we set  $P_{\leq \frac{1}{2}}=0.$ Let us define a function $\psi(z)=\eta(z)-\eta(2z)$ and
	\begin{equation}\label{psi N}
		\psi_N(z)=
		\begin{cases}
			\eta(z) & ~\text{for}~ N=1 ,\\
			\psi(N^{-1}z) & ~\text{for}~ N>1.
		\end{cases}
	\end{equation}
	Then using \eqref{lpope}, we get $$\widehat{P_{ N}f}(\xi)=\psi_{N} \widehat{f}(\xi).
	$$
	\subsection{Mixed Besov-type spaces} \label{Bs}
	
	Let $s \in \mathbb{R}$ and $r,\tilde r \in [1,\infty]$. We define the \textbf{mixed Besov-type space}
	\(
	B^s_{(r,\tilde r),\infty}
	= B^s_{(r,\tilde r),\infty}\big(\mathbb{T}^{d-k}\times\mathbb{T}^k\big)
	\)
	by
	\[
	B^s_{(r,\tilde r),\infty}
	= \left\{ f : \mathbb{T}^{d-k}\times\mathbb{T}^k \to \mathbb{C} \;:\;
	\|f\|_{B^s_{(r,\tilde r),\infty}} < \infty \right\}.
	\]
	
	where
	\[
	\|f\|_{B^s_{(r,\tilde{r}), \infty}} = \sup_{N \in 2^\mathbb{N} } N^{s} \|P_N f\|_{L_x^r L_y^{\tilde{r}}}.
	\]
	Note that for $r=\tilde{r}$, the space $B^s_{(r,\tilde{r}), \infty}$ coincides with the usual Besov space $B^s_{r, \infty}= \left\{ f : \|f\|_{B^s_{r, \infty}} < \infty \right\},$
	where
	\[
	\|f\|_{B^s_{r, \infty}} = \sup_{N \in 2^\mathbb{N} } N^{s} \|P_N f\|_{L^r}.
	\] 
See \cite{HansII} for Besov spaces on $\mathbb R^{d}$ and $\mathbb T^{d}$, 
and \cite{WangMei2025Navier} for mixed Besov-type spaces on $\mathbb R^{d}$.

	\subsection{Schatten spaces $\Sp^{\alpha}$} \label{schp}
	We denote the  Hilbert space $\mathcal{H}$ as $$\mathcal{H}= L_{t}^2 L_{x}^2L^{2}_{y}(I \times \T^{d-k} \times \T^k),$$ where $I \subset \mathbb R$  is  a time  interval. Let $\mathcal{B}(\cH)$ denotes the space of all bounded operators on $\cH.$ The \textbf{Schatten space} $\Sp^{\alpha }(\cH), \alpha \in [1, \infty),$ consists of all compact operators $A \in \mathcal{B}(\cH)$ such that $\operatorname{Tr}|A|^{\alpha} < \infty,$ where $|A| = \sqrt{A^* A}$ and $ A^* \text{ is adjoint of }A$.   The  Schatten space norm is  given by 
	$$\|A\|_{\Sp^{\alpha}}=\|A\|_{\Sp^{\alpha}(\cH)} = (\operatorname{Tr}|A|^{\alpha})^{\frac{1}{\alpha}}.$$ 
	For  $\alpha = \infty,$ the Schatten space norm coincides with the operator norm, i.e., 
	$$\|A\|_{\Sp^{\infty}}=\|A\|_{\Sp^{\infty}(\cH)} = \|A\|_{\cH\to \cH}.$$
	An operator belongs to the class \(\mathfrak{S}^{1}(\mathcal{H})\) is known as {\bf Trace class operator}. Also, an operator belongs to   \(\mathfrak{S}^{2}(\mathcal{H})\) is known as  {\bf Hilbert-Schmidt operator}.
	\subsubsection{\textbf{Duality principle in $\Sp^{\alpha}$}}
	\begin{lemma}[Lemma 3 in \cite{frank2017restriction}]\label{PL1}
		Let $q, r, \tilde{r} \geq 1$ and $\alpha \geq 1.$ Let $T$ be a bounded operator from a separable Hilbert space $\cH$ to
		$L_{t}^{2q}L_{x}^{2r}L_{y}^{2\tilde r}(I \times \T^{d-k} \times \T^{k}).$
		Then the following are equivalent.
		\begin{enumerate}
			\item[(i)] There is a constant $C > 0$ such that
			\begin{equation*}
				\left\|WTT^*W\right\|_{\Sp^\alpha (\cH)} \leq C
				\left\|W\right\|^{2}_{L_t^{2q'}L_x^{2r'}L_{y}^{2\tilde r'} (I \times \T^{d-k} \times \T^{k})},
			\end{equation*}
			for all $ W \in L_t^{2q'}L_x^{2r'}L_{y}^{2\tilde r'}(I \times \T^{d-k} \times \T^{k}),$
			where the function $W$ acts as a multiplication operator.
			\item[(ii)] There is a constant $C' > 0$ such that for any ONS $(f_j)_{j \in \mathbb{Z}} \subset L_{x}^{2}L_{y}^{2}(\T^{d-k} \times \T^{k})$ and any sequence $(\lambda_j)_{j \in \mathbb{Z}} \subset \mathbb{C}$,
			\begin{equation*}
				\left\| \sum_{j \in \mathbb{Z}} \lambda_j |Tf_j|^2 \right\|_{L_t^{q}L_x^{r}L_{y}^{\tilde r}(I \times \T^{d-k} \times \T^{k})}
				\leq C' \left( \sum_{j \in \mathbb{Z}} |\lambda_j|^{\alpha'} \right)^{1/\alpha'}.
			\end{equation*}
		\end{enumerate}
	\end{lemma}
	
	\subsubsection{\textbf{Complex interpolation in $\Sp^{\alpha}$}}
	Let $a_0, a_1 \in \mathbb R$ such that $a_0<a_1$ and strip $S= \{ z' \in \mathbb C: a_0 \leq \operatorname{Re}(z') \leq a_1\}.$  Denote Hilbert space $\mathcal{H}= L_{t}^2 L_{x}^2L^{2}_{y}(I \times \T^{d-k} \times \T^k),$ where $I \subset \mathbb R$  is  a time  interval.
	We say a  family of operators  $\{T_{z'}\} \subset \mathcal{B}(\mathcal{H})$  defined in a strip  $S$ is analytic in the sense of Stein if it has the following properties:
	\begin{itemize}
		\item[-] for each $z' \in S,$ $T_{z'}$ is a linear transformation  of simple functions on $I \times \T^{d-k} \times \T^{k}$;
		\item[-] for all simple functions $F, G$ on $I \times \T^{d-k} \times \T^{k},$ the map $z' \mapsto \langle G, T_{z'} F \rangle$ is analytic in the interior of the strip  $ {S}$ and continuous in $S$;
		\item[-] $\sup_{a_0 \leq \lambda \leq a_1} |\langle G, T_{\lambda +is} F \rangle| \leq C(s)$ for some $C(s)$ with at most a (double) exponential growth  in $s$.
	\end{itemize}
\begin{lemma}[Theorem 2.9 in \cite{simon2005trace}]\label{s-interpolation}
Let $\{T_{z'}\} \subset \mathcal{B}(\mathcal{H})$ be an analytic family of operators in the sense of Stein, defined on the strip $S,$ where $H = L^2\bigl(I \times \T^{d-k} \times \T^{k}\bigr).$
If there exist $M_0, M_1, b_0, b_1>0,$
$1 \leq q_0, r_0, \tilde r_0, q_1, r_1, \tilde r_1 \leq \infty$
and $1\leq a_0, a_1 \leq \infty$ such that for all simple functions
$W_1, W_2$ on $I \times \T^{d-k} \times \T^{k}$
and $s \in \mathbb R,$ we have
\[
\|W_1 T_{a_0 +is} W_2 \|_{\Sp^{a_0}(\mathcal{H})}
\leq
M_0 e^{b_0 |s|}
\|W_1\|_{L_t^{q_0}L_x^{r_0}L_y^{\tilde r_0}}
\|W_2\|_{L_t^{q_0}L_x^{r_0}L_y^{\tilde r_0}}
\]
and
\[
\|W_1 T_{a_1 +is} W_2 \|_{\Sp^{a_1}(\mathcal{H})}
\leq
M_1 e^{b_1|s|}
\|W_1\|_{L_t^{q_1}L_x^{r_1}L_y^{\tilde r_1}}
\|W_2\|_{L_t^{q_1}L_x^{r_1}L_y^{\tilde r_1}} .
\]
Then for all $\tau \in [0,1]$ we have
\[
\|W_1 T_{a_{\tau} +is} W_2 \|_{\Sp^{a_\tau}(\mathcal{H})}
\leq
M_0^{1-\tau}M_1^{\tau}
\|W_1\|_{L_t^{q_\tau}L_x^{r_\tau}L_y^{\tilde r_\tau}}
\|W_2\|_{L_t^{q_\tau}L_x^{r_\tau}L_y^{\tilde r_\tau}},
\]
where $a_\tau, r_\tau, q_\tau,$ and $\tilde r_\tau$ are defined by
\[
a_\tau = (1-\tau)a_0 + \tau a_1, \quad
\frac{1}{r_\tau}= \frac{1-\tau}{r_0}+ \frac{\tau}{r_1}, \quad
\frac{1}{q_\tau} = \frac{1-\tau}{q_0}+ \frac{\tau}{q_1}, \quad
\frac{1}{\tilde r_\tau} = \frac{1-\tau}{\tilde r_0}+ \frac{\tau}{\tilde r_1}.
\]
\end{lemma}
\begin{remark}
   Lemma \ref{PL1} can be proved by adapting the argument of Lemma 3 in \cite{frank2017restriction}, while Lemma \ref{s-interpolation} follows, with appropriate modifications, from the method of proof of Theorem 2.9 in \cite{simon2005trace}.
\end{remark}
	\begin{lemma}[Hardy--Littlewood--Sobolev inequality \cite{FrankJEMS2014}]\label{PL3}
		Let $q,r > 1$ and $\lambda \in [0,1).$ Let $\frac{1}{q}+\frac{1}{r}+\lambda=2$ and $\frac{1}{q}+\frac{1}{r} \geq 1.$ Let $f \in L^{q}(\mathbb{R})$ and $ g \in L^{r}(\mathbb{R}).$ Then  
		$$
		\left|\int_{\mathbb{R}}\int_{\mathbb{R}} \frac{f(x) g(y)}{|x-y|^{\lambda}} \,dx \, dy\right|
		\leq C \|f\|_{L^{q}(\mathbb{R})}\|g\|_{L^{r}(\mathbb{R})}.
		$$
	\end{lemma}

	\begin{lemma}[Hilbert-Schmidt operators \cite{simon2005trace}] \label{PL6}
		Let $H = L^2(M, d\mu)$, for some separable measure space $M$ (i.e., one with $H$ separable). If $A \in \Sp^2(H)$, then there exists a unique function $K \in L^2(M \times M, d\mu \otimes d\mu)$ with
		\begin{equation}\label{P4}
			(A\varphi)(x) = \int_{M} K(x,y) \varphi(y) d\mu(y).
		\end{equation}
		Conversely, any $K \in L^2(M \times M)$ defines an operator $A$ by \eqref{P4} which is in $\Sp^2(H)$ and
		\[
		\|A\|_{\Sp^{2}(H)} = \|K\|_{L^{2}(M \times M, d\mu \otimes d\mu)}.
		\]
	\end{lemma}
	\begin{proposition}[Dispersive estimate \cite{vega1992restriction}]\label{l8}
		We have
		\begin{equation}\label{kernel}
			\left| \sum_{\xi=-N}^{N} e^{2 \pi i (x \cdot \xi +t |\xi|^{2})} \right| \leq C |t|^{-\frac{1}{2}},
		\end{equation}
		for any $(t,x) \in  [-N^{-1},N^{-1}] \times \mathbb{T},$ where $N>1$ and $N \in \mathbb{Z}.$ 
	\end{proposition}
	
	\section{Harmonic analysis tools}\label{sec22}
	In this section we prove Littlewood-Paley theorem on mixed Lebesgue space as well as for densities of operators. These results play a crucial role in the proof of our refined Strichartz estimates and in the well-posedness analysis of the Hartree equation. Moreover, the reader may observe that these results are of independent interest. We start this section with the Littlewood-Paley theorem on mixed Lebesgue space. 
	\subsection{Littlewood-Paley theorem on mixed Lebesgue space} The Littlewood-Paley theorem  on mixed Lebesgue spaces $L_x^r L_y^{\widetilde{r}}\left(\mathbb{T}^{d-k} \times \mathbb{T}^k\right)$; for $1<r, \widetilde{r}<\infty$,
	is given by 	\begin{align}\label{norm}
		\|f\|_{L_x^r L_y^{\tilde{r}}} \sim\left\| \left(\sum_{N \in 2^\mathbb{N}} \left|P_N f\right|^2\right)^{1 / 2}\right\|_{L_x^r L_y^{\tilde{r}}} .
	\end{align}
	The method to prove the preceding norm inequality heavily depends   on  the uses of randomization by Rademacher signs together with the classical transference principle. In the transference principle, one approximates periodic functions by compactly supported Euclidean functions through  Gaussian cutoffs, and then invokes a limiting relation that links the operator on the torus $\T^d$ to the corresponding operator on $\R^d.$  Before that, we first recall the H\"ormander–Mikhlin multiplier theorem in the setting of mixed-norm spaces, as it provides a crucial tool for proving the Littlewood–Paley theorem.
	\begin{lemma}[H\"ormander–Mikhlin multiplier theorem (Theorem~3.2.1 in \cite{ward2010mixedLebesgue})]\label{Hormander-mikhlm}
		Let \(\alpha = (\alpha_{1}, \ldots, \alpha_{d})\) be a multi-index, meaning that each 
		\(\alpha_{j}\) is a non-negative integer and its order is defined by
		\[
		|\alpha| = \alpha_{1} + \cdots + \alpha_{d}.
		\]
		Assume that
		\[
		|\alpha| \leq \left\lfloor \frac{d}{2} \right\rfloor + 1,
		\]
		where \(\left\lfloor \frac{d}{2} \right\rfloor\) denotes the integer part of \(d/2\).
		Suppose that the operator $T_\mathfrak{M}$ on $L_{x}^{r}L_{y}^{\Tilde{r}}(\mathbb{R}^{d-k} \times \mathbb{R}^{k})$ is given by
		\[
		\widehat{T_{\mathfrak{M}} f}(\xi_{1},\xi_{2}) = \mathfrak{M}(\xi_1, \xi_2)\, \widehat{f}(\xi_1, \xi_2)
		\]
		for all Schwartz function $f$ and $(\xi_1, \xi_2) \in \R^{d-k} \times\R^k$, and that $\mathfrak{M}$ satisfies the condition
		\begin{equation}\label{Hormander condition}
			|\partial^{\alpha} \mathfrak{M}(\xi_1, \xi_2)| \leq \frac{C}{|(\xi_1, \xi_2)|^{|\alpha|}}.
		\end{equation}
		Then 
		$
		T_{\mathfrak{M}} : L_{x}^{r}L_{y}^{\Tilde{r}}(\mathbb{R}^{d-k} \times \mathbb{R}^{k}) \to L_{x}^{r}L_{y}^{\Tilde{r}}(\mathbb{R}^{d-k} \times \mathbb{R}^{k})
		$ is bounded operator for $ 1 < r, \tilde{r} < \infty.$
	\end{lemma}
	\begin{lemma}[Equation 3.6.10 in \cite{grafakos2008classical}]\label{periodic identity}
		For all continuous $1$-periodic functions $g$ on $\mathbb{R}^d$, we have
		\[
		\lim_{\varepsilon \to 0} \varepsilon^{\frac{d}{2}} \int_{\mathbb{R}^d} g(x) e^{-\varepsilon \pi |x|^2}\, dx  = \int_{\mathbb{T}^d} g(x)\, dx.
		\]
	\end{lemma}
	\begin{definition}
		A bounded function $b$ on $\R^d$ is called \textbf{regulated} at a point $\xi_{0} \in \R^{d}$ if $$\lim_{\varepsilon \to 0} \frac{1}{\varepsilon^{d}} \int_{|\xi| \leq \varepsilon} (b(\xi_{0}-\xi)-b(\xi_{0})) \, d\xi=0.$$ 
	\end{definition}
	To establish Proposition~\ref{transference}, we adapt the techniques used in the proof of the transference principle for the diagonal case $r = \tilde{r}$ as presented in \cite{grafakos2008classical}, making the necessary adjustments to accommodate our specific setting.  
	\begin{proposition}[transference principle]\label{transference} Let $1<r,\tilde{r}< \infty.$
		Let $T^{\mathbb{R}}$ be a bounded operator on $L_{x}^{r}L_{y}^{\Tilde{r}}(\mathbb{R}^{d-k} \times \mathbb{R}^{k})$ whose multiplier is $b(\xi,\eta)$ and let $T^{\mathbb{T}}$ be an operator on $L_{x}^{r}L_{y}^{\Tilde{r}}(\mathbb{T}^{d-k} \times \mathbb{T}^{k})$ whose multiplier is the sequence $\{b(m,n)\}_{(m,n) \in \Z^{d}}.$ Assume that $b(\xi,\eta)$ is regulated at every point $(m,n) \in \Z^d.$ Then $T^{\mathbb{T}}$ is bounded operator on $L_{x}^{r}L_{y}^{\Tilde{r}}(\mathbb{T}^{d-k} \times \mathbb{T}^{k})$ and $$\|T^{\mathbb{T}}\|_{L_{x}^{r}L_{y}^{\Tilde{r}} \to L_{x}^{r}L_{y}^{\Tilde{r}}} \leq \|T^{\mathbb{R}}\|_{L_{x}^{r}L_{y}^{\Tilde{r}} \to L_{x}^{r}L_{y}^{\Tilde{r}}}.$$ 
	\end{proposition}
	\begin{remark}[proof strategy] The following steps are in order    
		\begin{itemize}
			\item[-] We will prove the inequality for trigonometric polynomials because they are dense.
			
			\item[-] We use the duality argument to get the operator norm.
			
			\item[-] Extend periodic functions to Euclidean space using Gaussians mollifiers so that the boundedness of Euclidean multiplier $T^{\mathbb{R}}$ can be applied. We need regulated condition on the multiplier, so that we can apply \cite[Lemma 3.6.8]{grafakos2008classical}.
			
			\item[-] Use a limiting identity that relates the torus operator 
			$T^{\mathbb{T}}$ to the Euclidean operator $T^{\mathbb{R}}$.
			
			\item[-] And then apply H\"older's inequality and the boundedness of $T^{\mathbb{R}}$ to control the integral.
			
			\item[-] Use Lemma \ref{periodic identity} to pass to the limit and remove the Gaussian weights.
		\end{itemize}
	\end{remark}

	\begin{proof}[{\bf Proof of Proposition \ref{transference}} ]
		Let $P \in L_{x}^{r}L_{y}^{\tilde{r}}(\T^{d-k} \times \T^{k})$  and $Q \in L_{x}^{r'}L_{y}^{\tilde{r}'}(\T^{d-k} \times \T^{k})$ be two trigonometric polynomials on $\T^{d-k} \times \T^{k}.$ Then, by duality and using (\cite[Lemma 3.6.8]{grafakos2008classical}) (under the regulated assumption), we have
		\begin{align*}
			\left| 
			\int_{\T^{d-k} \times \T^{k}} 
			(T^{\T}P)(x,y)\, \overline{Q(x,y)} \, dy\, dx 
			\right|
			&=
			\left|
			\lim_{\varepsilon_{1} \to 0} 
			\lim_{\varepsilon_{2} \to 0} 
			\varepsilon_{1}^{\frac{d-k}{2}}
			\varepsilon_{2}^{\frac{k}{2}}
			\int_{\R^{d-k} \times \R^{k}}
			T^{\R}\!\big( 
			P\, \mathbf{L}_{\varepsilon_{1}/r}^{x} 
			\mathbf{L}_{\varepsilon_{2}/\tilde{r}}^{y}
			\big)(x,y)\,
			\right. \nonumber\\
			&\qquad\left.
			\times\;
			\overline{
				Q(x,y)\,
				\mathbf{L}_{\varepsilon_{1}/r'}^{x}(x)\,
				\mathbf{L}_{\varepsilon_{2}/\tilde{r}'}^{y}(y)
			}
			\, dy\, dx
			\right|,
		\end{align*}
		where for all $y \in \R^{k},$ $\mathbf{L}_{\varepsilon}^{x}(x,y)=e^{-\pi \varepsilon |x|^2}$ and  for all $x \in \R^{d-k},$ $\mathbf{L}_{\varepsilon}^{y}(x,y)=e^{-\pi \varepsilon |y|^2}.$ Then by  H\"older's inequality, we get  
		\begin{align*}
			&  \left| \int_{\T^{d-k} \times \T^k} (T^{\T}P)(x,y) \, \overline{Q(x,y)} \, dy \, dx \right| \\& \leq \limsup_{\varepsilon_1 \to 0} \limsup_{\varepsilon_2 \to 0} {\varepsilon}_{1}^{\frac{d-k}{2}} {\varepsilon}_{2}^{\frac{k}{2}}
			\| T^{\R}(P \mathbf{L}_{\varepsilon_1 / r}^{x} \mathbf{L}_{\varepsilon_2 / \tilde{r}}^{y}) \|_{L_x^{r} L_y^{\tilde{r}}} 
			\| Q \mathbf{L}_{\varepsilon_1 / r'}^{x} \mathbf{L}_{\varepsilon_2 / \tilde{r}'}^{y} \|_{L_x^{r'} L_y^{\tilde{r}'}} \\
			& \leq \| T^{\R} \|_{L_x^{r} L_y^{\tilde{r}} \to L_x^{r} L_y^{\tilde{r}}} 
			\limsup_{\varepsilon_1 \to 0} \limsup_{\varepsilon_2 \to 0} {\varepsilon}_{1}^{\frac{d-k}{2}} {\varepsilon}_{2}^{\frac{k}{2}}
			\| P \mathbf{L}_{\varepsilon_1 / r}^{x} \mathbf{L}_{\varepsilon_2 / \tilde{r}}^{y} \|_{L_x^{r} L_y^{\tilde{r}}} 
			\| Q \mathbf{L}_{\varepsilon_1 / r'}^{x} \mathbf{L}_{\varepsilon_2 / \tilde{r}'}^{y} \|_{L_x^{r'} L_y^{\tilde{r}'}} \\
			&= \| T^{\R} \|_{L_x^{r} L_y^{\tilde{r}} \to L_x^{r} L_y^{\tilde{r}}} \| P \|_{L_x^{r} L_y^{\tilde{r}}(\T^{d-k} \times \T^{k})} \| Q \|_{L_x^{r'} L_y^{\tilde{r}'}(\T^{d-k} \times \T^{k})},
		\end{align*}
		where in the last equality, we  used Lemma \ref{periodic identity}. Now taking   supremum over all trigonometric polynomials $Q \in L_{x}^{r'}L_{y}^{\tilde{r}'}(\T^{d-k} \times \T^{k})$ with norm at most $1$, we obtain that $T^{\T}$ is a bounded operator on $L_{x}^{r}L_{y}^{\Tilde{r}}(\mathbb{T}^{d-k} \times \mathbb{T}^{k})$ with $$\|T^{\mathbb{T}}\|_{L_{x}^{r}L_{y}^{\Tilde{r}} \to L_{x}^{r}L_{y}^{\Tilde{r}}} \leq \|T^{\mathbb{R}}\|_{L_{x}^{r}L_{y}^{\Tilde{r}} \to L_{x}^{r}L_{y}^{\Tilde{r}}}.$$
		
	\end{proof}
	
	\begin{lemma}[Jensen’s inequality]\label{Jensen inequality}
		Let $(\Omega, \mathcal{A}, \mu)$ be a probability space. Let
		$f : \Omega \to \mathbb{R}$ be a $\mu$-measurable function, and let
		$\mathbf{G} : \mathbb{R} \to \mathbb{R}$ be a convex function. Then
		\[
		\mathbf{G}\!\left( \int_{\Omega} f \, d\mu \right)
		\le
		\int_{\Omega} \mathbf{G} \circ f \, d\mu .
		\]
		
	\end{lemma}
	\begin{lemma}[Khintchine inequality for Rademacher signs (see Appendix C of \cite{grafakos2008classical})]\label{Khintchine inequality}
		Let $\varepsilon=(\varepsilon_{j})_{j \in \mathbb{Z}}$ be a sequence of independent Rademacher random variables satisfying
		\[
		\mathbb{P}(\varepsilon_{j}=1)=\mathbb{P}(\varepsilon_{j}=-1)=\tfrac12 .
		\]
		Let $(a_{j_1,\cdots,j_n})_{(j_1,\cdots,j_n) \in \mathbb{Z}^n}$ be any square-summable sequence of complex numbers indexed by $n$ variables. Define
		\[
		F_{n}(t_1,\cdots,t_n)
		=
		\sum_{j_{1} \in \mathbb{Z}} \cdots \sum_{j_{n} \in \mathbb{Z}}
		a_{j_1,\cdots,j_n}\,
		\varepsilon_{j_1}(t_1)\varepsilon_{j_2}(t_2)\cdots \varepsilon_{j_n}(t_n),
		\]
		which is a function on $[0,1]^n$.
		
		Then, for every $q \in [1,\infty)$, there exist constants $A_q,B_q>0$, depending only on $q$, such that
		\[
		A_{q}\left( \sum_{j_{1} \in \mathbb{Z}} \cdots \sum_{j_{n} \in \mathbb{Z}}
		|a_{j_1,\cdots,j_n}|^2 \right)^{1/2}
		\;\le\;
		\left\| F_{n} \right\|_{L_{t_1,\cdots,t_n}^{q}([0,1]^n)}
		\;\le\;
		B_{q}\left( \sum_{j_{1} \in \mathbb{Z}} \cdots \sum_{j_{n} \in \mathbb{Z}}
		|a_{j_1,\cdots,j_n}|^2 \right)^{1/2}.
		\]
	\end{lemma} 
	Now we are in a position to prove the Littlewood-Paley theorem on mixed Lebesgue space over the torus. 
	\begin{proposition}[Littlewood-Paley theorem]\label{mixed LP}
		Let $\{\psi_{N}\}_{N \in 2^{\mathbb{N}}}$ be the family of functions constructed in \eqref{psi N}  with smooth cutoff function $\phi$ and  $P_{N}$'s are the  Littlewood-Paley projectors on torus defined   in \eqref{lpope}. Then, for $1<r, \widetilde{r}<\infty$, we  have 
		\begin{align}\label{norm in terms of j}
			\left\| \left(\sum_{N \in 2^\mathbb{N}} \left|P_N f\right|^2\right)^{1 / 2}\right\|_{L_x^r L_y^{\tilde{r}}(\mathbb{T}^{d-k} \times \mathbb{T}^{k})}  \sim \|f\|_{L_x^r L_y^{\tilde{r}}(\mathbb{T}^{d-k} \times \mathbb{T}^{k})} .
		\end{align}
	\end{proposition}
	\begin{remark}[proof strategy] The following steps are in order    
		\begin{itemize}
			\item[-] We introduce Rademacher variables, it allows us to replace the square function by an average of linear combinations.
			
			\item[-] We define the randomized multiplier operator on torus. This operator is simply the sum of the Littlewood--Paley pieces, each multiplied by a random sign.
			
			\item[-] Because of how the Littlewood--Paley cutoff functions are constructed, the multiplier automatically satisfies the usual Euclidean Littlewood--Paley conditions.
			
			\item[-] Apply Lemma \ref{Hormander-mikhlm} (Euclidean multiplier theorem) together with Proposition \ref{transference} (transference principle). This tells us that the same bounds that hold in the Euclidean setting also hold on the torus.
			
			\item[-] Importantly, the operator norm we obtain does not depend on the particular choice of Rademacher variables.  
		\end{itemize}
	\end{remark}
	\begin{proof} We will prove this result in the case $1 < \tilde{r} \leq r < \infty$.
		The other cases follow by a similar argument.
		Let us take $\varepsilon=(\varepsilon_{N})_{N}$ be independent Rademacher random signs with probability $P(\varepsilon_{N}=1)=P(\varepsilon_{N}=-1)=\frac{1}{2}.$ Using Lemma \ref{Khintchine inequality} with $q=\tilde{r},$ $n=1$ and Minkowski's integral inequality, we get
		\begin{align}\label{khintchine application}
			\left\| \left(\sum_{N \in 2^\mathbb{N}} \left|P_N f\right|^2\right)^{1 / 2}\right\|_{L_x^r L_y^{\tilde{r}}} &\lesssim_{\tilde{r}} \Big\|\sum_{N \in 2^\mathbb{N}} \varepsilon_{N}(\cdot) P_{N}f \Big\|_{L_x^r L_y^{\tilde{r}}L_{t}^{\tilde{r}}(\T^{d-k} \times \T^{k} \times [0,1])} \nonumber \\
			&\lesssim_{\tilde{r}} \Big\|\sum_{N \in 2^\mathbb{N}} \varepsilon_{N}(\cdot) P_{N}f \Big\|_{L_{t}^{\tilde{r}}L_x^r L_y^{\tilde{r}}([0,1] \times \T^{d-k} \times \T^{k})}.
		\end{align}
		Let $T^{\T}_{\varepsilon}:L_{x}^{r}L_{y}^{\Tilde{r}} \to L_{x}^{r}L_{y}^{\Tilde{r}}$ be defined by 
		\begin{equation*}
			T^{\T}_{\varepsilon}f=\sum_{N \in 2^\mathbb{N}} \varepsilon_{N}(\cdot) P_{N}f.
		\end{equation*}
		Then
		\begin{equation*}
			\widehat{T^{\T}_{\varepsilon}f}(\xi_{1},\xi_{2})=\mathfrak{M}^{\T}_{\varepsilon}(\xi_{1},\xi_{2}) \widehat{f}(\xi_{1},\xi_{2}), \quad (\xi_{1},\xi_{2}) \in \Z^{d-k} \times \Z^{k},
		\end{equation*}
		where \[
		\mathfrak{M}^{\T}_{\varepsilon}(\xi_{1},\xi_{2}) 
		= \sum_{N \in 2^\mathbb{N}} \varepsilon_{N}(\cdot)\, \psi_{N}(\xi_{1},\xi_{2}), 
		\quad (\xi_{1},\xi_{2}) \in \Z^{d-k} \times \Z^{k}.
		\]
		This can be viewed as the restriction of the continuous function
		\[
		\mathfrak{M}^{\R}_{\varepsilon}(\xi_{1},\xi_{2}) 
		= \sum_{N \in 2^\mathbb{N}} \varepsilon_{N}(\cdot)\, \psi_{N}(\xi_{1},\xi_{2}), 
		\quad (\xi_{1},\xi_{2}) \in \R^{d-k} \times \R^{k},
		\]
		where $\psi_{1}$ is supported in the ball of radius $2$ and each $\psi_{N}$ is supported in the region $\{(\xi_{1},\xi_{2}) : |(\xi_{1},\xi_{2})| \approx N\}$ with  $\sum_{N \in 2^{\mathbb{N}}} \psi_{N} = 1.$ Then $\mathfrak{M}^{\R}$ satisfies  the Mikhlin multiplier condition \eqref{Hormander condition} uniformly in the choice of  \( \varepsilon \) (see \cite[p. 1]{schlag2007littlewood}), and by Lemma \ref{Hormander-mikhlm} and Proposition \ref{transference}, \( T^{\T}_{\varepsilon} \) acts as a multiplier on \( L_{x}^{r}L_{y}^{\Tilde{r}} \) and  its operator norm is independent of the choice of \( \varepsilon \), i.e., 
		$$\|T^{\T}_{\varepsilon}f \|_{L_{x}^{r}L_{y}^{\Tilde{r}}} \lesssim \|f\|_{L_{x}^{r}L_{y}^{\Tilde{r}}}.$$
		Then  from inequality  \eqref{khintchine application}, we get 
		\begin{align*}
			\left\| \left(\sum_{N \in 2^\mathbb{N}} \left|P_N f\right|^2\right)^{1 / 2}\right\|_{L_x^r L_y^{\tilde{r}}}  \lesssim \|f\|_{L_x^r L_y^{\tilde{r}}} .
		\end{align*}
		The reverse inequality follows from a standard duality argument (see \cite[Theorem~2.3.1]{ward2010mixedLebesgue}).
	\end{proof}
	\subsection{Vector-valued  Bernstein inequality}
	Let us denote the space $$L^r(\mathbb{C}^n,\, e^{-\pi|z|^2}\,dz)
	:= \left\{ F:\mathbb{C}^n\to\mathbb{C} :
	\int_{\mathbb{C}^n} |F(z)|^r e^{-\pi|z|^2}\,dz < \infty \right\},$$
	equipped with the  norm
	\[
	\|F\|_{L^r(\mathbb{C}^n,\, e^{-\pi|z|^2}\,dz)}
	=
	\left(
	\int_{\mathbb{C}^n} |F(z)|^r e^{-\pi|z|^2}\,dz
	\right)^{1/r}.
	\]
	\begin{lemma}[see \cite{grafakos2008classical}]\label{lem:5.5.2}
		For any $0< r < \infty$, let 
		\begin{equation}\label{5.5.7}
			B_r
			=
			\left(
			\frac{\Gamma\!\left(\frac{r}{2}+1\right)}{\pi^{\frac{r}{2}}}
			\right)^{\frac{1}{r}} .
		\end{equation}
		Then for any $w_1,w_2,\ldots,w_n \in \mathbb{C}$, we have
		\begin{equation}\label{5.5.9}
			\left(
			\int_{\mathbb{C}^n}
			\big|
			w_1 z_1 + \cdots + w_n z_n
			\big|^r
			e^{-\pi |z|^2}
			\,dz
			\right)^{\frac{1}{r}}
			=
			B_r
			\left(
			|w_1|^2 + \cdots + |w_n|^2
			\right)^{\frac12},
		\end{equation}
		where $dz = dz_1 \cdots dz_n = dx_1 dy_1 \cdots dx_n dy_n$
		with  $z_j = x_j + i y_j$.
	\end{lemma}
	The following vector-valued inequalities in mixed-norm spaces play a key role in establishing the vector-valued extension of Bernstein’s inequality, which in turn is essential for proving the well-posedness of the Hartree equation \eqref{Hartree} for infinitely many fermions. To establish Theorem~\ref{mixed vector valued}, we adapt the techniques used in the proof of the $\ell^2$-valued extension theorem for the diagonal case $r = \tilde{r}$ as presented in \cite{grafakos2008classical}, making the necessary adjustments to accommodate our specific setting.
	\begin{theorem}[$\ell^2$-valued extension]\label{mixed vector valued}
		Let $1 \le \tilde r \le r < \infty$ and let $T$ be a bounded linear operator
		from $L^r L^{\tilde r}(\mathbb{T}^{d-k}\times\mathbb{T}^k)$ to
		$L^r L^{\tilde r}(\mathbb{T}^{d-k}\times\mathbb{T}^k)$ with norm $\|T\|$.
		Then $T$ has an $\ell^2$-valued extension, i.e.,  for all complex-valued functions
		$f_j \in L^r L^{\tilde r}(\mathbb{T}^{d-k}\times\mathbb{T}^k)$, we have
		\begin{align}\label{Vector Valued}
			\left\|\left(\sum_j |T(f_j)|^2\right)^{\frac12}\right\|_{L^r L^{\tilde r}(\mathbb{T}^{d-k}\times\mathbb{T}^k)}
			\le C_{r,\tilde r}\|T\|
			\left\|\left(\sum_j |f_j|^2\right)^{\frac12}\right\|_{L^r L^{\tilde r}(\mathbb{T}^{d-k}\times\mathbb{T}^k)},
		\end{align}
		for some constant $C_{r,\tilde r}$  that depends only on $r$ and $\tilde{r}$. 
	\end{theorem}
	\begin{proof}[{\bf Proof of Theorem \ref{mixed vector valued}} ]
		Let $B_{\tilde r}$ be the constant as in \eqref{5.5.7}. Using the  
		boundedness of $T$, the identity \eqref{5.5.9} successively, Minkowski's integral  inequality, Jensen's inequality (see Lemma \ref{Jensen inequality}) for integral
		and H\"older’s inequality with respect to the measure $e^{-\pi |z|^2}$, we get
		\begin{align*}
			&\left\|\left(\sum_{j=1}^n |T(f_j)|^2\right)^{\frac12}\right\|^r_{L^r L^{\tilde r}
				(\mathbb{T}^{d-k}\times\mathbb{T}^k)}
			=
			\int_{\mathbb{T}^{d-k}}
			\left(
			\int_{\mathbb{T}^k}
			\left(\sum_{j=1}^n |T(f_j)|^2\right)^{\frac{\tilde r}{2}}
			\,dy
			\right)^{\frac{r}{\tilde r}}
			dx
			\\
			&=B_{\tilde r}^{- r}
			\int_{\mathbb{T}^{d-k}}
			\left(
			\int_{\mathbb{T}^k}
			\left(
			\int_{\mathbb{C}^n}
			\Big|\sum_{j=1}^n z_j T(f_j)\Big|^{\tilde r}
			e^{-\pi |z|^2}\,dz
			\right)
			dy
			\right)^{\frac{r}{\tilde r}}
			dx
			\\
			&=
			B_{\tilde r}^{-r}
			\int_{\mathbb{T}^{d-k}}
			\left(
			\int_{\mathbb{C}^n}
			\left(
			\int_{\mathbb{T}^k}
			\Big|\sum_{j=1}^n T(z_j f_j)\Big|^{\tilde{r}}
			\,dy
			\right)
			e^{-\pi |z|^2}\,dz
			\right)^{\frac{r}{\tilde r}}
			\,dx\\
			&=
			B_{\tilde r}^{-r}
			\left\|
			\left\|
			\int_{\mathbb{T}^k}
			\Big|\sum_{j=1}^n T(z_j f_j)\Big|^{\tilde r}
			\,dy
			\right\|_{L^{1}
				(\mathbb{C}^n,\, e^{-\pi|z|^2}\,dz)}
			\right\|_{L^{\frac{r}{\tilde r}}_x}^{\frac{r}{\tilde{r}}} \\
			&\le
			B_{\tilde r}^{-r}
			\left\|
			\left\|
			\int_{\mathbb{T}^k}
			\Big|\sum_{j=1}^n T(z_j f_j)\Big|^{\tilde r}
			\,dy
			\right\|_{L^{\frac{r}{\tilde r}}_x}
			\right\|_{L^{1}
				(\mathbb{C}^n,\, e^{-\pi|z|^2}\,dz)}^{\frac{r}{\tilde{r}}}
			\\
			&=
			B_{\tilde r}^{-r} \left(
			\int_{\mathbb{C}^n}
			\left\|
			T\!\big(\sum_{j=1}^n z_j f_j\big)
			\right\|_{L_x^r L_y^{\tilde r}}^{\tilde{r}}
			e^{-\pi |z|^2}\,dz \right)^{\frac{r}{\tilde{r}}}
			\\
			&\le
			B_{\tilde r}^{-r}
			\|T\|^r
			\left(\int_{\mathbb{C}^n}
			\Big\|
			\sum_{j=1}^n z_j f_j
			\Big\|_{L_x^r L_y^{\tilde r}}^{\tilde{r}}~
			e^{-\pi |z|^2}\,dz \right)^{\frac{r}{\tilde{r}}}
			\\
			&\leq 
			B_{\tilde r}^{-r}\|T\|^r
			\int_{\mathbb{T}^{d-k}}
			\Big\|
			\Big\|
			\sum_{j=1}^n z_j f_j
			\Big\|_{L_y^{\tilde r}}
			\Big\|_{L^r(\mathbb{C}^n,\, e^{-\pi|z|^2}\,dz)}^r
			\,dx
			\\
			&\le
			B_{\tilde r}^{-r}\|T\|^r
			\int_{\mathbb{T}^{d-k}}
			\Big\|
			\Big\|
			\sum_{j=1}^n z_j f_j
			\Big\|_{L^r(\mathbb{C}^n,\, e^{-\pi|z|^2}\,dz)}
			\Big\|_{L_y^{\tilde r}}^r
			\,dx
			\\
			&=
			B_{\tilde r}^{-r} B_{r}^{-r}\|T\|^r
			\int_{\mathbb{T}^{d-k}}
			\Big\|
			\left(
			\sum_{j=1}^n |f_j|^2
			\right)^{\frac12} 
			\Big\|_{L_y^{\tilde r}}^r
			\,dx=
			B_{\tilde r}^{-r} B_{r}^{-r} \|T\|^r
			\left\|
			\left(
			\sum_{j=1}^n |f_j|^2
			\right)^{\frac12}
			\right\|_{L_x^r L_y^{\tilde r}}^r .
		\end{align*}
		Finally, letting $n \to \infty$, we obtain the desired estimate for $\tilde r \le r$. 
	\end{proof}

	For $\rho \in \mathbb{R}$, we define an operator $T$ on 
	$L_x^{r}L_y^{\tilde r}(\mathbb T^{d-k}\times\mathbb T^{k})$ by
	\[
	Tg = N^{-\rho} P_{N}\,\langle \nabla \rangle^{\rho} g .
	\]
	
	By applying the Hörmander–Mikhlin Fourier multiplier theorem (Lemma \ref{Hormander-mikhlm}) together with the transference principle (Proposition \ref{transference}), it follows that 
	$T$ acts as a bounded Fourier multiplier on $L^r_xL^{\tilde{r}}_y$.  Therefore $$\|Tg\|_{L_x^{r}L_y^{\tilde r}(\mathbb T^{d-k}\times\mathbb T^{k})}
	\lesssim
	N^\rho
	\|g\|_{L_x^{r}L_y^{\tilde r}(\mathbb T^{d-k}\times\mathbb T^{k})}.$$  
	Now applying the vector-valued inequality \eqref{Vector Valued} from Theorem \ref{mixed vector valued} to the operator  $T$, we immediately obtain the following vector-valued version of Bernstein’s inequality.
	\begin{corollary}\label{vvr}
		For any $\rho \in \mathbb{R},$ $1< \tilde{r} \leq r < \infty,$ and any $(g_{j})_{j} \in L_{x}^{r}L_{y}^{\tilde{r}}(\ell^2)$ (not necessarily orthonormal), 
		$$\bigg\|
		\Big( \sum_j |P_N \langle \nabla\rangle^\rho g_j|^2 \Big)^{1/2}
		\bigg\|_{L_x^{r}L_y^{\tilde r}}
		\lesssim 
		N^{\rho}
		\bigg\|
		\Big( \sum_j |P_N g_j|^2 \Big)^{1/2}
		\bigg\|_{L_x^{r}L_y^{\tilde r}}.$$
	\end{corollary}
	
	\subsection{Littlewood-Paley decomposition of operator densities}\label{vector valued LP decomposition}
	This subsection is devoted to a generalization of \eqref{norm in terms of j} to densities of operators.
	When $\gamma \ge 0$ is a finite-rank operator on $L^2(\mathbb{T}^d)$, its density
	is defined as
	\[
	\rho_\gamma(x) := \gamma(x,x), \qquad \forall x \in \mathbb{T}^d,
	\]
	where $\gamma(\cdot,\cdot)$ denotes the integral kernel of $\gamma$.
	We prove that for any $\tfrac12 < \tilde{r},r < \infty$, there exists a constant
	$C>0$ such that for any finite-rank $\gamma \ge 0$ with
	$\rho_\gamma \in L_x^rL_y^{\tilde{r}}(\mathbb{T}^{d-k} \times \mathbb{T}^{k})$, we have
	\begin{equation}\label{eq:operator_density_LP}
		\frac{1}{C}
		\|\rho_\gamma\|_{L_x^rL_y^{\tilde{r}}(\mathbb{T}^{d-k} \times \mathbb{T}^{k})}
		\;\le\;
		\left\|
		\sum_{N \in 2^\mathbb{N}} \rho_{P_N \gamma P_N}
		\right\|_{L_x^rL_y^{\tilde{r}}(\mathbb{T}^{d-k} \times \mathbb{T}^{k})}
		\;\le\;
		C
		\|\rho_\gamma\|_{L_x^rL_y^{\tilde{r}}(\mathbb{T}^{d-k} \times \mathbb{T}^{k})}.
	\end{equation}
	
	When $\gamma$ is a rank-one operator, this last inequality is equivalent to
	the usual Littlewood--Paley estimates \eqref{norm in terms of j}.
	Indeed, if $u$ with $\|u\|_{L^2}=1$ belongs to the range of $\gamma$, then
	$\rho_\gamma = |u|^2$.
	The constant $C$ in \eqref{eq:operator_density_LP} is independent of the rank of $\gamma$, and in
	particular the inequality remains valid when $\gamma$ has infinite rank, as
	soon as one of the terms of the inequality is well-defined.
	Theorem \ref{vector LP} implies the Littlewood-Paley decomposition \eqref{eq:operator_density_LP} for densities using the spectral decomposition of $\gamma.$ For more details see \cite{sabin2016littlewood}.
	\begin{theorem}[vector-valued Littlewood-Paley theorem]\label{vector LP}
		Let $\{\psi_{N}\}_{N \in 2^{\mathbb{N}}}$ be the family of functions constructed in
		\eqref{psi N} with a smooth cutoff function $\phi$, and let $P_{N}$ denote the
		Littlewood--Paley projectors on the torus defined in \eqref{lpope}.
		Let $(\lambda_k)_{k=1}^{m} \subset (0,\infty)$ be a finite sequence of coefficients,
		and let $(u_k)_{k=1}^{m}$ be a finite sequence in
		$L_{x}^{2r} L_{y}^{2\tilde{r}}(\mathbb{T}^{d-k} \times \mathbb{T}^k)$.
		Then, for all
		$\tfrac12 < \tilde{r}, r < \infty$,
		\begin{equation}\label{eq vector LP}
			\frac{1}{C}
			\left\|
			\sum_{k=1}^{m} \lambda_k |u_k|^2
			\right\|_{L_x^{r} L_y^{\tilde r}}
			\;\le\;
			\left\|
			\sum_{N \in 2^\mathbb{N}} \sum_{k=1}^{m} \lambda_k |P_N u_k|^2
			\right\|_{L_x^{r} L_y^{\tilde r}}
			\;\le\;
			C
			\left\|
			\sum_{k=1}^{m} \lambda_k |u_k|^2
			\right\|_{L_x^{r} L_y^{\tilde r}},
		\end{equation}
		where the implicit constant is independent of the sequences
		$(\lambda_k)_{k=1}^{m}$ and $(u_k)_{k=1}^{m}.$
	\end{theorem}
	
	\begin{proof}
		We  prove this result in the case $1 < \tilde{r} \leq r < \infty$.
		The other case follow by a similar argument.
		Firstly, we will prove
		\begin{equation}\label{one way vector LP}
			\left\|
			\sum_{N \in 2^{\mathbb{N}}} \sum_{k=1}^{m}
			\lambda_k \, |P_N u_k|^2
			\right\|_{L_x^{r} L_y^{\tilde r}}
			\lesssim
			\left\|
			\sum_{k=1}^{m} \lambda_k |u_k|^2
			\right\|_{L_x^{r} L_y^{\tilde r}}.
		\end{equation}
		The above inequality is equivalent to 
		\[
		\left\|
		\left(
		\sum_{N \in 2^{\mathbb{N}}} \sum_{k=1}^{m}
		\bigl| \lambda_k^{1/2} P_N u_k \bigr|^2
		\right)^{1/2}
		\right\|_{L_x^{2r} L_y^{2\tilde r}}
		\lesssim
		\left\|
		\left(
		\sum_{k=1}^{m}
		\bigl| \lambda_k^{1/2} u_k \bigr|^2
		\right)^{1/2}
		\right\|_{L_x^{2r} L_y^{2\tilde r}} .
		\]
		Using Lemma \ref{Khintchine inequality} with $q=2\tilde{r},$ $n=2$ and Minkowski's integral inequality, we get
		\begin{align*}
			\Big\|
			\Big(
			\sum_{N \in 2^{\mathbb{N}}}\sum_{k=1}^{m} \lambda_k |P_N u_k|^2
			\Big)^{1/2}
			\Big\|_{L_x^{2r} L_y^{2\tilde r}}
			&\lesssim
			\Big\|
			\Big\| \Big\|
			\sum_{N \in 2^{\mathbb{N}}}\sum_{k=1}^{m}
			\lambda_k^{1/2} \varepsilon_{N} \varepsilon_{k} P_N u_k
			\Big\|_{L_{t_1,t_2}^{2\tilde r}}
			\Big\|_{L_y^{2\tilde r}}
			\Big\|_{L_{x}^{2r}} \\
			&\lesssim
			\Big\| \Big\|
			\sum_{N \in 2^{\mathbb{N}}}\sum_{k=1}^{m}
			\lambda_k^{1/2} \varepsilon_{N} \varepsilon_{k} P_N u_k
			\Big\|_{L_x^{2r} L_y^{2\tilde r}}
			\Big\|_{L_{t_1,t_2}^{2\tilde r}} .
		\end{align*}
		Again using Lemma~\ref{Hormander-mikhlm}, the Fourier multiplier
		\[
		\xi \mapsto \sum_{N \in 2^{\mathbb{N}}} \varepsilon_{N} \psi_{N}(\xi)
		\]
		is bounded on
		$L_x^{2r} L_y^{2\tilde r}(\mathbb{R}^{d-k} \times \mathbb{R}^{k})$.
		By Proposition~\ref{transference}, this boundedness transfers to
		$L_x^{2r} L_y^{2\tilde r}(\mathbb{T}^{d-k} \times \mathbb{T}^{k}),$ and the constant
		is independent of the choice of $\varepsilon$;
		for more details, we refer to the proof of Proposition~\ref{mixed LP}. Therefore
		\[
		\Big\| \Big\|
		\sum_{N \in 2^{\mathbb{N}}}\sum_{k=1}^{m}
		\lambda_k^{1/2} \varepsilon_{N} \varepsilon_{k} P_N u_k
		\Big\|_{L_x^{2r} L_y^{2\tilde r}}
		\Big\|_{L_{t_1,t_2}^{2\tilde r}} \lesssim 
		\Big\| \Big\|
		\sum_{k=1}^{m}
		\lambda_k^{1/2}  \varepsilon_{k} u_k
		\Big\|_{L_x^{2r} L_y^{2\tilde r}}
		\Big\|_{L_{t_2}^{2\tilde r}}.
		\]
		For simplicity, we assume $f
		=
		\sum_{k=1}^{m}\lambda_k^{1/2}\, \varepsilon_k\, u_k $ and $g=\|f\|_{L_y^{2\tilde r}} .$ Since $\frac{2\tilde r}{2r}\le 1,$ by Jensen's inequality, we have
		\begin{align}\label{11}\nonumber
			\Big\|
			\sum_{N \in 2^{\mathbb{N}}}\sum_{k=1}^{m}
			\lambda_k^{1/2} \varepsilon_{N} \varepsilon_{k} P_N u_k
			\Big\|_{L_x^{2r} L_y^{2\tilde r}}
			\Big\|_{L_{t_1,t_2}^{2\tilde r}} &\lesssim
			\Big\|\, \|f\|_{L_x^{2r}L_y^{2\tilde r}} \,\Big\|_{L_{t_2}^{2\tilde r}} =
			\Big\|\, \big\| g\big\|_{L_x^{2r}} \,\Big\|_{L_{t_2}^{2\tilde r}} \\\nonumber
			&=
			\left(
			\int_{[0,1]}
			\left(
			\int_{\mathbb{T}^{d-k}}
			|g|^{2r}\,dx
			\right)^{\frac{2\tilde r}{2r}} \,dt_2
			\right)^{\frac1{2\tilde r}} \\
			&\le
			\left(
			\int_{[0,1]}
			\int_{\mathbb{T}^{d-k}}
			|g|^{2 r}\,dx\,dt_2
			\right)^{\frac1{2 r}} = \big\|\|g\|_{L_{t_2}^{2 r}}^{2r}\big\|_{L_x^{1}}^{\frac{1}{2r}}.
		\end{align}
		Now, by Minkowski's inequality (since $2\tilde r \le 2r$), we get
		\begin{align*}
			\|g\|_{L_{t_2}^{2 r}}^{2r}
			&=
			\big\|\|f\|_{L_y^{2\tilde r}}\big\|_{L_{t_2}^{2 r}}^{2r}
			\le
			\big\|\|f\|_{L_{t_2}^{2r}}\big\|_{L_y^{2\tilde r}}^{2r} .
		\end{align*}
		Thus we have
		\begin{align*}
			\Big\|
			\sum_{N \in 2^{\mathbb{N}}}\sum_{k=1}^{m}
			\lambda_k^{1/2} \varepsilon_{N} \varepsilon_{k} P_N u_k
			\Big\|_{L_x^{2r} L_y^{2\tilde r}}
			\Big\|_{L_{t_1,t_2}^{2\tilde r}} &\lesssim \Big\|
			\big\|\|f\|_{L_{t_2}^{2r}}\big\|_{L_y^{2\tilde r}}^{2r} \Big\|_{L_{x}^{1}}^{\frac{1}{2r}} = \Big\| \big\| f \big\|_{L_{t_2}^{2r}} \Big\|_{L_{x}^{2r}L_{y}^{2\tilde{r}}}.
		\end{align*}
		Again using Lemma \ref{Khintchine inequality} with $q=2r,$ $n=1,$ we get
		\[
		\Big\|
		\Big(
		\sum_{j,k} \lambda_k |P_j u_k|^2
		\Big)^{1/2}
		\Big\|_{L_x^{2r} L_y^{2\tilde r}}
		\lesssim
		\Big\|
		\Big(
		\sum_k \lambda_k |u_k|^2
		\Big)^{1/2}
		\Big\|_{L_x^{2r} L_y^{2\tilde r}}.
		\]
		Therefore,  we obtained our required inequality \eqref{one way vector LP}.
		
		The reverse inequality will be proved using a duality argument together along with
		\eqref{one way vector LP}.
		Let $V$ satisfy $\|V\|_{L_x^{r'} L_y^{\tilde r'}} \le 1$.
		Then by H\"older's inequality
		\begin{align*}
			\int_{\mathbb{T}^d}
			\Big( \sum_{k=1}^{m} \lambda_k |u_k(x,y)|^2 \Big)\, V(x,y)\, dy\, dx 
			&=
			\sum_{k=1}^{m} \lambda_k
			\int_{\mathbb{T}^d}
			\overline{u_k(x,y)}\, V(x,y)\, u_k(x,y)\, dy\, dx \\
			&=
			\int_{\mathbb{T}^d}\sum_{N \in 2^{\mathbb{N}}}\sum_{k=1}^{m} \lambda_k
			\overline{P_N u_k(x,y)}\, \widetilde{P}_N V(x,y)\, u_k(x,y)\, dy\, dx,
		\end{align*}
		where $\widetilde{P}_N$ is defined in \eqref{tilde PN} and we used identity for two Schwartz functions $f$ and $g$ (see \cite[p. 36]{ward2010mixedLebesgue}): $$\int_{\mathbb{T}^d} f \overline{g} \,dx \, dy= \int_{\mathbb{T}^d}\sum_{N \in 2^{\mathbb{N}}}  P_{N}f \overline{\widetilde{P}_{N}g} \,dx \, dy.$$
		By H\"older's inequality
		\[
		\le
		\int_{\mathbb{T}^d}
		\Big( \sum_{N \in 2^{\mathbb{N}}}\sum_{k=1}^{m}
		\lambda_k |P_N u_k(x,y)|^2 \Big)^{1/2}
		\Big( \sum_{N \in 2^{\mathbb{N}}}\sum_{k=1}^{m}
		\lambda_k |\widetilde{P}_N V(x,y) u_k(x,y)|^2 \Big)^{1/2}
		dy\, dx .
		\]
		
		Applying H\"older inequality again we get 
		\[
		\le
		\left\|
		\Big( \sum_{N \in 2^{\mathbb{N}}}\sum_{k=1}^{m}
		\lambda_k |P_N u_k(x,y)|^2 \Big)^{1/2}
		\right\|_{L_x^{2r} L_y^{2\tilde r}}
		\left\|
		\Big( \sum_{N \in 2^{\mathbb{N}}}\sum_{k=1}^{m}
		\lambda_k |\widetilde{P}_N V(x,y) u_k(x,y)|^2 \Big)^{1/2}
		\right\|_{L_x^{(2r)'} L_y^{(2\tilde r)'}} .
		\]
		By applying \eqref{one way vector LP}, we get
		\[
		\lesssim
		\left\|
		\sum_{N \in 2^{\mathbb{N}}}\sum_{k=1}^{m}
		\lambda_k |P_N u_k(x,y)|^2 
		\right\|_{L_x^{r} L_y^{\tilde r}}^{\frac{1}{2}}
		\left\|
		\Big( \sum_{k=1}^{m}
		\lambda_k | V(x,y) u_k(x,y)|^2 \Big)^{1/2}
		\right\|_{L_x^{(2r)'} L_y^{(2\tilde r)'}} .
		\]
		Again using H\"older inequality, we get 
		\[
		\lesssim
		\left\|
		\sum_{N \in 2^{\mathbb{N}}}\sum_{k=1}^{m}
		\lambda_k |P_N u_k(x,y)|^2 
		\right\|_{L_x^{r} L_y^{\tilde r}}^{\frac{1}{2}}
		\left\|
		\Big( \sum_{k=1}^{m}
		\lambda_k | u_k(x,y)|^2 \Big)^{1/2}
		\right\|_{L_x^{2r} L_y^{2\tilde r}} \|V\|_{L_x^{r'}L_y^{\tilde{r}'}} .
		\]
		Finally we have, 
		$$\left\|
		\sum_{k=1}^{m}
		\lambda_k | u_k(x,y)|^2 
		\right\|_{L_x^{r} L_y^{\tilde r}} \lesssim \left\|
		\sum_{N \in 2^{\mathbb{N}}}\sum_{k=1}^{m}
		\lambda_k |P_N u_k(x,y)|^2 
		\right\|_{L_x^{r} L_y^{\tilde r}}
		$$ 
		and this completes the proof of the theorem.
	\end{proof}

	\section{Strichartz estimates for single function} \label{sec3} 
	In this section, we focus on establishing Strichartz estimates for the Schr\"odinger propagator $e^{-it\Delta}$  in the framework of partial regularity. Before presenting the main result, we first recall a few preparatory lemmas. The following lemma concerns the fixed-time decay estimate of the Schr\"odinger propagator  $e^{-it\Delta}$ on mixed-norm Lebesgue spaces.

	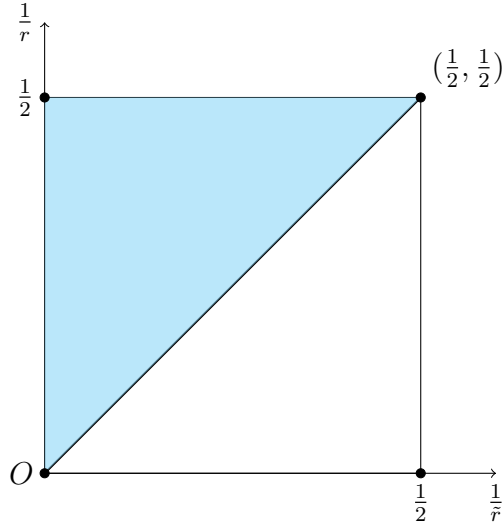
\begin{figure}[h]
		\centering
		\begin{tikzpicture}[scale=1]
			
			\draw[->] (0,0) -- (0,6) node[left] {$\tfrac{1}{r}$};
			\draw[->] (0,0) -- (6,0) node[below] {$\tfrac{1}{\tilde{r}}$};
			
			\draw (0,0) rectangle (5,5);
			\draw[thick] (0,0)--(5,5);

			
			
			\fill[cyan!40, opacity=0.6]
			(0,0) -- (5,5) -- (0,5) -- cycle;
			
			\filldraw[black] (5,5) circle (0.06) node[above right] {$(\tfrac{1}{2},\tfrac{1}{2})$};
			\filldraw[black] (0,5) circle (0.06) node[left] {$\tfrac{1}{2}$};
			\filldraw[black] (0,0) circle (0.06) node[left] {$O$};
			\filldraw[black] (5,0) circle (0.06) node[below] {$\tfrac{1}{2}$};
			
			
		\end{tikzpicture}
		\caption{The shaded region represent $2 \leq \tilde{r} \leq r \leq \infty.$ }
		\label{fig: tilda r and r}
	\end{figure}

	\begin{lemma}[fixed-time decay]\label{eq6}
		Let $d \geq 2 $ and $1 \leq k \leq d$.    Assume that $2 \leq \widetilde{r} \leq r \leq \infty,$ i.e., the shaded region in Figure \ref{fig: tilda r and r}. Then, for $(x, y) \in \mathbb{T}^{d-k} \times \mathbb{T}^k$ and $t\in [ -\frac{1}{2N}, \frac{1}{2N}]$, we have the following estimate:
		\begin{align}\label{eq5}
			\left\|e^{-i t\Delta} P_{\leq N} f\right\|_{L_x^r L_y^{\tilde{r}}} \lesssim |t|^{-\beta(r, \widetilde{r})}\|f\|_{L_x^{r^{\prime}} L_y^{\tilde{r}^{\prime}}}, 
		\end{align}
		where $P_{\leq N}$ is the Littlewood-Paley operator defined in \eqref{P less N}, constructed using the cutoff $\mathbf{1}_{S_{d,1}}$  and $$\beta(r, \widetilde{r}):=(d-k)\left(\frac{1}{2}-\frac{1}{r}\right)+k\left(\frac{1}{2}-\frac{1}{\widetilde{r}}\right).$$
	\end{lemma} 
	\begin{proof}

		In order to get the estimate (\ref{eq5}),  by Riesz-Thorin interpolation theorem,  it is enough to obtain estimate (\ref{eq5})  for the following three cases:
		\begin{enumerate}[\textbf{Case} (a)]
			\item $r=\widetilde{r}=2$,
			\item  $r=\widetilde{r}=\infty$,
			\item  $r=\infty$ and $\tilde{r}=2$.
		\end{enumerate}
		By  the Fourier transform and its inversion formula, we can write
		$$
		\begin{aligned}
			e^{-i t\Delta} P_{\leq N} f(x, y) & =\int_{\mathbb{T}^{d-k}} \int_{\mathbb{T}^k} K_N\left(x-x^{\prime}, y-y^{\prime}, t\right) f\left(x^{\prime}, y^{\prime}\right) d y^{\prime} d x^{\prime} \\
			& =(K_N *_{x,y} f)(x,y),
		\end{aligned}
		$$
		where  the kernel $K_N$ is given by 
		\begin{align}\label{eq9}
			K_N(x, y, t)= \sum_{\xi_{1} \in S_{d-k,N}} \sum_{\xi_{2} \in S_{k,N}} e^{2\pi i (x\cdot \xi_{1}+y \cdot \xi_{2}+ t  |(\xi_{1}, \xi_{2})|^2)} ,
		\end{align}
		where $S_{d,N}$ is defined in \eqref{rdp}.
		It is clear that \textbf{Case} (a) follows directly from the Plancherel theorem of the Fourier transform. Indeed, 
		\begin{align*}
			\|e^{-i t\Delta} P_{\leq N} f\|_{L^2}&=\| K_N* f\|_{L^2}\\
			&=\| \widehat{K_N} \widehat{f}\|_{\ell^2} \leq \| \widehat{f}\|_{\ell^2}=\|  f\|_{L^2}.
		\end{align*}
		
		For \textbf{Case} (b), that is, when  $r=\widetilde{r}=\infty$,  using Young's inequality, we get
		\begin{align}\label{Kernel estimate 1}
			\|e^{-i t\Delta} P_{\leq N} f\|_{L ^\infty}=\| K_N*  f\|_{L^\infty}\leq\left\|K_N\right\|_{L^{\infty}}\|f\|_{L^1} .
		\end{align}
		Now, we have to estimate $L^{\infty}$  norm of the kernel $K_N$.  From Proposition \ref{l8}, 
		$$\left\|K_N\right\|_{L^{\infty}}\leq C |t|^{-\frac{d}{2}},$$
		for any \((x, t) \in \mathbb{T} \times [-N^{-1}, N^{-1}].
		\)
		Then from the inequality \eqref{Kernel estimate 1}, we get 
		$$		\|e^{-i t\Delta} P_{\leq N} f\|_{L ^\infty} \leq\left\|K_N\right\|_{L ^{\infty}}\|f\|_{L ^1}  \lesssim |t|^{-\beta(\infty, \infty)} \|f\|_{L ^1} $$
		with $ \beta(\infty, \infty)=\frac{d}{2} \geq 0$ and   this concludes {\bf Case} (b).
		
		For \textbf{Case} (c), that is, when $r=\infty$ and $\tilde{r}=2$. This implies that $ \beta(\infty, 2)=\frac{d-k}{2} $.  Using Minkowski's inequality and Plancherel's theorem with respect to $y$ variable, we get 
		\begin{align}\label{Kernel 2}\nonumber
			\left\|e^{-i t\Delta} P_{\leq N} f\right\|_{L_x^{\infty} L_y^2} & =\left\|\| K_N *_{x,y} f \|_{L_y^2}\right\|_{L_x^{\infty}} \\\nonumber
			& \leq\left\|\int_{\mathbb{T}^{d-k}} \| K_N\left(t, x-x^{\prime}, \cdot\right) *_y f\left(x^{\prime}, \cdot\right) \|_{L_y^2} d x^{\prime}\right\|_{L_x^{\infty}} \\
			& =\left\|\int_{\mathbb{T}^{d-k}} \| \widetilde{K_N}\left(t, x-x^{\prime}, \cdot\right) \tilde{f}\left(x^{\prime}, \cdot\right) \|_{\ell_{\xi_{2}}^2} d x^{\prime}\right\|_{L_x^{\infty}},
		\end{align}
		where  $\tilde{f}=\mathcal{F}_y(f(x, \cdot))$ denotes the  Fourier transform of $f$ with respect to  the   $y \in \mathbb{T}^k$ variable. Let us denote    $\tilde{K}_N=\mathcal{F}_y(K_N(x, \cdot))$. Then    $\tilde{K}_N$  can be written   as 
		$$
		\widetilde{K}_N(t, x, {\xi_{2}})= \sum_{\xi_{1}\in S_{d-k,N}} e^{2\pi i (x \cdot \xi_{1}+ t  |(\xi_{1}, {\xi_{2}})|^2)} . 
		$$
		Again from  Lemma \ref{l8}, we have 
		\begin{align}\label{eq7}
			\sup _{{\xi_{2}} \in  \Z^k} \left|\widetilde{K}_N(t, x, {\xi_{2}})\right|= \sup _{{\xi_{2}} \in  \Z^k} \left| \sum_{\xi_{1}\in S_{d-k,N}} e^{2\pi i (x \cdot \xi_{1}+ t  |(\xi_{1}, {\xi_{2}})|^2)} \right|  \lesssim | t|^{-\frac{d-k}{2}} 
		\end{align}
		for   $ (x, t) \in \mathbb{T}^{d-k} \times [-N^{-1}, N^{-1}]$.
		
		Coming back to the inequality \eqref{Kernel 2},	using  the estimate  \eqref{eq7}, we obtain 
		$$
		\begin{aligned}
			\left\|e^{-i t\Delta} P_{\leq N} f\right\|_{L_x^{\infty} L_y^2}
			&	\lesssim 	\left\|\int_{\mathbb{T}^{d-k}}  |t|^{-\frac{d-k}{2}}\|   \tilde{f}\left(x^{\prime}, \cdot\right) \|_{\ell_{\xi_{2}}^2} d x^{\prime}\right\|_{L_x^{\infty}}.
		\end{aligned}
		$$
		Let $g(x)=1$ for all $x\in \T^{d-k}$. Then using Young's convolution inequality, we get
		$$
		\begin{aligned}
			\left\|e^{-i t\Delta} P_{\leq N} f\right\|_{L_x^{\infty} L_y^2}  
			&	\lesssim 	\left\|\int_{\mathbb{T}^{d-k}}  |t|^{-\frac{d-k}{2}}\|   \tilde{f}\left(x^{\prime}, \cdot\right) \|_{\ell_{\xi_{2}}^2} d x^{\prime}\right\|_{L_x^{\infty}}\\ 
			& \lesssim\left\||t|^{-\frac{d-k}{2}} g*_x \| \tilde{f} \|_{\ell_{\xi_{2}}^2}\right\|_{L_x^{\infty}} \\
			& =\left\||t|^{-\frac{d-k}{2}} g*_x \| f\|_{L_y^2}\right\|_{L_x^{\infty}} \\
			& \lesssim |t|^{-\frac{d-k}{2}}  \|g\|_{L^\infty_x} \|f\|_{L_x^1 L_y^2} \\
			& \leq |t|^{-\frac{d-k}{2}}\|f\|_{L_x^1 L_y^2}\\
			& =|t|^{-\beta(\infty, 2)}\|f\|_{L_x^1 L_y^2}
		\end{aligned}
		$$
		and this completes the proof of the lemma.  
	\end{proof}
	Next, using the above fixed time decay estimates, we have the following  frequency localized estimates
	\begin{lemma}[frequency localized estimate] \label{Loc} Assume that $d \geq 2,1 \leq k \leq d$, and
		\begin{align*}
			\frac{2}{q} \geq (d-k)\left(\frac{1}{2}-\frac{1}{r}\right)+k\left(\frac{1}{2}-\frac{1}{\widetilde{r}}\right), \quad 2 \leq \widetilde{r} \leq r<\infty, \quad 2< q \leq \infty. 
		\end{align*}
		Then, we have  
		\begin{align}\label{eq4}
			\left\|e^{-i t\Delta} P_{\leq N} f\right\|_{L_t^q\left(I_N; L_x^r (\T^{d-k}; L_y^{\tilde{r}}(\T^k) \right)} \lesssim\|f\|_{L^2},
		\end{align}
		where $I_N=[-\frac{1}{2N}, \frac{1}{2N}]$.
	\end{lemma}
	\begin{proof}
		By   a standard $T T^*$ argument (see \cite[Corollary 2.3]{zucco2010schrodingerequation}),  the required estimate (\ref{eq4}) is equivalent to
		$$
		\left\|\int_{0}^t e^{-i(t-s)\Delta} P_{\leq N}  g(s,\cdot) d s\right\|_{L_t^q\left(I_N; L_x^r (\T^{d-k}; L_y^{\tilde{r}}(\T^k) \right)}  \lesssim\|g\|_{L_t^{q'}\left(I_N; L_x^{r'} (\T^{d-k}; L_y^{\tilde{r}'}(\T^k) \right)}.
		$$
		By   Lemma \ref{eq6}, we obtain
		$$
		\begin{aligned}
			\left\|\int_{0}^t e^{-i(t-s)\Delta} P_{\leq N}  g(s,\cdot)\; d s\right\|_{L_t^q\left(I_N; L_x^r (\T^{d-k}; L_y^{\tilde{r}}(\T^k) \right)} 
			& \leq\left\|\int_{0}^t \| e^{-i(t-s)\Delta}  P_{\leq N}  g(s,\cdot) \|_{L_x^r L_y^{\tilde{r}}}\; ds\right\|_{L_t^q(I_N)} \\
			& \lesssim\left\|\int_{0}^t  |t-s|^{-\beta(r, \widetilde{r})}\|g(s,\cdot)\|_{L_x^{r^{\prime}} L_y^{\tilde{r}^{\prime}}}\; ds\right\|_{L_t^q(I_N)} \\
			&=\left\||\cdot|^{-\beta(r, \widetilde{r})} *_t \| g \|_{L_x^{r^{\prime}} L_y^{\tilde{r}^{\prime}}}\right\|_{L_t^q(I_N)}.
		\end{aligned}
		$$
		From the notation of $\beta(r, \tilde{r})$,  the inequality $	\frac{2}{q} \geq (d-k)\left(\frac{1}{2}-\frac{1}{r}\right)+k\left(\frac{1}{2}-\frac{1}{\widetilde{r}}\right)$  reduces to   $\frac{2}{q}\geq  \beta(r, \tilde{r})$. We now consider the following three cases:
		\begin{itemize}
			\item When $\frac{2}{q} > \beta(r, \tilde{r})$ and $2 \leq q \leq \infty$:    Then  Young's convolution inequality yields
			$$
			\begin{aligned}
				\left\| |\cdot|^{-\beta(r, \widetilde{r})} *_t \| g \|_{L_x^{r^{\prime}} L_y^{\widetilde{r}'}}\right\|_{L_t^q(I_N)} & \lesssim\left\||\cdot|^{-\beta(r, \widetilde{r})}\right\|_{L_t^{\frac{q}{2}}(I_N)}\|g\|_{L_t^{q'}\left(I_N; L_x^{r'} (\T^{d-k}; L_y^{\tilde{r}'}(\T^k) \right)} \\
				& \lesssim\|g\|_{L_t^{q'}\left(I_N; L_x^{r'} (\T^{d-k}; L_y^{\tilde{r}'}(\T^k) \right)}.
			\end{aligned}
			$$
			\item 	When $\frac{2}{q}=\beta(r, \tilde{r})$ with $2<q<\infty$, Applying one-dimensional Hardy-Littlewood Sobolev inequality, we get 
			$$
			\begin{aligned}
				\left\||\cdot|^{- \beta(r, \tilde{r})} *_t \| g \|_{L_x^{r^{\prime}} L_y^{\tilde{r}^{\prime}}}\right\|_{L_t^q(I_N)}
				& \lesssim\|g\|_{L_t^{q'}\left(I_N; L_x^{r'} (\T^{d-k}; L_y^{\tilde{r}'}(\T^k) \right)}.
			\end{aligned}
			$$
			\item 	When  $q=\infty$ and $ \beta(r, \widetilde{r})=0$: This implies  that  $r=\widetilde{r}=2$. In this case, the required estimate (\ref{eq5})   holds trivially by Plancherel's theorem. This completes the proof of the lemma.
	\end{itemize} \end{proof} 
	
	Let $I' \subset \mathbb{R}$ be an interval of length $1/N$, and denote by $C(I')$ the center of $I'$. By Lemma \ref{Loc}, we have  
	\begin{align*}
		\left\| e^{-i t' \Delta} P_{\leq N} f \right\|_{L_{t'}^{q}(I') L_{x}^{r} L_{y}^{\tilde r}}
		&= \left\| e^{-i t \Delta} P_{\leq N} g \right\|_{L_{t}^{q}(I_{N}) L_{x}^{r} L_{y}^{\tilde r}} \\
		&\leq \|g\|_{L^{2} }
		= \|f\|_{L^{2} },
	\end{align*}
	where 
	\[
	g = \big( e^{\, 2 \pi i\, C(I')\, |(\xi_{1},\xi_{2})|^{2}} \, \widehat{f}(\xi_{1},\xi_{2}) \big)^{\vee}.
	\]
	
	We partition the time torus $\mathbb{T}$ into $N$ subintervals $\{I_{k,N}\}_{k=1}^{N}$, each of length $1/N$, so that  
	\[
	\mathbb{T} = \bigcup_{k=1}^{N} I_{k,N}.
	\]
	Then  
	\begin{align*}
		\left\| e^{-i t \Delta} P_{\leq N} f \right\|_{L_t^{q}(\mathbb{T}) L_x^{r} L_y^{\tilde r}}^{q}
		&\leq 
		N \left\| e^{-i t \Delta} P_{\leq N} f \right\|_{L_t^{q}(I_{k,N}) L_x^{r} L_y^{\tilde r}}^{q} \\
		&\leq N \|f\|_{L^{2}}^{q}.
	\end{align*}
	Thus,
	\begin{equation}\label{Time interval}
		\left\| e^{-i t \Delta} P_{\leq N} f \right\|_{L_t^{q}(\mathbb{T}) L_x^{r} L_y^{\tilde r}}
		\leq N^{1/q} \|f\|_{L^{2}}.
	\end{equation}
	Now we are in a position to prove  Strichartz estimates in the partial frame work. 
	\begin{proof}[\textbf{Proof of Theorem \ref{SE Single function}}] We begin the proof by substituting $e^{-i t\Delta} f$  into the estimate  \eqref{norm} as follows  $$
		\|e^{-i t\Delta} f\|_{L_t^qL_x^r L_y^{\tilde{r}}} \lesssim \left\| \left(\sum_{N \in 2^\mathbb{N}} \left|P_N e^{-i t\Delta} f\right|^2\right)^{1 / 2}\right\|_{L_t^qL_x^r L_y^{\tilde{r}}}. 
		$$ Then, by the Minkowski inequality, we  get 
		\begin{align}\label{eq12}\nonumber
			\left\|e^{-i t\Delta} f\right\|_{L_t^q L_x^r L_y^{\tilde{r}}}^2 &\lesssim\left\| \left( \sum_{N \in 2^\mathbb{N}} \left|P_N e^{-i t\Delta} f\right|^2\right)^{1 / 2}\right\|_{L_t^q L_x^r L_y^{\tilde{r}}}^2 \\
			& \lesssim \sum_{N \in 2^\mathbb{N}}\left\|e^{-i t\Delta} P_N f\right\|_{L_t^q L_x^r L_y^{\tilde{r}}}^2 \nonumber \\
			&= \left\|e^{-i t\Delta} P_1 f\right\|_{L_t^q L_x^r L_y^{\tilde{r}}}^2+\sum_{N \in 2^\mathbb{N}, N \neq 1}\left\|e^{-i t\Delta} P_N f\right\|_{L_t^q L_x^r L_y^{\tilde{r}}}^2,
		\end{align}
		where in the second step, we used the fact that $P_N$ commutes with the semigroup $e^{-i t\Delta}$.		 
		
		Now, choose compactly supported Schwartz functions $\tilde{\psi}_{N}$ for all 
		$N \in 2^{\mathbb{N}},\, N \neq 1$, such that 
		$\tilde{\psi}_{N} \equiv 1$ in a neighborhood of $\operatorname{supp}(\psi_{N})$.  
		Using these functions, define $\tilde{P}_{N}$ as 
		\begin{equation}\label{tilde PN}
			\widehat{\tilde{P}_{N} f} = \tilde{\psi}_{N}\, \widehat{f}.
		\end{equation}
		With this choice, we can write $ 
		P_{N} f = P_{N} \tilde{P}_{N} f.$    Then  from \eqref{eq12} and   \eqref{Time interval} with $s=\frac{1}{q}$, we obtain  
		$$
		\begin{aligned}
			\left\|e^{-i t\Delta} f\right\|_{L_t^q L_x^r L_y^{\tilde{r}}}^2  
			& \lesssim \left\|e^{-i t\Delta} P_1 f\right\|_{L_t^q L_x^r L_y^{\tilde{r}}}^2+\sum_{N \in 2^\mathbb{N}, N \neq 1}\left\|e^{-i t\Delta} P_N \tilde{P}_{N} f\right\|_{L_t^q L_x^r L_y^{\tilde{r}}}^2\\
			& \lesssim  \|f\|_{L^{2}}^{2}+ \sum_{N \in 2^\mathbb{N}, N \neq 1} N^{2s}\left\|\widetilde{P}_N f\right\|_{L^2}^2 \lesssim\|f\|_{{H}^s}^2
		\end{aligned}
		$$
		with $s=\frac{1}{q}$ and this completing the proof of the theorem. 
	\end{proof}
	\section{Local well-posedness for the Hartree type equations}\label{sec4}
	In this section, we apply the refined Strichartz estimate from Theorem  \ref{SE Single function}  to study the well-posedness of the nonlinear equation \eqref{NLS} when the initial data are chosen from the partial regularity space   $X_k^s;$ rather than the full regularity space {${H}^s$}. As a preliminary step, we first recall the well-known fractional chain rule on the torus, which will play a key role in the proof.

	\begin{proposition}[Fractional chain rule on torus \cite{lee2019wellposedness}]\label{fractional chain}
		Suppose $F : \mathbb{C} \to \mathbb{C}$ satisfies 
		\[
		|F(u) - F(v)| \lesssim |u - v|\big(G(u) + G(v)\big)
		\]
		for some $G : \mathbb{C} \to [0, \infty)$. Let $d \ge 1$, $0 \leq s < 1$, $1 < \mathfrak{a} < \infty$, and $1 < c \le \infty$, such that 
		\[
		\frac{1}{\mathfrak{a}} = \frac{1}{b} + \frac{1}{c}.
		\]
		Then
		\[
		\|\,|\nabla|^s F(u)\|_{L^\mathfrak{a}} 
		\lesssim 
		\|\,|\nabla|^s u\|_{L^{b}} 
		\|G(u)\|_{L^{c}}.
		\]
	\end{proposition}
	Let $G_{p}(u)=w \ast F_{p}(u),$ and we collect some basic properties of function $F_{p}(u)$ to be used in the proof:
	\begin{equation}\label{difference of Fp}
		| F_{p}(u)-F_{p}(v)| \leq C (|u|^{p-1}+|v|^{p-1})|u-v|
	\end{equation}
	and 
	\begin{equation}\label{fractional chain for Fp}
		\|\,|\nabla|^s F_{p}(u)\|_{L^\mathfrak{a}} 
		\lesssim 
		\|\,|\nabla|^s u\|_{L^{b}} 
		\|u\|_{L^{c(p-1)}}^{p-1},
	\end{equation}
	where $0\leq s<1,$ $1<\mathfrak{a},c< \infty$ and $\frac{1}{\mathfrak{a}}=\frac{1}{b}+\frac{1}{c}.$  The inequality
	\eqref{fractional chain for Fp} directly follows from \eqref{difference of Fp} and Proposition \ref{fractional chain}. For a detailed study, we refer to \cite{koh2023local}.
	
	
	\begin{lemma}[nonlinear estimate]\label{nonlinearity estimate x y}
		Let $d \geq 3,$ $0 \leq s < 1,$ $ a \in [0,\infty),$ $1<p<1+\frac{4}{d-2s}$ and $w \in W_{x,y}^{\frac{2}{q},1}.$ For any time $I=[0,T],$ there are triples $(q,r,\tilde{r}) \in \mathbb{A}(d,2)$ (see \eqref{Adk}) and $\beta_{0}(d,p,s)>0$ such that

		\begin{equation}\label{nonlinearity estimate xy both}
			\| \langle \nabla_y \rangle^{s} \langle \nabla \rangle^{a} G_p(u) \|_{L_t^{q'}(I; L_x^{r'} L_y^{\tilde{r}'})}
			\leq C T^{\beta_0(d,p,s)} \| w \|_{ W ^{a,1}}
			\| u \|_{L_t^{q}(I; L_x^{r} W_y^{s,\tilde{r}})}^{p}. 
		\end{equation}
		Moreover,
		\begin{equation}\label{difference of Gu GV}
			\resizebox{\textwidth}{!}{$\| \langle \nabla \rangle^{a} (G_p(u) - G_p(v)) \|_{L_t^{q'}(I; L_x^{r'} L_y^{\tilde{r}'})}
				\lesssim \| w \|_{ W_{x,y}^{a,1}} T^{\beta_0(d,p,s)} 
				\left( 
				\| u \|_{L_t^{q}(I; L_x^{r} W_y^{s,\tilde{r}})}^{p-1}
				+ \| v \|_{L_t^{q}(I; L_x^{r} W_y^{s,\tilde{r}})}^{p-1}
				\right)
				\| u - v \|_{L_t^{q}(I; L_x^{r} L_y^{\tilde{r}})}.$}
		\end{equation}
	\end{lemma}
	\begin{proof}
		Let $G_{a,p}(u)=\langle \nabla \rangle^{a}  G_p(u)=\langle \nabla \rangle^{a}  (w \ast F_{p}(u)).$ 
		Let $d \ge 3$ and $0 \leq s < 1$. For fixed $1 < p < 1 + \dfrac{4}{d - 2s}$, 
		we take a small $\varepsilon_{0} > 0$ such that
		\[
		\dfrac{(1-s)(p-1)}{2} < \varepsilon_{0} < 
		\min\left\{ 
		\dfrac{d-2}{4}\left(1 + \dfrac{4}{d-2} - p\right), 
		\dfrac{p-1}{2}
		\right\},
		\]
		and then take a triple $(q, r, \tilde{r})$ as
		\begin{equation}\label{conditions q r} 
			\frac{1}{r} = \frac{1}{p+1}, \qquad 
			\frac{1}{\tilde{r}} = \frac{1}{2} - \frac{\varepsilon_{0}}{p+1}, \qquad 
			\frac{1}{q} = \frac{d-2}{4} - \frac{d-2}{2(p+1)} + \frac{\varepsilon_{0}}{p+1},
		\end{equation}
		so that $(q, r, \tilde{r}) \in \mathbb{A}(d,2).$ 
		Using the notation given in \eqref{homogeneous dominate}, together with Minkowski's inequality, we get 
		\begin{equation*}
			\| \langle \nabla_y \rangle^{s} G_{a,p}(u) \|_{L_t^{q'} L_x^{r'} L_y^{\tilde{r}'}} \lesssim \| G_{a,p}(u) \|_{L_t^{q'} L_x^{r'} L_y^{\tilde{r}'}}+\| \lvert \nabla_y \rvert^{s} G_{a,p}(u) \|_{L_t^{q'} L_x^{r'} L_y^{\tilde{r}'}}.
		\end{equation*}
		For the first term of the above sum, using Young’s convolution inequality, \eqref{growth condition},  H\"older’s inequality,  and    the first two conditions in \eqref{conditions q r}, we obtain 
		\begin{align*}
			\| G_{a,p}(u) \|_{L_t^{q'}(I; L_x^{r'} L_y^{\tilde{r}'})} &=\| (\langle \nabla \rangle^{a}w) \ast F_{p}(u) \|_{L_t^{q'} L_x^{r'} L_y^{\tilde{r}'}} \\
			&\lesssim \| \langle \nabla \rangle^{a}w \|_{L_x^{1} L_y^{1}} \|  F_{p}(u) \|_{L_t^{q'} L_x^{r'} L_y^{\tilde{r}'}} \\
			&\lesssim   \| w \|_{ W ^{a,1}}  \||u|^{p} \|_{L_t^{q'} L_x^{r'} L_y^{\tilde{r}'}} \\
			& \lesssim T^{1-\frac{p+1}{q}} \| w \|_{ W ^{a,1}}  \|u\|^{p-1}_{L_t^{q} L_x^{r} L_y^{{(p-1)(p+1)}/{2 \varepsilon_{0}}}} \|u \|_{L_t^{q} L_x^{r} L_y^{\tilde{r}}}.
		\end{align*}
		For the second term, we again apply Young’s convolution inequality along with \eqref{fractional chain for Fp}, and then use Hölder’s inequality as before to obtain
		\begin{align*}
			\| \lvert \nabla_y \rvert^{s} G_{a,p}(u) \|_{L_t^{q'} L_x^{r'} L_y^{\tilde{r}'}}&=\| (\langle \nabla \rangle^{a}w) \ast (|\nabla_{y}|^{s}F_{p}(u)) \|_{L_t^{q'} L_x^{r'} L_y^{\tilde{r}'}} \\
			&\lesssim \| \langle \nabla \rangle^{a}w \|_{L_x^{1} L_y^{1}} \|  |\nabla_{y}|^{s}F_{p}(u) \|_{L_t^{q'} L_x^{r'} L_y^{\tilde{r}'}} \\
			& \lesssim T^{1-\frac{p+1}{q}}  \| w \|_{ W ^{a,1}} \|u\|^{p-1}_{L_t^{q} L_x^{r} L_y^{{(p-1)(p+1)}/{2 \varepsilon_{0}}}} \|u \|_{L_t^{q} L_x^{r} W_y^{s,\tilde{r}}}
		\end{align*}
		Since $W^{s,\tilde{r}} \subseteq L^{\tilde{r}}$ for $s \geq 0,$ and 
		\begin{equation}\label{Embedding}
			W^{s,\tilde{r}} \hookrightarrow W^{1-\frac{2\varepsilon_{0}}{p-1},\tilde{r}} \hookrightarrow L^{\frac{(p-1)(p+1)}{2\varepsilon_{0}}},
		\end{equation}
		by Lemma~\ref{sobolev embbeding T} (Sobolev embedding in 2D) together with \eqref{inhomogeneous dominate}, we finally get \eqref{nonlinearity estimate xy both} with 
		$$
		\beta_0(d,p,s)=1-\frac{p+1}{q}=\frac{d-2}{4}(1+\frac{4}{d-2}-p) -\varepsilon_{0}>0,
		$$
		using the last condition in \eqref{conditions q r} together with the choice of $\varepsilon_{0}.$ 
		Similarly, the estimate \eqref{difference of Gu GV} is obtained more directly by combining Young's convolution inequality with \eqref{difference of Fp}, H\"older's inequality, and an appropriate embedding.
	\end{proof}
	To establish the well-posedness, some preliminary results are required. Let $x \in \mathbb{T}^{d-2}$ and $y \in \mathbb{T}^{2}.$
	For any triple $(q, r, \tilde{r}) \in \mathbb{A}(d,2),$ it follows from \eqref{For single} with $k = 2$ that
	\begin{equation}\label{SE k=2}
		\| e^{it\Delta} f \|_{L_{t}^{q}(I; L_{x}^{r} L_{y}^{\tilde{r}})} 
		\le C \| f \|_{{H}^{\frac{1}{q}}},
	\end{equation}
	for any time interval $I = [0,T].$ Furthermore, this leads to the inhomogeneous estimate
	\begin{equation}\label{single inhomogeneous estimate}
		\left\| \langle \nabla \rangle^{\frac{-2}{q}}
		\int_{0}^{t} e^{-i(t-t')\Delta} G_{p}(u) \, dt'
		\right\|_{L_{t}^{q}(I; L_{x}^{r} L_{y}^{\tilde{r}})}
		\le
		C \| G_{p}(u) \|_{L_{t}^{a'}(I; L_{x}^{b'} L_{y}^{\tilde{b}'})},
	\end{equation}
	for any triples $(q, r, \tilde{r}), (a, b, \tilde{b}) \in \mathbb{A}(d,2),$ by applying the standard $TT^{*}$ argument together with the Christ–Kiselev lemma.
	
	By Duhamel's principle, the solution map of \eqref{NLS} is given by 
	\begin{equation}\label{Duhamel formula}
		\Phi(u)=e^{-it \Delta}f - i \int_{0}^{t} e^{-i (t-t') \Delta} G_{p}(u) \, dt'.
	\end{equation}
	\begin{proof}[\textbf{Proof of Theorem \ref{wellposedness}}]
		For $0 \leq s < 1$ and suitable values of $T, A > 0,$ it suffices to prove that 
		$\Phi$ defines a contraction map on
		\[
		X(T, A) = \left\{
		\langle \nabla_{y} \rangle^{s} u \in C([0, T]; {H} ^{\frac{1}{q}}) \cap L_t^{q}([0, T]; L_x^{r} L_y^{ \tilde{r}}) :
		\sup_{(q, r, \tilde{r}) \in \mathbb{A}(d, 2)}
		\| \langle \nabla_{y} \rangle^{s}u\|_{L_t^{q}([0, T]; L_x^{r} L_y^{ \tilde{r}})} \le A
		\right\},
		\]
		equipped with the distance
		\[
		d(u, v) = 
		\sup_{(q, r, \tilde{r}) \in \mathbb{A}(d, 2)}
		\|u - v\|_{L_t^{q}(I; L_x^{r} L_y^{\tilde{r}})}.
		\]
		Let us now prove that $\Phi$ is a contraction on $X(T, A)$. 
		We first show that $\Phi(u) \in X$ for $u \in X.$ Using Plancherel’s theorem and the dual formulation of \eqref{SE k=2}, we deduce that 
		\begin{align*}
			\sup_{t \in I} 
			\| \langle \nabla_{y} \rangle^{s} \Phi(u)\|_{{H} ^{\frac{1}{q}}}
			&\lesssim
			\|\langle \nabla_{y} \rangle^{s} f\|_{{H} ^{\frac{1}{q}}}
			+ 
			\sup_{t \in I} 
			\left\|\langle \nabla_{y} \rangle^{s}
			\int_{0}^{t} e^{-i(t-t')\Delta} G_{p}(u) \, dt'
			\right\|_{{H} ^{\frac{1}{q}}} \\
			&=
			\|\langle \nabla_{y} \rangle^{s}f\|_{{H} ^{\frac{1}{q}}}
			+ 
			\sup_{t \in I} 
			\left\|\langle \nabla \rangle^{\frac{2}{q}}
			\int_{0}^{t} e^{i t'\Delta} \langle \nabla_{y} \rangle^{s}  G_{p}(u) \, dt'
			\right\|_{{H} ^{-\frac{1}{q}}} \\
			&\lesssim \|\langle \nabla_{y} \rangle^{s}f\|_{{H} ^{\frac{1}{q}}}+ \| \langle \nabla_y \rangle^{s} \langle \nabla \rangle^{\frac{2}{q}} G_p(u) \|_{L_t^{q'}(I; L_x^{r'} L_y^{\tilde{r}'})} \\
			&\lesssim \|\langle \nabla_{y} \rangle^{s}f\|_{{H} ^{\frac{1}{q}}}+T^{\beta_0(d,p,s)} \| w \|_{ W ^{\frac{2}{q},1}} \| u \|_{L_t^{q}(I; L_x^{r} W_y^{s,\tilde{r}})}^{p}
		\end{align*}
		for some $(q, r, \tilde{r}) \in \mathbb{A}(d, 2)$ for which Lemma \ref{nonlinearity estimate x y} and \eqref{nonlinearity estimate xy both} holds.  
		Thus we have for $u \in X(T,A).$ 
		\begin{equation*}
			\sup_{t \in I} 
			\| \langle \nabla_{y} \rangle^{s} \Phi(u)\|_{{H} ^{\frac{1}{q}}}
			\lesssim  \|\langle \nabla_{y} \rangle^{s} f\|_{{H} ^{\frac{1}{q}}}+T^{\beta_0(d,p,s)} \| w \|_{ W ^{\frac{2}{q},1}} A^p.
		\end{equation*}
		Now using \eqref{SE k=2},\eqref{single inhomogeneous estimate} and \eqref{nonlinearity estimate xy both} in Lemma \ref{nonlinearity estimate x y}, we get 
		\begin{equation*}
			\|\Phi(u)\|_{L_{t}^{q}(I; L_{x}^{r} W_{y}^{s, \tilde{r}})}
			\lesssim \|\langle \nabla_{y} \rangle^{s} f\|_{{H} ^{\frac{1}{q}}}+T^{\beta_0(d,p,s)} \| w \|_{ W ^{\frac{2}{q},1}} A^p
		\end{equation*}
		for $u \in X(T,A).$ Therefore, $\Phi(u) \in X(T,A)$ if 
		\begin{equation}\label{condition on A}
			C\|\langle \nabla_{y} \rangle^{s} f\|_{{H} ^{\frac{1}{q}}}+C T^{\beta_0(N,p,s)}\| w \|_{ W ^{\frac{2}{q},1}} A^p \leq A.
		\end{equation}
		On the other hand, applying \eqref{single inhomogeneous estimate} together with \eqref{difference of Gu GV} in Lemma~\ref{nonlinearity estimate xy both},
		\begin{align*}
			d(\Phi(u), \Phi(v))
			&= d\!\left(
			\int_{0}^{t} e^{-i(t-t')\Delta} G_{p}(u) \, dt', 
			\int_{0}^{t} e^{-i(t-t')\Delta} G_{p}(v) \, dt'
			\right) \\
			& \lesssim  \|\langle \nabla \rangle^{\frac{2}{q}} (G_p(u) - G_p(v))\|_{L_{t}^{q'}(I; L_{x}^{r'} L_{y}^{\tilde{r}'})} \\
			&\le 2C \| w \|_{ W ^{\frac{2}{q},1}} T^{\beta_{0}(d,p,s)} A^{p-1} d(u,v),
		\end{align*}
		for $u, v \in X$. By choosing $A = 2C  \|\langle \nabla_{y} \rangle^{s} f\|_{ H ^{\frac{1}{q}}}$ and taking $T$ sufficiently small so that
		\[
		2C C_{w} T^{\beta_{0}(d,p,s)} A^{p-1} \le \frac{1}{2},
		\]
		we see that \eqref{condition on A} holds and
		\[
		d(\Phi(u), \Phi(v)) \le \frac{1}{2} d(u,v)
		\]
for all $u, v \in X(T,A).$ Thus, it follows that the map  $\Phi$ is a contraction mapping on the Banach space $X(T,A)$. By fixed point theorem, we completes the proof of the local
well-posedness result.      
\end{proof}

	\section{Strichartz estimates for system of orthonormal functions}\label{sec5} 
	If $z \in \T^{d-k} \times \T^{k},$ then we write $z=(x,y),$ where $x \in \T^{d-k} $ and $y \in \T^{k}.$ Similarly, if $\xi \in \Z^{d-k} \times \Z^{k},$ then we denote $\xi=(\xi_{1},\xi_{2}),$ where $\xi_{1} \in \Z^{d-k}$ and $\xi_{2} \in \Z^{k}.$  For $a  \in \ell_{\xi_{1}}^{2}\ell_{\xi_{2}}^2,$ the  \textbf{Fourier extension operator} $\mathcal{E}_N$  is given by 
	\begin{equation}\label{FEP}
		\mathcal{E}_N a (t,z) = \sum_{\xi \in S_{d,N}}{a} (\xi)  e^{2\pi i (z \cdot \xi + t |\xi|^{2})} , \quad(t, z) \in I \times \T^{d-k} \times \T^{k},
	\end{equation}
	where $S_{d,N}$ is defined in \eqref{rdp}. By the Plancherel theorem, with $\mathcal H = L^{2}\!\left(\mathbb T; L^{2}(\mathbb T^{d-k}\times\mathbb T^{k})\right)$, the operator
	$\mathcal E_{N} : \ell^{2}_{\xi_{1}}\ell^{2}_{\xi_{2}} \to \mathcal H$ is bounded. In fact, we have 
	\begin{align} \label{d1}
		\|\mathcal{E}_{N}a\|_{\cH}^{2}&= \int_{I} \|\mathcal{E}_{N}a(t, \cdot)\|^2_{L_{z}^{2}} \, dt =\int_{I} \|\widehat{\mathcal{E}_{N}a(t, \cdot)}(\xi)\|^2_{\ell_{\xi_{1}}^{2}\ell_{\xi_{2}}^2} \, dt \lesssim \|a\|_{\ell_{\xi_{1}}^{2} \ell_{\xi_{2}}^{2}}^{2}.
	\end{align}    
	The dual\footnote{the dual operator of $\mathcal{E}_N$ means that
		$\langle \mathcal{E}_N a, F \rangle_{\cH} = \langle a, \mathcal{E}_N^* F \rangle_{\ell_{\xi_{1}}^{2}\ell_{\xi_{2}}^{2}}
		$
		holds for all $a$ and $F$} of  the operator $\mathcal{E}_N$ is denoted by $\mathcal{E}^*_N$-so called the \textbf{Fourier restriction operator}. Specifically, $\mathcal{E}_N^*: \cH \to \ell_{\xi_{1}}^{2}\ell_{\xi_{2}}^{2}$ is
	given by
	\begin{equation}\label{FRO}
		F \mapsto \mathcal{E}_N^* F (\xi_{1},\xi_{2}) = \begin{cases}
			\int_{I \times \T^{d}} F(t, z) e^{-2\pi i (z \cdot \xi + t |\xi|^{2})} \, dz dt  & \text{if} \quad \xi \in S_{d,N}, \\
			0 & \text{otherwise}.
		\end{cases}
	\end{equation}
	We note that a composition of $\mathcal{E}_N$ and $\mathcal{E}^*_N$  gives us a convolution operator $\cH$. In fact 
	\begin{flalign}\label{P1}
		\mathcal{E}_N \circ \mathcal{E}_N^* F(t,z) &= \sum_{\xi \in S_{d,N} } \mathcal{E}_N^*F(\xi) e^{2\pi i (z \cdot \xi + t|\xi|^{2})}  \nonumber \\
		&=\sum_{\xi \in S_{d,N} }\int_{I \times \T^{d}} F(t',z') e^{-2\pi i (z' \cdot \xi + t'|\xi|^{2})}e^{2\pi i (z \cdot \xi + t|\xi|^{2})}  \, dz' \, dt' \nonumber \\
		&=\int_{I \times \T^{d}} F(z',t')\sum_{\xi \in S_{d,N} }e^{2\pi i [(z-z') \cdot \xi + (t-t')|\xi|^{2}]}  \, dz' \, dt' \nonumber \\
		& =  \int_{I \times \T^d} K_N (z - z', t - t') F(z', t') \, dz' dt',
	\end{flalign}  
	where 
	\begin{equation}\label{kernel of extension}
		K_N (t,z) = \sum_{\xi \in S_{d,N} } e^{2\pi i (z \cdot \xi + t |\xi|^{2})} .
	\end{equation}
	\begin{remark}\label{usrev} We can restate the inequality \eqref{IN12} as follows: For every $N > 1$, any $\lambda \in \ell^{\alpha'}$, and any orthonormal system $(f_j)_j \subset L_{z}^2(\T^{d})$, the inequality \eqref{IN12} holds  if and only if
		\begin{equation}\label{P2}
			\left\| \sum_j \lambda_j |\mathcal{E}_N a_j|^2 \right\|_{L_t^q L_x^r L_{y}^{\Tilde{r}}(I \times \T^{d-k} \times \T^k)} \leq C_{|I|} N^{\frac{1}{q}} \| \lambda \|_{\ell^{\alpha'}}
		\end{equation}
		for any $N > 1$, $\lambda \in \ell^{\alpha'}$, and any ONS $(a_j)_j \subset \ell_{\xi_{1}}^{2}\ell_{\xi_{2}}^{2}.$ This follows from the observation that,  if we let $a_j = \widehat{f_j}$, then the orthonormality of $(f_j)_j$ in $L_{z}^2(\T^{d})$ is equivalent
		to the one of $(a_j)_j$ in $\ell_{\xi_{1}}^{2}\ell_{\xi_{2}}^{2}$ and $$e^{it \Delta} P_{\le N}f_j = \mathcal{E}_N a_j.$$
	\end{remark}

	\begin{proposition}\label{oseP}
		Let $1 \leq \Tilde{r} \leq r \leq \infty$ and $1 \leq \gamma < \frac{d+1}{d-1},$ where $\gamma$ defined in \eqref{gamma def}. Then for any triple $(q,r,\Tilde{r})$ satisfying \eqref{adc}, any $N > 1,$ any $\lambda \in \ell^{\frac{2\gamma}{\gamma+1}}, $ and any ONS $(a_{j})_{j}$ in $ \ell_{\xi_{1}}^{2}\ell_{\xi_{2}}^{2},$ 
		\begin{equation}\label{PR1}
			\left\| \sum_{j} \lambda_{j} |\mathcal{E}_{N}a_{j} |^{2} \right\|_{L_{t}^{q} (I_N, L_{x}^{r} L_{y}^{\Tilde{r}}(\T^{d-k} \times \T^{k}))} \leq C \| \lambda \|_{\ell^{\alpha'}},
		\end{equation}
		where $I_{N}=[\frac{-1}{2N},\frac{1}{2N}],$ $\mathcal{E}_{N}$ is defined in \eqref{FEP} and $\alpha'=\frac{2 \gamma}{\gamma+1}.$
	\end{proposition}
	\begin{proof}[Proof]  
		Since $1 \leq \Tilde{r} \leq r \leq \infty,$ this implies that
		\begin{equation*}
			\frac{1}{r}\left(1-\frac{k}{d} \right) \leq \frac{1}{\Tilde{r}}\left(1-\frac{k}{d} \right),
		\end{equation*}
		and hence
		\begin{equation*}
			\frac{1}{r}\left(1-\frac{k}{d} \right)+\frac{1}{\Tilde{r}}\frac{k}{d} \leq \frac{1}{\Tilde{r}}.
		\end{equation*}
		From the expression of $\gamma$ in \eqref{gamma def} and using the above inequality expression, we get
		\begin{equation*}
			\gamma \geq \Tilde{r}. 
		\end{equation*}
		Since $L^{\gamma} \hookrightarrow L^{\tilde{r}}$, to establish \eqref{PR1} it suffices to prove the following estimate:
		\begin{equation}\label{PR1 gamma}
			\left\| \sum_{j} \lambda_{j} |\mathcal{E}_{N}a_{j} |^{2} \right\|_{L_{t}^{q} (I_N, L_{x}^{r} L_{y}^{\gamma}(\T^{d-k} \times \T^{k}))} \leq C \| \lambda \|_{\ell^{\alpha'}}.
		\end{equation}
		For $\gamma=1,$ it is trivially true using Plancherel theorem and Minkowiski inequality.
		We will prove \eqref{PR1 gamma} for the range $\frac{1}{d+2} \leq \frac{1}{\alpha} < \frac{1}{d+1}.$
		
		For any $z=(x,y) \in \T^{d-k} \times \T^{k},$ we define $$\mathbf{1}_{I_{N}}(t,z)=\begin{cases}
			1 & \text{if} \quad  t \in  I_{N}, \\
			0  & \text{if} \quad t \notin  I_{N} .
		\end{cases}$$
		Put $\cH= L^2(I, L^2(\T^{d}))$ and $I=\T.$ By duality principle Lemma \ref{PL1}, to prove  the desired estimate \eqref{PR1 gamma} for all pair $(q,\gamma)$ satisfying \eqref{adc gamma} and $\frac{d+1}{d} \leq \gamma < \frac{d+1}{d-1}$ and for all $1 \leq r \leq \infty,$ it suffices to show that 
		\begin{equation}\label{PR2}
			\|W_{1}\mathbf{1}_{I_{N}}\mathcal{E}_{N} \mathcal{E}_{N}^{*}[\mathbf{1}_{I_{N}}W_{2}]\|_{\Sp^{2\gamma'} } \leq C \|W_{1} \|_{L_{t}^{2q'}L_{x}^{2r'}L_{y}^{2\gamma'}}
			\|W_{2} \|_{L_{t}^{2q'}L_{x}^{2r'}L_{y}^{2\gamma'}}.
		\end{equation}
		Put $\alpha =2\gamma'$ and $\beta=2q'.$  We may rewrite \eqref{PR2} as follows:
		\begin{equation}\label{PR3}
			\|W_{1}\mathbf{1}_{I_{N}}\mathcal{E}_{N} \mathcal{E}_{N}^{*}[\mathbf{1}_{I_{N}}W_{2}]\|_{\Sp^\alpha } \leq C \|W_{1} \|_{L_{t}^{\beta}L_{x}^{2r'}L_{y}^{\alpha}}
			\|W_{2} \|_{L_{t}^{\beta}L_{x}^{2r'}L_{y}^{\alpha}},
		\end{equation} 
		for all $\alpha,\beta \geq 1$ such that $\frac{2}{\beta}+ \frac{d}{\alpha}=1$ and $\frac{1}{d+2} \leq \frac{1}{\alpha} < \frac{1}{d+1}.$

		To this end, for $\epsilon > 0,$ we define a map $T_{N,\epsilon}:\cH  \to \cH $ by
		$$F \mapsto T_{N,\epsilon}F(t,z)=\int_{I \times \T^{d}} K_{N,\epsilon} (t - t', z-z') F( t', z') \, dz'  \, dt',$$
		where  
		\begin{align*}
			K_{N,\epsilon}(t,z)& =\mathbf{1}_{\epsilon < |t| <N^{-1}} (t,z) K_{N}(t,z)\\
			&=\mathbf{1}_{\epsilon < |t| <N^{-1}} (t,z)\sum_{\xi \in S_{d,N}} e^{2 \pi i (z \cdot \xi+ t|\xi|^{2})}.    
		\end{align*}
		Further, for $z_{1} \in \mathbb{C},$ with $\operatorname{Re}(z_{1}) \in [-1,\frac{d}{2}],$ we define
		\[T_{N,\epsilon}^{z_{1}}=K_{N,\epsilon}^{z_{1}}*,\]
		where 
		$K_{N,\epsilon}^{z_{1}}(t,z)=t^{z_{1}}K_{N,\epsilon}(t,z).$
		Since $\Sp^2$-norm is the Hilbert-Schmidt norm, we  first find the kernel of the operator $ W_1 \mathbf{1}_{I_{N}} T^{z_{1}}_{N,\epsilon} (\mathbf{1}_{I_{N}} W_2)$. In fact, for   $f \in \cH $, we have 
		\[
		\begin{aligned}
			W_1 \mathbf{1}_{I_{N}} T^{z_{1}}_{N,\epsilon} (\mathbf{1}_{I_{N}} W_2) f(t,z) 
			&= W_1(t,z)\mathbf{1}_{I_{N}}(t,z) K^{z_{1}}_{N,\epsilon} * (\mathbf{1}_{I_{N}} W_2 f)(t,z) \\
			&= \int_{I_N \times \T^{d}} K_1(t,z, t',z') f(t', z') \, dz' \, dt',
		\end{aligned}
		\]
		where kernel  $ K_1(t,z,t',z') = W_1(t,z) \mathbf{1}_{I_{N}}(t,z) K^{z_{1}}_{N,\epsilon} (t-t',z-z') W_2(t',z').$ By Proposition \ref{l8} and the fact that $2 \leq 2r',$ we have 
		\begin{flalign}\label{ds1}
			\|W_{1}\mathbf{1}_{I_{N}}T_{N,\epsilon}^{z_{1}}[\mathbf{1}_{I_{N}}W_{2}]\|_{\Sp^{2}}^{2}  =\left\|K_1\right\|_{\cH \times \cH}^{2} \nonumber 
			& = \int_{I} \int_{I} \frac{\|W_{1}(t)\|_{L_{x}^{2}L_{y}^{2}}^{2}\|W_{2}(t')\|_{L_{x}^{2}L_{y}^{2}}^{2}}{|t-t'|^{d-2\operatorname{Re}(z_{1})}} dt dt' \nonumber \\
			& \lesssim \int_{I} \int_{I} \frac{\|W_{1}(t)\|_{L_{x}^{2r'}L_{y}^{2}}^{2}\|W_{2}(t')\|_{L_{x}^{2r'}L_{y}^{2}}^{2}}{|t-t'|^{d-2\operatorname{Re}(z_{1})}} dt dt'.
		\end{flalign} 
		Now, by applying Lemma \ref{PL3} (1D Hardy--Littlewood--Sobolev inequality), we obtain the condition
		\begin{eqnarray}\label{c1}
			\frac{1}{u}=\frac{1}{2}+\frac{1}{2}\Big(\operatorname{Re}(z_{1})-\frac{d}{2}\Big) \in \left(\frac{1}{4},\frac{1}{2}\right].
		\end{eqnarray}
		Assuming that \eqref{c1} holds, it follows that
		\begin{flalign}\label{PR4}
			\|W_{1}\mathbf{1}_{I_{N}}T_{N,\epsilon}^{z_{1}}[\mathbf{1}_{I_{N}}W_{2}]\|_{\Sp^{2}}^{2}
			& =
			C' \|W_{1} \|^2_{L_{t}^{u}L_{x}^{2r'}L_{y}^{2}}
			\|W_{2} \|^2_{L_{t}^{u}L_{x}^{2r'}L_{y}^{2}}.
		\end{flalign}
		
		For $\operatorname{Re}(z_{1})=-1$, using the same argument as in \cite[Proposition 3.3]{nakamura2020orthonormal} and \cite[Proposition 5.1]{bhimani2025orthonormal}, the operator
		$T_{N,\epsilon}^{z_{1}}:\cH \to \cH$ is bounded, with operator norm depending exponentially on $\operatorname{Im}(z_{1})$, namely,
		\begin{eqnarray}\label{eb}
			\|T_{N,\epsilon}^{z_{1}}\|_{\cH \to \cH} \leq C (\text{Im}(z_{1})).
		\end{eqnarray}
		Let $f \in \cH,$ and $\|f\|_{\cH} \leq 1,$ then
		\begin{align*}
			\|W_{1}\mathbf{1}_{I_{N}}T_{N,\epsilon}^{z_{1}}[\mathbf{1}_{I_{N}}W_{2}]f\|_{\cH} &\leq \|W_{1}\|_{L_{t}^{1}L_{x}^{1}L_{y}^{1}} \|T_{N,\epsilon}^{z_{1}}([\mathbf{1}_{I_{N}}W_{2}]f)\|_{L_{t}^{1}L_{x}^{1}L_{y}^{1}} \\
			& \lesssim \|W_{1}\|_{L_{t}^{\infty}L_{x}^{2r'}L_{y}^{\infty}} \|T_{N,\epsilon}^{z_{1}}([\mathbf{1}_{I_{N}}W_{2}]f)\|_{L_{t}^{2}L_{x}^{2}L_{y}^{2}} \\
			& \lesssim \|W_{1}\|_{L_{t}^{\infty}L_{x}^{2r'}L_{y}^{\infty}} \|T_{N,\epsilon}^{z_{1}}\|_{\cH \to \cH} \|[\mathbf{1}_{I_{N}}W_{2}]f\|_{L_{t}^{2}L_{x}^{2}L_{y}^{2}} \\
			& \lesssim C(\operatorname{Im}(z_{1})) \|W_{1}\|_{L_{t}^{\infty}L_{x}^{2r'}L_{y}^{\infty}}  \|\mathbf{1}_{I_{N}}W_{2}\|_{L_{t}^{1}L_{x}^{1}L_{y}^{1}} \|f\|_{L_{t}^{1}L_{x}^{1}L_{y}^{1}} \\
			& \lesssim C(\operatorname{Im}(z_{1})) \|W_{1}\|_{L_{t}^{\infty}L_{x}^{2r'}L_{y}^{\infty}}  \|W_{2}\|_{L_{t}^{\infty}L_{x}^{2r'}L_{y}^{\infty}} \|f\|_{\cH} \\
			& \lesssim C(\operatorname{Im}(z_{1})) \|W_{1}\|_{L_{t}^{\infty}L_{x}^{2r'}L_{y}^{\infty}}  \|W_{2}\|_{L_{t}^{\infty}L_{x}^{2r'}L_{y}^{\infty}}.
		\end{align*}       
		Thus, for $\operatorname{Re}(z_{1})=-1,$ we have  
		\begin{align}\label{PR6}
			\|W_{1}\mathbf{1}_{I_{N}}T_{N,\epsilon}^{z_{1}}[\mathbf{1}_{I_{N}}W_{2}]\|_{\Sp^{\infty}} &=\|W_{1}\mathbf{1}_{I_{N}}T_{N,\epsilon}^{z_{1}}[\mathbf{1}_{I_{N}}W_{2}]\|_{\cH \to \cH} \nonumber \\
			& \leq C(\operatorname{Im}(z_{1}))\|W_{1}\|_{L_{t}^{\infty}L_{x}^{2r'}L_{y}^{\infty}} \|W_{2}\|_{L_{t}^{\infty}L_{x}^{2r'}L_{y}^{\infty}}.
		\end{align}
		We may choose $ \tau=\operatorname{Re}(z_{1})/(1+ \operatorname{Re}(z_{1})) \in [0,1]$ in Lemma \ref{s-interpolation},  and interpolate between \eqref{PR4} and \eqref{PR6},  to obtain
		\begin{flalign*}
			\|W_{1}\mathbf{1}_{I_{N}}T_{N,\epsilon}^{0}[\mathbf{1}_{I_{N}}W_{2}]\|_{\Sp^{2(\operatorname{Re}(z_{1})+1)}}  \lesssim \|W_{1} \|_{L_{t}^{u(\operatorname{Re}(z_{1})+1)}L_{x}^{2r'}L_{y}^{2(\operatorname{Re}(z_{1})+1)}} \|W_{2} \|_{L_{t}^{u(\operatorname{Re}(z_{1})+1)}L_{x}^{2r'}L_{y}^{2(\operatorname{Re}(z_{1})+1)}},
		\end{flalign*} 
		where the constant depends on $\theta$ and independent of $\epsilon.$
		Recalling $T^{0}_{N, \epsilon}=T_{N, \epsilon}$ and taking $\epsilon \to 0,$ we get
		\begin{flalign*}
			\|W_{1}\mathbf{1}_{I_{N}}T_{N}[\mathbf{1}_{I_{N}}W_{2}]\|_{\Sp^{2(\operatorname{Re}(z_{1})+1)}}  \lesssim  \|W_{1} \|_{L_{t}^{u(\operatorname{Re}(z_{1})+1)}L_{x}^{2r'}L_{y}^{2(\operatorname{Re}(z_{1})+1)}} \|W_{2} \|_{L_{t}^{u(\operatorname{Re}(z_{1})+1)}L_{x}^{2r'}L_{y}^{2(\operatorname{Re}(z_{1})+1)}}.
		\end{flalign*}
		Put  $\beta=u(\operatorname{Re}(z_{1})+1)$ and  $\alpha=2(\operatorname{Re}(z_{1})+1)$. Note that condition \eqref{c1} is compatible with the following condition 
		\begin{eqnarray}\label{c2}
			\frac{2}{\beta}+\frac{d}{  \alpha}=1 \quad \text{ and } \quad \frac{1}{ (d+2)}\leq \frac{1}{\alpha} < \frac{1}{d+1}.    
		\end{eqnarray}
		Thus the desired inequality  \eqref{PR3} holds as long as $\alpha$ and $\beta$ satisfies \eqref{c2}.
	\end{proof}
	
	\begin{proof}[\textbf{Proof of Theorem \ref{ose}}]
		Let  $I'_{N}$  be an arbitrary interval whose length is $N^{-1}.$
		Denote $c(I'_{N})$ by  the center of the interval $I'_{N}$. We wish to  shift this interval $I'_N$ to the interval   $I_{N}=[\frac{-1}{2N},\frac{1}{2N}]$ with centre origin. Thus, by the change of variables, we get 
		\begin{equation}\label{ts1}
			\left\|\sum_{j} \lambda_{j}|\mathcal{E}_{N}a_{j}|^{2}\right\|_{L^{q}_{t}L^{r}_{x}L_{y}^{\Tilde{r}}( I'_{N} \times \T^{d-k} \times \T^{k})}=\left\|\sum_{j} \lambda_{j}|\mathcal{E}_{N}b_{j}|^{2}\right\|_{L^{q}_{t}L^{r}_{x}L_{y}^{\Tilde{r}}( I_{N} \times \T^{d-k} \times \T^{k})},
		\end{equation}
		where $b_{j}(\xi)=a_{j}(\xi)e^{-2 \pi ic(I'_{N}) |\xi|^{2} }.$
		Note that if $(a_{j})_{j}$ is orthonormal, then $(b_{j})_{j}$ is also orthonormal in $\ell^{2}$. In view of the above identity,  Proposition \ref{oseP} is applicable to any time interval whose length is $1/N$.
		Since $I$  is an interval of finite length, we may cover it by taking finitely many intervals of length, $1/N,$ say $I\subset \bigcup_{i=1}^{N}I_{i}.$ 
		Now taking  Remark \ref{usrev}, \eqref{ts1} and Proposition \ref{oseP} into account, we obtain
		$$   \left\|\sum_{j} \lambda_{j}|\mathcal{E}_{N}a_{j}|^{2}\right\|_{L^{q}_{t}L^{r}_{x}L_{y}^{\Tilde{r}}}^{q}   \lesssim  \sum_{i=1}^{N}\left\|\sum_{j} \lambda_{j}|\mathcal{E}_{N}a_{j}|^{2}\right\|_{L^{q}_{t}L^{r}_{x}L_{y}^{\Tilde{r}}}^{q}
		\lesssim N \|\lambda\|_{\ell^{\alpha'}}^q.
		$$ 
		This completes the proof of  \eqref{IN12} if $\alpha' \leq  \frac{2\gamma}{\gamma+1}$.
	\end{proof}
	
	\section{Well-posedness to infinite system of  Hartree equations}\label{sec7}
	This section focuses on establishing the well-posedness of the Hartree equation \eqref{Hartree} for infinitely many fermions as an application of our orthonormal inequality \eqref{PR1} in the partial frame.  Throughout this section, we assume that $q,r,\tilde r,s$ and $\sigma$ satisfies the same conditions as in
	Theorem \ref{ose}. Note that the conclusions of Theorems~\ref{ose} (by the same proof) remain valid if $P_{\le N}$ is replaced by $P_N$ for any
	$N\in 2^{\mathbb N}$. As a consequence, the following estimate
	\begin{equation}\label{WL1}
		\bigg\| \sum_{j}\lambda_j 
		\,\big| e^{it\Delta} P_N f_j \big|^2 
		\bigg\|_{L^{q}_{t}L_x^{r}L_y^{\tilde r}}
		\le 
		C_{\sigma}\, N^{\sigma}\, \|\lambda\|_{\ell^{\alpha'}}
	\end{equation}
	holds true for any $N\in 2^{\mathbb N}$, where $(f_j)_j$ is an orthonormal system in
	$L^2(\mathbb T^{d-k}\times \mathbb T^{k})$.

	As a consequence of the above estimate \eqref{WL1} along with the vector-valued version of Bernstein’s inequality in Corollary \ref{vvr}, we have the following result. 
	\begin{corollary}
		For any $\varepsilon>0$ and any orthonormal system $(f_j)_j$ in
		$L^2(\mathbb T^{d-k}\times \mathbb T^{k})$, we have the following estimate:
		\begin{equation}\label{WL2}
			\bigg\| \sum_{j}\lambda_j 
			\big| e^{it\Delta}
			\langle \nabla \rangle^{-\frac{\sigma}{2}-\varepsilon} f_j \big|^2 
			\bigg\|_{L^{q}_{t}L_x^{r}L_y^{\tilde r}}
			\le 
			C_{\sigma,\varepsilon} \|\lambda\|_{\ell^{\alpha'}}.
		\end{equation}
	\end{corollary}
	
	\begin{proof}
		By the  Littlewood-Paley theorem for densities of operators in  \eqref{eq:operator_density_LP}, together with the triangle
		inequality, Corollary \ref{vvr}, and \eqref{WL1}, we obtain
		\begin{flalign*}
			\bigg\| \sum_{j}\lambda_j 
			\big| e^{it\Delta}
			\langle \nabla \rangle^{-\frac{\sigma}{2}-\varepsilon} f_j \big|^2 
			\bigg\|_{L^{q}_{t}L_x^{r}L_y^{\tilde r}}
			&\lesssim 
			\bigg\| \sum_j \lambda_j 
			\big| \langle \nabla\rangle^{-\frac{\sigma}{2}-\varepsilon}
			e^{it\Delta} P_1 f_j \big|^2 
			\bigg\|_{L^{q}_{t}L_x^{r}L_y^{\tilde r}}
			\\
			&\qquad +
			\sum_{N\in 2^{\mathbb N}\setminus\{1\}}
			\bigg\| \sum_{j}\lambda_j 
			\big| \langle \nabla\rangle^{-\frac{\sigma}{2}-\varepsilon} 
			e^{it\Delta} P_N f_j \big|^2 
			\bigg\|_{L^{q}_{t}L_x^{r}L_y^{\tilde r}}
			\\
			&\lesssim_{\sigma,\varepsilon}
			\|\lambda\|_{\ell^{\alpha'}}
			+
			\sum_{N>1} N^{-(\sigma+\varepsilon)}
			\bigg\| \sum_{j}\lambda_j 
			\big| e^{it\Delta} P_N f_j \big|^2 
			\bigg\|_{L^{q}_{t}L_x^{r}L_y^{\tilde r}}
			\\
			&\lesssim_{\sigma,\varepsilon} 
			\|\lambda\|_{\ell^{\alpha'}},
		\end{flalign*}
		and this completes the proof of \eqref{WL2}.
	\end{proof}
	
	Now	consider the operator
	$\gamma_0 :=\sum_{j} \lambda_{j}\left|f_{j}\right\rangle\left\langle f_{j}\right|$ for  an  ONS \( (f_j)_{j} \) in \( L^{2} (\T^d) \) such that
	$
	\| \gamma_{0} \|_{\Sp^{\alpha'} (L^{2}(\T^d))} = \| \lambda \|_{\ell^{\alpha'}} \leq 1.$ Since  the sequence \(f_{j}\) form an orthonormal system therefore,
	the scalers, \(\lambda_{j}\) are precisely the eigenvalues of the operator $\gamma_0$. The evolved operator
	$$\gamma(t) :=e^{-i t \Delta} \gamma_0 e^{i t \Delta}=\sum_{j} \lambda_{j}\left|e^{-i t \Delta} f_{j}\right\rangle\left\langle e^{-i t \Delta} f_{j}\right| .$$
	Let us introduce the density  
	\begin{align}\label{y}
		\rho_{\gamma(t)} :=\sum_{j} \lambda_{j}\left|e^{-i t \Delta} f_{j}\right|^{2}.
	\end{align} 
	Here, we focus specifically on the density   $$
	\rho_{\langle \nabla \rangle^{-s}\gamma(t)\langle \nabla \rangle^{-s}}
	=
	\sum_{j} \lambda_j
	\left|
	e^{it\Delta}\langle \nabla \rangle^{-s} f_j
	\right|^2,
	$$
	associated with the evolve operator $$\langle \nabla \rangle^{-s} \gamma(t)\langle \nabla \rangle^{-s}=e^{-it\Delta}\langle \nabla \rangle^{-s}\gamma_{0}\langle \nabla \rangle^{-s}e^{it\Delta}.$$ 
	Moreover, for any multiplication operator $V$ on $L^{2}(\mathbb T^{d-k}\times\mathbb T^{k})$, we have the well-known identity (see \cite{nakamura2020orthonormal})
	\[
	\operatorname{Tr}(V\gamma)
	=
	\int_{\mathbb T^{d-k}\times\mathbb T^{k}} V(z)\rho_{\gamma}(z)\,dz.
	\]
	Then by the cyclic property of trace and H\"older's inequality,  for a time-dependent potential $V(t, z)$, the density function $\rho_{\gamma_s(t)}$  satisfies 
	\begin{flalign*}
		&\left|
		\operatorname{Tr}_{L^{2} }
		\left(
		\gamma_{0}
		\int_{I}
		e^{it\Delta}
		\langle \nabla \rangle^{-s}
		V(t,z)
		\langle \nabla \rangle^{-s}
		e^{-it\Delta}
		\,dt
		\right)
		\right|
		\\
		&=\left| \int_{I} \operatorname{Tr}_{L^{2} } \left( V\, e^{-it\Delta} \langle \nabla \rangle^{-s} \gamma_{0} \langle \nabla \rangle^{-s} e^{it\Delta} \right) \,dt \right| \\
		&=
		\left|
		\int_{I}
		\int_{\mathbb T^{d-k}\times\mathbb T^{k}}
		V(z)\,
		\rho_{\gamma_s(t)}(z)
		\,dz\,dt
		\right|
		\\
		&\le
		\| V \|_{L_t^{q'}L_x^{r'}L_y^{\tilde r'}}
		\left\|
		\rho_{\gamma_s(t)}
		\right\|_{L_t^{q}L_x^{r}L_y^{\tilde r}}
		\\
		&=
		\| V \|_{L_t^{q'}L_x^{r'}L_y^{\tilde r'}}
		\left\|
		\sum_{j} \lambda_j
		\left|
		e^{it\Delta}
		\langle \nabla \rangle^{-s}
		f_j
		\right|^2
		\right\|_{L_t^{q}L_x^{r}L_y^{\tilde r}}.
	\end{flalign*} 
	Thus, by the duality argument and \eqref{WL2} with $s=\frac{\sigma}{2}+\varepsilon$, we get 
	\begin{equation}\label{dual version}
		\left\|
		\int_{I}
		e^{it\Delta}
		\langle \nabla \rangle^{-s}
		V(t,z)
		\langle \nabla \rangle^{-s}
		e^{-it\Delta}
		\,dt
		\right\|_{\Sp^{\alpha}\!\left(L_z^{2}(\mathbb T^{d-k}\times\mathbb T^{k})\right)}
		\le
		C_{|I|,\sigma,\varepsilon}
		\| V \|_{L_t^{q'}L_x^{r'}L_y^{\tilde r'}}.
	\end{equation}

	Let $R(t') : L^{2}(\mathbb T^{d-k}\times\mathbb T^{k}) \to L^{2}(\mathbb T^{d-k}\times\mathbb T^{k})$ be self-adjoint for each $t'\in I$, and define
	\[
	\gamma(t)
	=
	\int_{0}^{t}
	e^{-i(t-t')\Delta}
	R(t')
	e^{i(t-t')\Delta}
	\,dt',
	\qquad t\in I.
	\] 
	We also denote
	\[
	A(t)
	=
	e^{it\Delta}
	\langle \nabla \rangle^{-s}
	|V(t,z)|
	\langle \nabla \rangle^{-s}
	e^{-it\Delta},
	\quad \text{and}\quad 
	B(t')
	=
	e^{it'\Delta}
	|R(t')|
	e^{-it'\Delta}.
	\]
	Then for all time-dependent potential $V$ with \(
	\| V \|_{L_t^{q'} L_x^{r'} L_y^{\tilde r'}(I \times \mathbb T^{d-k}\times\mathbb T^{k})}
	\le 1
	\), using the dual inequality \eqref{dual version} and an argument similar to that in \cite{nakamura2020orthonormal}, we get  
	\begin{align*}
		&\left| \int_{I} \int_{\mathbb{T}^{d-k} \times \mathbb{T}^{k}} V(t,z) \rho_{\langle \nabla \rangle^{-s} \gamma{(t)} \langle \nabla \rangle^{-s}}(z) \, dz \, dt \right| \\
		&=\left|\int_{I} \int_{0}^{t} \operatorname{Tr}_{L^{2} } \left( e^{it\Delta} \langle \nabla \rangle^{-s} V(t,z) \langle \nabla \rangle^{-s} e^{-it\Delta} \,e^{it'\Delta} R(t') e^{-it'\Delta} \right) \,dt'\,dt \right| \\
		&\leq \operatorname{Tr}_{L^{2} } \left( \left( \int_{I} A(t) \, dt \right) \left(\int_{I} B(t') dt' \right) \right) \\
		&\leq \left \| \int_{I} A(t) \, dt \right\|_{\Sp^{\alpha}(L^{2} (\mathbb{T}^{d-k} \times \mathbb{T}^{k}))} \left \| \int_{I} B(t') \, dt' \right\|_{\Sp^{\alpha'}(L^{2} (\mathbb{T}^{d-k} \times \mathbb{T}^{k}))}\\
		&\leq C_{\abs{I},\theta, \sigma, \varepsilon} \| V \|_{L_t^{q'}L_x^{r'}L_y^{\tilde r'}} \left \| \int_{I} B(t') \, dt' \right\|_{\Sp^{\alpha'}(L^{2} (\mathbb{T}^{d-k} \times \mathbb{T}^{k}))},
	\end{align*}
	where in the second last line, we used H\"older’s inequality for traces. Again by a the duality argument
	\begin{equation*}
		\left\|
		\rho_{\langle \nabla \rangle^{-s}\gamma(t)\langle \nabla \rangle^{-s}}
		\right\|_{L_t^{q}L_x^{r}L_y^{\tilde r}}
		\le
		C_{|I|,\sigma,\varepsilon}
		\left\|
		\int_{I}
		e^{it'\Delta}
		|R(t')|
		e^{-it'\Delta}
		\,dt'
		\right\|_{\Sp^{\alpha'}\!\left(L_z^{2}(\mathbb T^{d-k}\times\mathbb T^{k})\right)}.
	\end{equation*}
	In summary, combining the properties discussed above, we arrive at the following result.
	\begin{proposition}\label{operator inhomogeneous estimate}\
		\begin{enumerate}
			\item \label{ss1}
			The orthonormal inequality \eqref{WL2} is equivalent to for any
			$V\in L_t^{q'}L_x^{r'}L_y^{\tilde r'}$
			and $s=\frac{\sigma}{2}+\varepsilon,$
			\begin{equation}\label{WL3}
				\left\|
				\int_{I}
				e^{it\Delta}
				\langle \nabla \rangle^{-s}
				V(t,z)
				\langle \nabla \rangle^{-s}
				e^{-it\Delta}
				\,dt
				\right\|_{\Sp^{\alpha}\!\left(L_z^{2}(\mathbb T^{d-k}\times\mathbb T^{k})\right)}
				\le
				C_{|I|,\sigma,\varepsilon}
				\| V \|_{L_t^{q'}L_x^{r'}L_y^{\tilde r'}}.
			\end{equation}
			
			\item \label{ss2} (Inhomogeneous estimate)
			Let
			$R(t') : L^{2}(\mathbb T^{d-k}\times\mathbb T^{k})
			\to L^{2}(\mathbb T^{d-k}\times\mathbb T^{k})$
			be self-adjoint for each $t'\in I$ and define
			\[
			\gamma(t)
			=
			\int_{0}^{t}
			e^{-i(t-t')\Delta}
			R(t')
			e^{i(t-t')\Delta}
			\,dt',
			\qquad t\in I.
			\]
			If \eqref{WL2} holds true, then
			\begin{equation}\label{WL4}
				\left\|
				\rho_{\langle \nabla \rangle^{-s}\gamma(t)\langle \nabla \rangle^{-s}}
				\right\|_{L_t^{q}L_x^{r}L_y^{\tilde r}}
				\le
				C_{|I|,\sigma,\varepsilon}
				\left\|
				\int_{I}
				e^{is\Delta}
				|R(s)|
				e^{-is\Delta}
				\,ds
				\right\|_{\Sp^{\alpha'}\!\left(L_z^{2}(\mathbb T^{d-k}\times\mathbb T^{k})\right)}.
			\end{equation}
		\end{enumerate}
	\end{proposition}
	Now, we are in a position to prove the well-posedness of the Hartree equation for infinitely many fermions.
	\begin{proof}[\textbf{Proof of Theorem \ref{TIN7} \eqref{TIN71}}] Denote $\Sp^{\alpha',s}=\Sp^{\alpha',s}(L^2 (\mathbb{T}^{d-k} \times \mathbb{T}^{k}))$ (see definition in \eqref{sobolev schatten}).
Let \(R > 0\) such that 
$\| \gamma_{0} \|_{\Sp^{\alpha', s} } < R,$ and 
$T = T(R)>0$ to be chosen later. Set 
\[
X_T
=
\left\{
(\gamma,\rho)
\in
C\!\left([0,T],\Sp^{\alpha',s}\right)
\times
L_t^{q}L_x^{r}L_y^{\tilde r}\!\left([0,T]\times\mathbb T^{d-k}\times\mathbb T^{k}\right)
:
\|(\gamma,\rho)\|_{X_T}
\le
C^{*}R
\right\},
\]
where 
\[
\| (\gamma, \rho) \|_{X_T}
=
\| \gamma \|_{C([0,T],\Sp^{\alpha',s})}
+
\| \rho \|_{L_t^{q}L_x^{r}L_y^{\tilde r}([0,T]\times\mathbb T^{d-k}\times\mathbb T^{k})}.
\]
Here \(C^* > 0\) is to chosen later independent of \(R\). We define
\[
\Phi(\gamma, \rho) = \left( \Phi_1(\gamma, \rho), \rho \left[ \Phi_1(\gamma, \rho) \right] \right),
\]
\[
\Phi_1(\gamma, \rho)(t)
=
S(t)
+
\int_{0}^{t}
e^{-i(t-t')\Delta}
\left[
\omega \ast \rho(t'),
\gamma(t')
\right]
e^{i(t-t')\Delta}
\,dt',
\]
where $S(t)=e^{-it\Delta}\,\gamma_{0}\,e^{it\Delta}$ and \(\rho[\gamma] = \rho _{\gamma}\) (the density function of $\gamma$). The solution of \eqref{Hartree} is fixed point of \(\Phi,\) i.e.,
\( \Phi(\gamma, \rho _{\gamma}) = (\gamma, \rho _{\gamma}).\)
Fix any \(t \in [0, T]\) and 
\begin{align*}
    &\| \Phi_1(\gamma, \rho)(t) \|_{\Sp^{\alpha', s}} \\ 
    &\leq \| S(t) \|_{\Sp^{\alpha', s}}
    + \int_{0}^{T}
    \| e^{-i(t-t')\Delta}
    \left[ \omega \ast \rho(t'), \gamma(t') \right]
    e^{i(t-t')\Delta}
    \|_{\Sp^{\alpha', s}} \, dt'.
\end{align*}
Since $(f_{j})_{j}$ is ONS in $L^{2} \!\left(\mathbb T^{d-k}\times\mathbb T^{k}\right)
,$ then $\bigl(e^{-it\Delta} f_j\bigr)_{j}
$ is as well for each $t$ and so
\[
\| e^{-it\Delta}\,\gamma_{0}\,e^{it\Delta} \|_{\Sp^{\alpha', s}}
=
\| \gamma_{0} \|_{\Sp^{\alpha', s}}
<
R.
\]Denote $\Sp^{\alpha',s}=\Sp^{\alpha',s}(L^2 (\mathbb{T}^{d-k} \times \mathbb{T}^{k}))$ (see definition in \eqref{sobolev schatten}).
Let \(R > 0\) such that 
$\| \gamma_{0} \|_{\Sp^{\alpha', s} } < R,$ and 
$T = T(R)>0$ to be chosen later. Set 
\[
X_T
=
\left\{
(\gamma,\rho)
\in
C\!\left([0,T],\Sp^{\alpha',s}\right)
\times
L_t^{q}L_x^{r}L_y^{\tilde r}\!\left([0,T]\times\mathbb T^{d-k}\times\mathbb T^{k}\right)
:
\|(\gamma,\rho)\|_{X_T}
\le
C^{*}R
\right\},
\]
where 
\[
\| (\gamma, \rho) \|_{X_T}
=
\| \gamma \|_{C([0,T],\Sp^{\alpha',s})}
+
\| \rho \|_{L_t^{q}L_x^{r}L_y^{\tilde r}([0,T]\times\mathbb T^{d-k}\times\mathbb T^{k})}.
\]
Here \(C^* > 0\) is to chosen later independent of \(R\). We define
\[
\Phi(\gamma, \rho) = \left( \Phi_1(\gamma, \rho), \rho \left[ \Phi_1(\gamma, \rho) \right] \right),
\]
\[
\Phi_1(\gamma, \rho)(t)
=
S(t)
+
\int_{0}^{t}
e^{-i(t-t')\Delta}
\left[
\omega \ast \rho(t'),
\gamma(t')
\right]
e^{i(t-t')\Delta}
\,dt',
\]
where $S(t)=e^{-it\Delta}\,\gamma_{0}\,e^{it\Delta}$ and \(\rho[\gamma] = \rho _{\gamma}\) (the density function of $\gamma$). The solution of \eqref{Hartree} is fixed point of \(\Phi,\) i.e.,
\( \Phi(\gamma, \rho _{\gamma}) = (\gamma, \rho _{\gamma}).\)
Fix any \(t \in [0, T]\) and 
\begin{align*}
    &\| \Phi_1(\gamma, \rho)(t) \|_{\Sp^{\alpha', s}} \\ 
    &\leq \| S(t) \|_{\Sp^{\alpha', s}}
    + \int_{0}^{T}
    \| e^{-i(t-t')\Delta}
    \left[ \omega \ast \rho(t'), \gamma(t') \right]
    e^{i(t-t')\Delta}
    \|_{\Sp^{\alpha', s}} \, dt'.
\end{align*}
Since $(f_{j})_{j}$ is ONS in $L^{2} \!\left(\mathbb T^{d-k}\times\mathbb T^{k}\right)
,$ then $\bigl(e^{-it\Delta} f_j\bigr)_{j}
$ is as well for each $t$ and so
\[
\| e^{-it\Delta}\,\gamma_{0}\,e^{it\Delta} \|_{\Sp^{\alpha', s}}
=
\| \gamma_{0} \|_{\Sp^{\alpha', s}}
<
R.
\]
		For the second term, again we, use the inequality, we get
\begin{align*}
    &\| e^{-i(t-t')\Delta} \left[ \omega * \rho(t'), \gamma(t') \right] e^{i(t-t')\Delta} \|_{\Sp^{\alpha', s}}\\
    &\leq \big\{ \left\| \langle \nabla \rangle^s \omega * \rho(t') \langle \nabla \rangle^{-s} \right\|_{\Sp^{\infty} }
    +
    \left\| \langle \nabla \rangle^{-s} \omega \ast \rho(t') \langle \nabla \rangle^{s} \right\|_{\Sp^{\infty} } \big\}\| \gamma(t') \|_{\Sp^{\alpha', s}}.
\end{align*}
The estimate we employ to evaluate the above nonlinear term is the following (see in \cite[Corollary on p. 205]{HansII}), where the inequality was proved for the \(\mathbb{R}^d\) case, but the same proof is applicable in the case of $\mathbb T^{d-k}\times\mathbb T^{k}$:
\[
\| f \cdot g \|_{H^{r}(\mathbb T^{d-k}\times\mathbb T^{k})} \leq C_{s, \delta} \| f \|_{B_{\infty, \infty}^{|r|+\delta}(\mathbb T^{d-k}\times\mathbb T^{k})} \| g \|_{H^{r}(\mathbb T^{d-k}\times\mathbb T^{k})},
\]
where \(r \in \mathbb{R}\) and \(\delta > 0\) are arbitrary.
From the above estimate and Young's inequality, we obtain
\begin{align*}
    \left\| \langle \nabla \rangle^s \omega \ast \rho(t') \langle \nabla \rangle^{-s} \right\|_{\Sp^{\infty} } 
    &\leq C'_{s, \delta} \| \omega \ast \rho(t') \|_{B^{s+\delta}_{\infty, \infty}(\mathbb T^{d-k}\times\mathbb T^{k})}\\
    &\leq C'_{s, \delta} \| \omega \|_{B^{s+\delta}_{(r',\tilde{r}'), \infty}(\mathbb T^{d-k}\times\mathbb T^{k})}
    \| \rho(t') \|_{L^r_x L_y^{\tilde{r}}(\mathbb T^{d-k}\times\mathbb T^{k})}.
\end{align*}

Similarly,
\[
\| \langle \nabla \rangle^{-s} \omega \ast \rho(t') \langle \nabla \rangle^s \|_{\Sp^{\infty}} \leq C'_{-s, \delta} \| \omega \|_{B^{s+\delta}_{(r',\tilde{r}'), \infty}(\mathbb T^{d-k}\times\mathbb T^{k})}
\| \rho(t') \|_{L^r_x L_y^{\tilde{r}}(\mathbb T^{d-k}\times\mathbb T^{k})}.
\]
In total from \((\gamma, \rho) \in X_T\), we estimate
$$ 
\int_0^T
\| e^{-i(t-t')\Delta}
\nabla^{s}
[ \omega \ast \rho(t'), \gamma(t')]
e^{i(t-t')\Delta}
\|_{\Sp^{\alpha', s}}
\, dt'
\le
C_{s,\delta}
\| \omega \|_{B^{s+\delta}_{(r,\tilde r),\infty}(\mathbb T^{d-k}\times\mathbb T^{k})}
T^{1/q'}
(C^{*}R)^2,
$$
where \(C_{s, \delta} = C'_{s, \delta} + C'_{-s, \delta}.\)
Therefore
\begin{eqnarray}\label{WL5}
    \| \Phi_{1} (\gamma, \rho) \|_{C \left([0, T] ;\Sp^{\alpha', s} \right)}
    \leq
    R
    +
    c_{s,\delta}
    T^{\frac{1}{q'}}
    \| \omega \|_{B^{s+\delta}_{(r,\tilde r), \infty}(\mathbb T^{d-k}\times\mathbb T^{k})}
    (C^*R)^2.
\end{eqnarray}
		Similarly, using Proposition \ref{operator inhomogeneous estimate}, we get
\begin{align*}
    &T^{-\frac{1}{p}} C^{-1}_{\sigma,\varepsilon}
    \|\rho[ \Phi_{1} (\gamma, \rho) ] \|_{L_t^{q}L_x^{r}L_y^{\tilde r}([0,T]\times\mathbb T^{d-k}\times\mathbb T^{k})} \\
    &\leq
    \| \langle \nabla \rangle^s \gamma_{0} \langle \nabla \rangle^s \|_{\Sp^{\alpha'} }
    +
    \left\|
    \int_0^T
    e^{it'\Delta}
    \langle \nabla \rangle^s
    \left| [ \omega \ast \rho(t') , \gamma(t')] \right|
    \langle \nabla \rangle^s
    e^{-it'\Delta}
    \, dt'
    \right \|_{\Sp^{\alpha'} }.
\end{align*}
For the first term,
\[
\| \langle \nabla \rangle^s \gamma_{0} \langle \nabla \rangle^s \|_{\Sp^{\alpha'} } = \| \gamma_{0} \|_{\Sp^{\alpha', s} } \leq R.
\]
For the second term, we may employ the same argument as for
$ 
\| \Phi_{1} (\gamma, \rho) \|_{C \left([0, T] ;\Sp^{\alpha', s} \right)},
$ 
and we get
\begin{eqnarray}\label{WL6}
    \| \rho [ \Phi_1 (\gamma, \rho) ] \|_{L_t^{q}L_x^{r}L_y^{\tilde r}([0,T]\times\mathbb T^{d-k}\times\mathbb T^{k})}
    \le
    C_{\sigma,\varepsilon}
    T^{\frac{1}{q}}
    \left\{
    R
    +
    C_{s,\delta}
    T^{\frac{1}{q'}}
    \| \omega \|_{B^{s+\delta}_{(r,\tilde{r}), \infty}(\mathbb T^{d-k}\times\mathbb T^{k})}
    (C^{*}R)^2
    \right\}.
\end{eqnarray}
Combining these two estimates, we have 
\[
\|\Phi(\gamma, \rho) \|_{X_{T}} \leq (1+C_{ \sigma, \epsilon} T^{1/q}) \left\{R + C_{s, \delta} T^{\frac{1}{q'}} \| \omega \|_{B^{s+\delta}_{(r,\tilde{r}), \infty}(\mathbb T^{d-k}\times\mathbb T^{k})} (C^* R)^2 \right\}.
\]
This shows that $\Phi :X_{T} \to X_{T}$ if we take $C^*$ large and $T$ (if $T$ relies on $\| \omega \|_{B^{s+\delta}_{(r,\tilde{r}), \infty}(\mathbb T^{d-k}\times\mathbb T^{k})}$, not $\| \omega \|_{B^{s}_{(r,\tilde{r}), \infty}(\mathbb T^{d-k}\times\mathbb T^{k})}$, but this distinction is inconsequential since $s=\sigma+\varepsilon$ and both $\varepsilon$ and $\delta$ can be arbitrarily small) small enough. Similarly, we can prove that $\Phi$ is a contraction on $X_{T}$ and therefore admits a fixed point in $X_{T}$. Uniqueness can be established using a similar estimate.

We establish \eqref{TIN72} by combining the bounds obtained in
\eqref{WL5} and \eqref{WL6}. These estimates imply that the nonlinear
operator $\Phi$ obeys the inequality
\begin{equation}\label{Phi-bound}
    \|\Phi(\gamma,\rho)\|_{X_T}
    \le
    \bigl(1 + C_{\sigma,\varepsilon} T^{1/p}\bigr)
    \Bigl(
    \|\gamma_0\|_{\Sp^{\alpha',\,s}}
    +
    C_{s,\delta} T^{1/p'}
    \|\omega\|_{B^{s+\delta}_{(r',\tilde r'),\infty}(\mathbb T^{d-k}\times\mathbb T^{k})}
    \|(\gamma,\rho)\|_{X_T}^2
    \Bigr).
\end{equation}

		Guided by this inequality, we fix a threshold
\[
R_T = R_T\!\left(\|\omega\|_{B^{s}_{(r',\tilde r'),\infty}(\mathbb T^{d-k}\times\mathbb T^{k})},\, s\right)
\]
small enough. More precisely, the choice of $R_T$ depends on
$\|\omega\|_{B^{s+\delta}_{(r',\tilde r'),\infty}(\mathbb T^{d-k}\times\mathbb T^{k})}$
rather than on
$\|\omega\|_{B^{s}_{(r',\tilde r'),\infty}(\mathbb T^{d-k}\times\mathbb T^{k})}$,
which does not affect the argument. This allows us to select a constant
$M>0$ such that, for every $y\in[0,M]$, the inequality
\begin{equation}\label{M-condition}
    \bigl(1 + C_{\sigma,\varepsilon} T^{1/p}\bigr)
    \Bigl(
    \|\gamma_0\|_{\Sp^{\alpha',\,s}}
    +
    C_{s,\delta} T^{1/p'}
    \|\omega\|_{B^{s+\delta}_{(r',\tilde r'),\infty}(\mathbb T^{d-k}\times\mathbb T^{k})}
    \, y^2
    \Bigr)
    \le M
\end{equation}
is satisfied whenever
\[
\|\gamma_0\|_{\Sp^{\alpha',\,s}} \le R_T.
\]

We now introduce the space $X_{T,M}$ by
\[
X_{T,M}
:=
\bigl\{ (\gamma,\rho)\in X_T : \|(\gamma,\rho)\|_{X_T} \le M \bigr\}.
\]
Estimates \eqref{Phi-bound} and \eqref{M-condition} ensure that
\[
\Phi(X_{T,M}) \subset X_{T,M}.
\]

By possibly reducing $R_T$ further, one verifies through analogous
estimates that $\Phi$ is in fact a contraction on $X_{T,M}$. An
application of the Banach fixed point theorem therefore yields a
solution
\[
\gamma \in C\!\left([0,T],\,\Sp^{\alpha',\,s}\right)
\]
to \eqref{Hartree} on $[0,T]\times\mathbb T^{d-k}\times\mathbb T^{k}$, whose
associated density $\rho_\gamma$ satisfies
\[
\rho_\gamma \in
L_t^q L_x^r L_y^{\tilde r}\!\left([0,T]\times\mathbb T^{d-k}\times\mathbb T^k\right).
\]
	\end{proof}

	\section*{Acknowledgments} 
	The second author wishes to thank the National Board for Higher Mathematics (NBHM) for
	the research fellowship and Indian Institute of Science Education and Research, Pune for
	the support provided during the period of this work.
	The third author is supported by the DST-INSPIRE Faculty Fellowship DST/INSPIRE/04/2023/002038.

	\bibliographystyle{siam}
	\bibliography{bib.bib}
\end{document}